\documentclass[12pt,a4paper]{article}

\usepackage{graphicx}
\usepackage{subfig}
\usepackage{amsmath,amsthm,amssymb}
\usepackage{esint}
\usepackage{mathrsfs}
\usepackage{appendix}

\usepackage{hyperref}

\usepackage[left=2.2cm,right=2.2cm,top=3cm,bottom=3.5cm]{geometry}

\usepackage{color}
\usepackage{array}
\usepackage{url}

\usepackage{float}

\newtheorem{theorem}{Theorem}[section]
\newtheorem*{theorem*}{Theorem}
\newtheorem{lemma}[theorem]{Lemma}
\newtheorem{corollary}[theorem]{Corollary}
\newtheorem{proposition}[theorem]{Proposition}

\theoremstyle{definition}
\newtheorem{definition}[theorem]{Definition}
\newtheorem{conjecture}[theorem]{Conjecture}
\newtheorem{example}[theorem]{Example}

\theoremstyle{remark}
\newtheorem{remark}[theorem]{Remark}
\usepackage{authblk}

\usepackage{mathtools}
\usepackage{amssymb}
\usepackage{appendix}

\numberwithin{equation}{section}

\usepackage[numbers,sort&compress]{natbib}

\DeclareMathOperator{\supp}{supp}

\def \del {\partial}

\allowdisplaybreaks

\usepackage{scalerel}
\usepackage{stackengine,wasysym}

\newcommand\reallywidetilde[1]{\ThisStyle{%
  \setbox0=\hbox{$\SavedStyle#1$}%
  \stackengine{-.1\LMpt}{$\SavedStyle#1$}{%
    \stretchto{\scaleto{\SavedStyle\mkern.2mu\AC}{.5150\wd0}}{.6\ht0}%
  }{O}{c}{F}{T}{S}%
}}

\begin{document}
	
\title{H\"older regularity of the pressure for weak solutions of the 3D Euler equations in bounded domains}
	
\author[]{Claude Bardos\footnote{Laboratoire J.-L. Lions, BP187, 75252 Paris Cedex 05, France; Wolfgang Pauli Institute - Inst. CNRS Pauli IRL 2842 c/o Faculty of Mathematics, University of Vienna, Oskar-Morgenstern-Platz 1, A-1090 Wien, Austria. Email: \textsf{claude.bardos@gmail.com}} \,, Daniel W. Boutros\footnote{Department of Applied Mathematics and Theoretical Physics, University of Cambridge, Cambridge CB3 0WA UK. Email: \textsf{dwb42@cam.ac.uk}} \space and Edriss S. Titi\footnote{Department of Mathematics, Texas A\&M University, College Station, TX 77843-3368, USA; Department of Applied Mathematics and Theoretical Physics, University of Cambridge, Cambridge CB3 0WA UK; also Department of Computer Science and Applied Mathematics, Weizmann Institute of Science, Rehovot 76100, Israel. Emails: \textsf{titi@math.tamu.edu} \; \textsf{Edriss.Titi@maths.cam.ac.uk}}}

\date{October 16, 2024} 
	
\maketitle

\begin{abstract}
We consider the three-dimensional incompressible Euler equations on a bounded domain $\Omega$ with $C^4$ boundary. We prove that if the velocity field $u \in C^{0,\alpha} (\Omega)$ with $\alpha > 0$ (where we are omitting the time dependence), it follows that the corresponding pressure $p$ of a weak solution to the Euler equations belongs to the Hölder space $C^{0, \alpha} (\Omega)$. We also prove that away from the boundary $p$ has $C^{0,2\alpha}$ regularity. \\
In order to prove these results we use a local parametrisation of the boundary and a very weak formulation of the boundary condition for the pressure of the weak solution, as was introduced in [C. Bardos and E.S. Titi, Philos. Trans. Royal Soc. A, 380 (2022), 20210073], which is different than the commonly used boundary condition for classical solutions of the Euler equations. Moreover, we provide an explicit example illustrating the necessity of this new very weak formulation of the boundary condition for the pressure. Furthermore, we also provide a rigorous derivation of this new formulation of the boundary condition for weak solutions of the Euler equations. \\
This result is of importance for the proof of the first half of the Onsager Conjecture, the sufficient conditions for energy conservation of weak solutions to the three-dimensional incompressible Euler equations in bounded domains. In particular, the results in this paper remove the need for separate regularity assumptions on the pressure in the proof of the Onsager conjecture. 
\end{abstract}


\vspace{0.1cm} \noindent \textbf{Mathematics Subject Classification:} 35Q31 (primary), 35Q35, 76B03, 35J05, 35J08, 35J25, 35D30 (secondary)
	
\bigskip

\tableofcontents

\section{Introduction}
\subsection{Background and introduction of the problem}

In this work, we study the existence and regularity of the pressure for weak solutions of the three-dimensional Euler equations of an ideal incompressible fluid in a $C^4$ bounded domain $\Omega \subset \mathbb{R}^3$. These equations are given by
\begin{equation} \label{eulerequations}
\partial_t u + \nabla \cdot (u \otimes u) + \nabla p = 0, \quad \nabla \cdot u = 0,
\end{equation}
where $u : \Omega \times (0,T) \rightarrow \mathbb{R}^3$ is the velocity field and $p : \Omega \times (0,T) \rightarrow \mathbb{R}$ is the pressure, which are the unknowns. In addition, we assume the following impermeability boundary condition
\begin{equation} \label{velocitybc}
(u \cdot n) \lvert_{\partial \Omega} = 0,
\end{equation}
where $n$ is the outward normal vector to the boundary.

Since the work of Lars Onsager in \cite{onsager}, which emphasised the role of weak solutions of the Euler equations in the theory of turbulence, such solutions have received a lot of attention in the mathematical literature. To be precise, Onsager conjectured that there is a relation between the H\"older (or Besov) regularity of the velocity vector field of a weak solution of the Euler equations on the one hand and the (potential) loss of energy via anomalous dissipation on the other hand. 

This relation is referred to as Onsager's conjecture, which can be stated as follows (a survey of the literature will be given later in the introduction).
\begin{conjecture}[Onsager's conjecture]
Let $u \in L^\infty ((0,T); L^2 (\Omega))$ be a weak solution of the incompressible Euler equations. Then:
\begin{enumerate}
    \item If $u \in L^3 ((0,T); C^{0,\alpha} (\Omega))$ with $\alpha > \frac{1}{3}$, then the kinetic energy of the solution $u$ is conserved, i.e. the spatial $L^2 (\Omega)$ norm is constant in time.
    \item There exist solutions $u \in L^3 ((0,T); C^{0,\alpha} (\Omega))$ with $\alpha < \frac{1}{3}$, such that the kinetic energy is decreasing.
\end{enumerate}
\end{conjecture}

Until relatively recently, most of the work on this conjecture has focused on the problem in the absence of physical boundaries. However, it has long been known that boundary effects play an important role in the understanding of hydrodynamic turbulence (see for example \cite{bardos2021,frisch} and references therein). 

One result that remarkably illustrates the role of boundary effects is the Kato criterion, which was established in \cite{kato} (see also \cite{bardoseuler,bardosturbulence}). The criterion states that for solutions of the Navier-Stokes equations with the no-slip boundary condition, the vanishing of the total energy dissipation integrated over a thin layer (the so-called Prandtl-Von Karman boundary layer) as the viscosity $\nu$ goes to zero, is equivalent to the validity of the inviscid limit of the Navier-Stokes solutions to the corresponding solution of the Euler equations (under the assumption that a strong Euler solution exists). Note that the thickness of the Prandtl-Von Karman boundary layer is of order $\nu$. Additional Kato-type criteria were proved in \cite{wang,constantinremarks,chenkato}. These results suggest that in the case of bounded domains there is an equivalence between the presence of anomalous dissipation and the loss of regularity, as previously noted in \cite{titi2019}. 

Because of the aforementioned crucial role of physical boundaries in the study of turbulent flows, it is therefore worthwhile to try to understand the interaction of weak solutions with the boundary, and in particular the pressure. The goal of this paper is to prove that in a $C^4$ bounded domain $\Omega \subset \mathbb{R}^3$ for a velocity field $u$ of a weak solution of the Euler equations that belongs to $C^{0,\alpha} (\Omega)$ for $\alpha \in (0,1)$, the corresponding pressure of the weak solution satisfies $p \in C^{0,\alpha} (\Omega)$ with the same exponent $\alpha$. Such a regularity and existence result is needed for the completion of the proof of the first part of the Onsager conjecture in the case of bounded domains \cite{titi2018} (this will be explained in more detail later).

In standard treatments of mathematical fluid mechanics, such as \cite{robinsonbook,constantinbook,temambook}, the pressure is removed from the equations by using the Leray-Helmholtz decomposition. The pressure can then be recovered from the velocity by solving the elliptic equation
\begin{equation} \label{pressureeq}
-\Delta p = (\nabla \otimes \nabla) : (u \otimes u), \quad \text{in } \Omega.
\end{equation}
In our setting however, one also needs to derive an associated boundary condition for the pressure for the elliptic equation \eqref{pressureeq} in order to have a complete boundary-value problem. A major part of this paper is the derivation and justification of a new boundary condition for the pressure, which is different than the commonly used boundary condition in the setting of classical solutions, and which is crucial to have a well-defined elliptic problem in the low-regularity setting (as was explained previously, the study of weak solutions is fundamentally related to the study of turbulence). 

In the remainder of this introduction, we will provide a formal derivation of this new boundary condition and an overview of the literature. After that we will state our main results and several corollaries. 
\subsection{Introduction of the new boundary condition for the pressure} \label{introdboundcondition}
We will start by providing a formal derivation of the new boundary condition that was already mentioned. If equation \eqref{eulerequations} is considered in the presence of boundaries, formally one can take the normal component of the Euler equations in order to find
\begin{equation*}
\partial_t (u \cdot n) + \big[ \nabla \cdot (u \otimes u) \big] \cdot n + \partial_n p = 0.
\end{equation*}
Note that in this equation and in what follows we will use the extension of the outward normal vector to the interior of the domain $\Omega$. For such an extension to be performed in a unique fashion, one has to perform the computations sufficiently close to the boundary, which will always be the case in this paper. Following \cite{bardos2021}, for a $C^2$ bounded domain the advective term can be written formally as follows (where we are using the Einstein summation convention)
\begin{align*}
\big[ \nabla \cdot (u \otimes u) \big] \cdot n &= (u_i \partial_i u_j) n_j = - \big( u \otimes u \big) : \nabla n + u \cdot \nabla (u \cdot n). 
\end{align*}
We now calculate the second term in the above expression. To fix ideas, we do the calculation and highlight the issue in three dimensions (but it works the same way for any dimension larger than or equal to 2), see also \cite{bardos2021}. It should be stressed that these computations are done in the context of smooth, say classical, solutions. We will formally derive the boundary condition that will be rigorously justified in this paper. In what follows $\tau_1$ and $\tau_2$ will denote the orthogonal tangent vectors to the boundary at a given point $x \in \partial \Omega$, we compute that
\begin{align*}
&u \cdot \nabla (u \cdot n) = (u \cdot \tau_1) \partial_{\tau_1} (u \cdot n) + (u \cdot \tau_2) \partial_{\tau_2} (u \cdot n) + (u \cdot n) \partial_n (u \cdot n) \\
&= \frac{1}{2} \partial_n (u \cdot n)^2 + \partial_{\tau_1} \big( (u \cdot \tau_1) (u \cdot n) \big) + \partial_{\tau_2} \big( (u \cdot \tau_2) (u \cdot n) \big) - (u \cdot n) \partial_{\tau_1} (u \cdot \tau_1) - (u \cdot n) \partial_{\tau_2} (u \cdot \tau_2)  \\
&= \partial_n (u \cdot n)^2 + \partial_{\tau_1} \big( (u \cdot \tau_1) (u \cdot n) \big) + \partial_{\tau_2} \big( (u \cdot \tau_2) (u \cdot n) \big),
\end{align*}
where for the last equality we have used the incompressibility of the velocity field, i.e. that $\partial_{\tau_1} (u \cdot \tau_1) + \partial_{\tau_2} (u \cdot \tau_2) = -  \partial_n (u\cdot n)$. Therefore we can write the normal component of the Euler equations at the boundary as follows
\begin{equation*}
\partial_n p = \big( u \otimes u \big) : \nabla n - \partial_n (u \cdot n)^2 - \partial_{\tau_1} \big( (u \cdot \tau_1) (u \cdot n) \big) - \partial_{\tau_2} \big( (u \cdot \tau_2) (u \cdot n) \big) - \partial_t (u \cdot n).
\end{equation*}
Since $\partial_t (u \cdot n) \lvert_{\partial \Omega} = 0$ and $\partial_{\tau_1} \big( (u \cdot \tau_1) (u \cdot n) \big) \lvert_{\partial \Omega} = \partial_{\tau_2} \big( (u \cdot \tau_2) (u \cdot n) \big) \lvert_{\partial \Omega} = 0$, we conclude that the boundary condition associated with equation \eqref{pressureeq} is
\begin{equation} \label{boundarycondition}
\partial_n \big( p + (u \cdot n)^2 \big) = \big( u \otimes u \big) : \nabla n \quad \text{on } \partial \Omega.
\end{equation}

We will study the elliptic boundary-value problem given by equation \eqref{pressureeq} and Neumann boundary condition \eqref{boundarycondition}. Let us now introduce some terminology.
\begin{definition} \label{veryweakboundconddef}
Consider a velocity field $u \in C^{0,\alpha} (\Omega)$ and a pressure $p \in C^{0,\alpha} (\Omega)$ such that for the region within distance $\delta$ of the boundary $\partial \Omega$ (for some fixed $\delta > 0$), $p + (u \cdot n)^2$ has $C^{1,\alpha}$ regularity in the normal coordinate variable and $H^{-2}$ regularity in the tangential coordinate variables. We say that $(u,p)$ satisfies the \emph{very weak boundary condition} if 
\begin{equation} \label{veryweakboundcondeq}
\partial_n \big( p + (u \cdot n)^2 \big) = \big( u \otimes u \big) : \nabla n, \quad \text{on } \partial \Omega,
\end{equation}
holds in $H^{-2} (\partial \Omega)$. A pair $(u,p) \in C^{0,\alpha} (\Omega)$ (which is a weak solution of the Euler equations) is said to satisfy the usual (common) weak boundary condition if
\begin{equation} \label{usualboundarycond}
\partial_n p = \big( u \otimes u \big) : \nabla n, \quad \text{on } \partial \Omega.
\end{equation}
\end{definition}
\begin{remark}
We remark that in the case $(u,p) \in C^{0,\alpha} (\Omega)$, the regularity does not suffice to apply a trace theorem in order to give meaning to either $\partial_n \big( p + (u \cdot n)^2 \big) \lvert_{\partial \Omega}$ or $\partial_n p \, \lvert_{\partial \Omega}$, i.e. either the very weak or usual boundary conditions. Therefore we require additional Hölder regularity in the normal coordinate variable near the boundary (at the expense of negative Sobolev regularity in the tangential coordinate variables), so that $\partial_n \big( p + (u \cdot n)^2 \big)$ is well-defined as a function near the boundary. In this work, the required regularity will be obtained directly from the equation.
\end{remark}
The first thing to observe is that the two notions of boundary condition are equivalent if $\alpha > \frac{1}{2}$, as we will show in the next lemma.
\begin{lemma} \label{equivalencelemma}
Let $u \in C^{0,\alpha} (\Omega)$ with $\alpha > \frac{1}{2}$ be a velocity field such that $(u \cdot n) \lvert_{\partial \Omega} = 0$, then
\begin{equation}
\partial_n (u \cdot n)^2 = 0,
\end{equation}
and consequently boundary conditions \eqref{veryweakboundcondeq} and \eqref{usualboundarycond} are equivalent.
\end{lemma}
\begin{proof}
One can compute for $x \in \partial \Omega$ and $0 < s \ll 1$ (recall that in this paper we will use the convention that the normal vector points outward from the domain)
\begin{align*}
\lvert \partial_n (u \cdot n)^2 (x) \rvert &= \bigg\lvert \lim_{s \rightarrow 0} \frac{(u \cdot n)^2 (x - s n) - (u \cdot n)^2 (x)}{-s} \bigg\rvert = \bigg\lvert \lim_{s \rightarrow 0} \frac{(u \cdot n)^2 (x - s n) }{-s} \bigg\rvert \\
&= \bigg\lvert \lim_{s \rightarrow 0} \frac{\big( (u \cdot n) (x - s n) - (u \cdot n) (x) \big)^2}{-s} \bigg\rvert \leq \lVert u \rVert_{C^{0,\alpha} (\Omega)}^2 \lim_{s \rightarrow 0} \frac{s^{2 \alpha}}{s} = 0,
\end{align*}
where we have used the boundary condition \eqref{velocitybc} for $u$.
\end{proof}

In this paper, we are predominantly interested in the low regularity setting, i.e., for $\alpha \in (0, \frac{1}{2}]$ (although the regularity result we will prove holds for all $\alpha \in (0,1)$). For $\alpha \in (0, \frac{1}{2}]$ the property $\partial_n (u \cdot n)^2 = 0$ does not generally hold, as we will show later in section \ref{examplesection}. In fact, in section \ref{examplesection} we will construct an explicit example of a Hölder continuous incompressible velocity field $u \in C^{0,\alpha} (\Omega)$ satisfying the boundary condition $(u \cdot n) \lvert_{\partial \Omega} = 0$ for which $\partial_n (u \cdot n)^2 \lvert_{\partial \Omega} \notin \mathcal{D}' (\partial \Omega)$. 

The regularity result for the pressure that we will prove is that if $u \in C^{0,\alpha} (\Omega)$, then $p \in C^{0,\alpha} (\Omega)$ for any $\alpha \in (0,1)$ such that the pressure obeys the very weak boundary condition \eqref{veryweakboundcondeq}. Note that we will omit the time dependence of the pressure and the velocity field throughout this paper, as it does not play a role of significance here.

The reason we prove this regularity result in the setting of H\"older spaces is because these spaces play an essential role in the theory of turbulence, as they can be related to the scaling properties of the structure functions. 
\subsection{Overview of the literature}
Before we state the main results of this paper, we provide a brief overview of the literature. As Onsager's conjecture has been the focus of many works recently, by necessity this review will be an incomplete overview of all the relevant results. In particular, we will mostly focus on results which pertain to the case of bounded domains.

Onsager's conjecture was first proved in the torus $\mathbb{T}^3$ or on the full space $\mathbb{R}^3$ in a series of works \cite{eyink,constantin,duchon,cheskidovconstantin,isettproof,buckmasteradmissible}. As noted in \cite{drivas2018}, it is natural to consider Onsager's conjecture in the presence of physical boundaries, as all experimental evidence for anomalous dissipation from laboratory experiments is for flows with physical boundaries. The conservation of energy part of Onsager's conjecture was proven on bounded domains with $C^2$ boundary in \cite{titi2018}, after results for the half plane in \cite{robinson}. Then in \cite{titi2019} the first half of the conjecture was proven under only an interior H\"older regularity assumption above the Onsager exponent on the velocity field, but with a continuity assumption on the normal component of the energy flux near the boundary, as well as a regularity assumption on the pressure.

It was shown in \cite{drivas2018} using a different proof that in the interior of the domain a Besov regularity rather than Hölder regularity assumption suffices. A technical improvement of this result was given in \cite{nguyen}. Moreover, in \cite{drivas2019,constantinremarks,bardosturbulence,bardoseuler} sufficient conditions were provided for convergence in bounded domains of Leray-Hopf weak solutions of the Navier-Stokes equations to weak (potentially dissipative) solutions of the Euler equations under structure function scaling assumptions. Further work on the mathematical analysis of wall-bounded turbulence can be found in \cite{quan,eyinkonsager} (and see references therein). 

In order to construct dissipative solutions of the Euler equations (with periodic boundary conditions), the method of convex integration was introduced and employed to incompressible fluid mechanics in \cite{lellisadmissibility,lellisinclusion}. Since then, the techniques and implementation of convex integration have been steadily improved to construct H\"older continuous solutions of the Euler equations in a sequence of papers \cite{lelliscontinuous,buckmaster,buckmasteradmissible,isettproof,daneri} (and see references therein). Reviews of these techniques can be found in \cite{buckmasterreview,lellis}.

Although the majority of results established using convex integration are for the case of periodic boundary conditions, there are also several results in the presence of physical boundaries. In \cite{bardosnonuniqueness} the nonuniqueness of admissible weak solutions on an annulus with rotational data was established. In \cite{vasseurboundary} the nonuniqueness of weak solutions for plug flow initial data in the channel was proved, and it was demonstrated that the deviation from the steady solution coincides with upper bounds on the layer separation. Finally, the surveys \cite{bardosturbulence,bardoseuler,eyinkreview} provide an overview of results on the role of boundaries in incompressible turbulence. 

For the proof in \cite{titi2018} of the conservation part of the Onsager conjecture with boundaries, the propagation of $C^{0, \alpha}$ regularity from the velocity to the pressure was a necessary additional assumption, see Proposition 1.2 in \cite{titi2018}. The purpose of this paper is to give a full proof of this statement in the three-dimensional case.

This type of pressure regularity problem for the Euler equations, to the knowledge of the authors, was first considered in \cite{silvestre,constantinlocal}, in the case without physical boundaries. It was proven in these papers that if $u \in C^{0, \alpha} (\mathbb{R}^3)$, then $p \in C^{0, 2\alpha} (\mathbb{R}^3)$ for $\alpha \in (0,1/2) \cup (1/2,1)$. In particular, if $\alpha > \frac{1}{2}$, this means that $p \in C^{1,2\alpha-1} (\mathbb{R}^3)$. These results were then extended to Besov spaces in \cite{colombo} (see also \cite{colomboregularity}). The (usual/common) weak boundary condition \eqref{usualboundarycond} for the pressure was first derived in \cite{temameuler}, for strong solutions of the Euler equations. 

In the two-dimensional setting with boundaries, the pressure regularity problem was addressed in \cite{bardos2021}. In particular, it was shown that the velocity field and the pressure have the same H\"older regularity for a bounded domain $\Omega$ in two dimensions. Moreover, the very weak formulation of the Neumann boundary condition for the pressure given in equation \eqref{boundarycondition} was introduced in \cite{bardos2021}. This was an essential part of the proof, as it allows to construct a trace formula which establishes that the normal derivative of $p + \phi (u\cdot n)^2$ is continuous in the $H^{-2} (\partial \Omega)$ norm near the boundary, where we have introduced a smooth cutoff function $\phi$ which is $1$ near the boundary and vanishes in the interior of the domain. This trace formula is then applied in the elliptic estimates for equation \eqref{pressureeq} in order to establish the $C^{0,\alpha} (\Omega)$ regularity of the pressure.

One goal of this paper is to extend the approach in \cite{bardos2021} to three dimensions. In order to go from two to three dimensions, a change in approach of the proof was necessary. In \cite{bardos2021} the proof relies on a global parametrisation of the boundary in two dimensions, which facilitates the definition of a global localisation in a small vicinity of the boundary. This global localisation in turn allows for the decomposition of the velocity field near the boundary and away from the boundary. In this contribution we do not rely on the global parametrisation of the boundary and instead we modify this localisation, namely we introduce a partition of unity of the region near the boundary itself. The reason for doing so is that in two dimensions, the boundary can be parametrised globally, but in three dimensions this is not possible. 

That means that the near-boundary analysis has to be done in a local coordinate system and then extended globally. This in turn required the use of new elliptic Schauder-type estimates, which we establish in Appendices \ref{schauderappendix1} and \ref{schauderappendix2}. We expect that our proof is quite robust, i.e. it can be extended to higher dimensions without much effort and for other hydrodynamical systems for which the pressure satisfies a similar elliptic boundary-value problem. 

The other aim of this paper is to demonstrate that \eqref{veryweakboundcondeq} is the only viable boundary condition for the pressure for weak solutions of the Euler equations at this low level of regularity of the velocity, namely $u \in C^{0,\alpha} (\Omega)$ for $\alpha \in (0,1)$, in particular when $\alpha \in (0, \frac{1}{2}]$. Moreover, we also obtain the interior double Hölder exponent regularity of the pressure, i.e. $p \in C^{0,2 \alpha}_{\mathrm{loc}} (\Omega)$ for $\alpha \in (0, \frac{1}{2})$ (and also $\alpha \in (\frac{1}{2},1)$).

While an earlier draft of this work was completed, the paper \cite{derosa} came to our attention. In that paper the authors also prove a regularity result for the pressure, but with a different boundary condition. Subsequently, the double Hölder exponent regularity of the pressure (satisfying that boundary condition) was obtained in \cite{derosa2024}. In particular, in \cite{derosa,derosa2024} the authors use the usual weak boundary condition $\partial_n p = (u \otimes u) : \nabla n$ (cf. \eqref{usualboundarycond}) as opposed to the very weak boundary condition \eqref{boundarycondition}, as in the abstract of \cite{derosa} it is noted that: ``Differently from Bardos and Titi (2022), we do not introduce a new boundary condition for the pressure, but instead work with the natural one [i.e. $\partial_n p = (u \otimes u) : \nabla n$].''. These two formulations of the boundary condition are equivalent say if $u \cdot \nabla (u \cdot n) \lvert_{\partial \Omega} = 0$. 

As we stressed above, this is true for classical solutions as well as for the case when $u\in C^{0,\alpha} (\Omega)$ with $\alpha > 1/2$ (as we proved in Lemma \ref{equivalencelemma}), but in the case when $u \in C^{0, \alpha} (\Omega)$ for $0 < \alpha \leq \frac{1}{2}$ we will show in section \ref{examplesection} that $u \cdot \nabla (u \cdot n) \lvert_{\partial \Omega}$ in some cases does not make sense, not even as an element of $\mathcal{D}' (\partial \Omega)$. In particular, one cannot take $u \cdot \nabla (u \cdot n) \lvert_{\partial \Omega}$ to be zero. This point has been overlooked previously in the literature, and it is a highlight of this paper. For this reason, we prove the pressure regularity result with the weaker formulation of the boundary condition for the pressure \eqref{boundarycondition}, which holds for all H\"older continuous velocity fields (and hence all Hölder continuous weak solutions of the Euler equations). In particular, we rigorously derive the very weak boundary condition \eqref{veryweakboundcondeq} directly from the weak formulation of the Euler equations in Theorem \ref{veryweakboundcondthm}. Moreover, our regularity proof is more explicit because it relies on localisation arguments. 

\subsection{Main results of the paper and overview}
We will first give a rough descriptive version of the result that we will prove (the precise version is stated in section \ref{parametrisationsection} as Theorem \ref{pressureregularity}). Let us introduce $\widetilde{\phi} : [0, \infty) \rightarrow [0,1]$ to be a nonincreasing smooth $C^\infty$ function defined as follows (for some fixed $\delta > 0$)
\begin{equation} \label{phifunction}
\widetilde{\phi} (s) \coloneqq \begin{cases}
1, \quad \text{if } s \in [0,\delta], \\
0, \quad \text{if } s \geq 2 \delta.
\end{cases}
\end{equation}
Then we define
\begin{equation}
\phi (x) \coloneqq \widetilde{\phi} (d(x, \partial \Omega)),
\end{equation}
where $d(x, \partial \Omega)$ is the distance of the point $x$ to the boundary.
Throughout this paper, we will assume that $\delta > 0$ is chosen suitably small such that the outward normal vector to $\partial \Omega$ can be smoothly extended adequately (in a sense to be described later) in the set $\{ x \in \Omega \lvert \; d (x, \partial \Omega) \leq 2 \delta \}$. Now we will state the informal version of the main result of this paper.
\begin{theorem} \label{regularitytheorem}
Let $u \in C^{0, \alpha} (\Omega)$, for $\alpha \in (0,1)$, be a divergence-free velocity field in an open set $\Omega \subset \mathbb{R}^3$ with a $C^4$ boundary. Moreover we assume that $u$ satisfies the boundary condition $(u \cdot n ) \lvert_{\partial \Omega} = 0$. By introducing a new weak formulation of the boundary condition \eqref{boundarycondition}, we therefore consider the following elliptic boundary-value problem for the pressure (subject to a Neumann boundary condition)
\begin{align}
\begin{cases}
-\Delta \big( p + \phi (x) (u \cdot n)^2 \big) = (\nabla \otimes \nabla) : (u \otimes u) - \Delta ( \phi (x) (u \cdot n)^2) \quad \text{in } \Omega, \\
\partial_n \big( p + (u \cdot n)^2 \big) = \big( u \otimes u \big) : \nabla n \quad \text{on } \partial \Omega.
\end{cases} \label{neumannpressureproblem}
\end{align}
Then there exists a unique solution $p \in C^{0,\alpha} (\Omega)$ to this boundary-value problem, and it satisfies the following estimate
\begin{equation} \label{regularityestimate}
\lVert p \rVert_{C^{0, \alpha} (\Omega)} \leq C \lVert u \otimes u \rVert_{C^{0, \alpha} (\Omega)}.
\end{equation}
Moreover, in particular it holds that $\partial_n \big(p + (u \cdot n)^2 \big) \lvert_{\partial \Omega} \in H^{-2} (\partial \Omega)$.
\end{theorem}
It should be emphasised that the vector field $u$ in Theorem \ref{regularitytheorem} is not required to be a weak solution of the Euler equations. First we make several remarks on the result. 
\begin{remark}
In fact, we not only prove that $\partial_n \big(p + (u \cdot n)^2 \big) \lvert_{\partial \Omega}$ lies in $H^{-2} (\partial \Omega)$. Let $\delta > 0$ be suitably small and $V_\delta$ be the region of $\Omega$ within distance $\delta$ of the boundary $\partial \Omega$ (see equation \eqref{Vdeltadef} for a definition). We will establish in sections \ref{tracesection} and \ref{limitsection} that $\partial_n \big(p + \phi (x) (u \cdot n)^2 \big)$ is Hölder continuous with exponent $\alpha$ in the normal coordinate variable, and has $H^{-2}$ regularity in the tangential coordinate variables in the region $V_\delta$ (see Lemma \ref{tracelemma} and the proof of Theorem \ref{pressureregularity} in section \ref{limitsection}). To the best of our knowledge, such a result and interpretation of the boundary condition is new in the context of solutions of elliptic equations such as \eqref{neumannpressureproblem}.
\end{remark}
\begin{remark}
As was mentioned before, the result of Theorem \ref{regularitytheorem} was needed for the completion of the proof of the first part of the Onsager conjecture with physical boundaries in \cite[~Proposition 1.2]{titi2018}. In particular, the result removes the necessity for separate existence and regularity assumptions on the pressure when formulating sufficient criteria for energy conservation in the presence of physical boundaries (we will state a precise result below in Theorem \ref{onsagerconjecture}). Separate assumptions on the existence and regularity of the pressure in the context of Onsager-type results were also made in \cite[~Theorem 4.1]{titi2019} and \cite[~Theorem 1 and Remark 3]{drivas2018}. Moreover, the results of this paper clarify the precise boundary-value problem satisfied by the pressure of weak solutions of the 3D Euler equations.
\end{remark}
\begin{remark}
We emphasise that the result of Theorem \ref{regularitytheorem} holds for any Hölder exponent $0 < \alpha < 1$, and does not require an assumption of the type $\alpha > \frac{1}{3}$, as in the case of the Onsager conjecture. In particular, the pressure regularity result also holds for dissipative weak solutions of the Euler equations in bounded domains with Hölder regularity below the Onsager exponent (see \cite{bardoseuler,bardosturbulence} for results on this notion of solution). 
\end{remark}
\begin{remark}
To the best of our knowledge, an estimate of the type \eqref{regularityestimate} is not covered by the existing Schauder theory for elliptic equations. Therefore, in the proof of Theorem \ref{regularitytheorem} we need to extend the Hölder regularity theory for elliptic equations and use an overlapping compact localisation near and away from the boundary (together with a finite covering of the boundary). In particular, we prove new Schauder-type estimates in Appendices \ref{schauderappendix1} and \ref{schauderappendix2}. The idea of using the overlapping localisation has been inspired by \cite{kukavicainviscid}.
\end{remark}
Once we have proved the existence and $C^{0,\alpha} (\Omega)$ Hölder regularity of the pressure in $\Omega$, we then show that away from the boundary in fact it belongs to the space $C^{0,2 \alpha}_{\mathrm{loc}} (\Omega)$ (analogously to the results in \cite{silvestre}). We define the set (cf. equation \eqref{Vdeltadef})
\begin{equation*}
V_{\eta} \coloneqq \{ x \in \Omega \lvert \; d(x, \partial \Omega) < \eta \}.
\end{equation*}
\begin{theorem} \label{doubleregularitythm}
Let $0 < \alpha < \frac{1}{2}$, let $\partial \Omega \in C^4$, and let $u \in C^{0,\alpha} (\Omega)$ be a divergence-free velocity field which is tangential to the boundary. Also, let $p \in C^{0,\alpha} (\Omega)$ be the pressure from Theorem \ref{regularitytheorem}. Then $p \in C^{0,2 \alpha} (\Omega \backslash V_{\eta} )$ for any $\eta > 0$ and it satisfies the estimate
\begin{equation}
\lVert p \rVert_{C^{0,2 \alpha} (\Omega \backslash V_{\eta} )} \leq C \lVert u \rVert_{C^{0,\alpha} (\Omega)}^2,
\end{equation}
where the constant $C$ depends on $\eta$.
\end{theorem}
As has been stressed repeatedly up to this point, if the velocity field $u$ has low Hölder regularity (i.e. $\alpha < \frac{1}{2}$), it is essential to use the very weak boundary condition \eqref{veryweakboundcondeq}. In particular, we establish this by means of an explicit example. We have the following result (which should be contrasted with Lemma \ref{equivalencelemma}).
\begin{theorem} \label{exampletheorem}
Let $\alpha \in \left(0,\frac{1}{2}\right)$. There exists a divergence-free velocity field $u \in C^{0,\alpha} (\Omega)$ such that $(u \cdot n) \lvert_{\partial \Omega} = 0$ and
\begin{equation}
\partial_n (u \cdot n)^2 \big\lvert_{\partial \Omega} \notin \mathcal{D}' (\partial \Omega).
\end{equation}
In particular, this implies that the associated pressure $p$ solving problem \eqref{neumannpressureproblem} (of which the existence and Hölder regularity was established in Theorem \ref{regularitytheorem}) has the property
\begin{equation}
\partial_n p \space \big\lvert_{\partial \Omega} \notin \mathcal{D}' (\partial \Omega).
\end{equation}
In the endpoint case $\alpha = \frac{1}{2}$, there exists a divergence-free velocity field $u$ obeying the same boundary condition $(u \cdot n ) \lvert_{\partial \Omega} = 0$ such that
\begin{equation}
\partial_n (u \cdot n)^2 \big\lvert_{\partial \Omega} \neq 0.
\end{equation}
\end{theorem}
In addition, we also derive the very weak boundary condition directly from the weak formulation of the Euler equations. We obtain the following result.
\begin{proposition} \label{weaksolutionproposition}
Let $(u,p)$ be any weak solution of the Euler equations, such that $u \in L^\infty ((0,T); C^{0,\alpha} (\Omega))$. Then the pressure $p \in L^\infty ((0,T); C^{0,\alpha} (\Omega))$ satisfies the very weak boundary condition \eqref{veryweakboundcondeq}.
\end{proposition}
Finally, by using the results in this paper, we can prove the following sharper result regarding the Onsager conjecture with physical boundaries.
\begin{theorem} \label{onsagerconjecture}
Let $u \in L^\infty ((0,T); L^2 (\Omega))$ be a weak solution of the Euler equations with the following properties:
\begin{itemize}
    \item For every open subset $\widetilde{\Omega} \subset \subset \Omega$, there exists an exponent $\alpha (\widetilde{\Omega}) > \frac{1}{3}$ such that $u \in L^3 ((0,T); C^{0,\alpha (\widetilde{\Omega})} (\widetilde{\Omega}))$.
    \item There exists a $\widetilde{\lambda} > 0$ such that $u \in L^3 ((0,T); C^{0,\widetilde{\lambda}} (\Omega))$.
\end{itemize}
Then the solution $u$ conserves energy, i.e. for almost every $t_1, t_2 \in (0,T)$
\begin{equation}
\lVert u(\cdot, t_1) \rVert_{L^2} = \lVert u (\cdot, t_2) \rVert_{L^2}.
\end{equation}
\end{theorem}
\begin{proof}
For some suitably small $\gamma_0 > 0$, we define $\widetilde{\sigma} : V_{\gamma_0} \rightarrow \partial \Omega$ to be the projection onto the boundary, i.e. the map $x \in V_{\gamma_0} \mapsto \widetilde{\sigma} (x)$ such that $d(x, \partial \Omega) = \lvert x - \widetilde{\sigma} (x) \rvert$. By Theorem 4.1 in \cite{titi2019}, we know that the weak solution $u$ conserves energy under the first assumption of the theorem together with the following two additional assumptions:
\begin{enumerate}
    \item There exist $M, \gamma_0 > 0$ such that
    \begin{equation} \label{conservationassumption1}
    p \in L^{3/2} ((0,T); H^{-M} (V_{\gamma_0})).
    \end{equation}
    \item The following limit has to hold (for $\gamma_1 \in (0, \gamma_0)$)
    \begin{equation} \label{conservationassumption2}
    \lim_{\gamma_1 \rightarrow 0} \int_0^T \frac{1}{\gamma_1} \int_{\big\{ x \in \Omega \big\lvert \frac{\gamma_1}{4} < d(x, \partial \Omega) < \frac{\gamma_1}{2} \big\}} \bigg\lvert \bigg( \frac{\lvert u \rvert^2}{2} + p \bigg) u(t,x) \cdot n (\widetilde{\sigma}(x)) \bigg\rvert dx dt = 0.
    \end{equation}
\end{enumerate}
By applying Theorem \ref{regularitytheorem} and using the second assumption of the theorem we immediately find that $p \in L^{3/2} ((0,T); C^{0,\widetilde{\lambda}} (\Omega))$. This means that assumption \eqref{conservationassumption1} is satisfied. This also implies that $\bigg( \frac{\lvert u \rvert^2}{2} + p \bigg) u(t,x) \in L^1 ((0,T); C^{0,\widetilde{\lambda}} (\Omega))$ . We can then obtain the following estimate (using the boundary condition $(u \cdot n) \lvert_{\partial \Omega} = 0$ and the $C^{0,\widetilde{\lambda}} (\Omega)$ Hölder continuity of $u$ and $p$)
\begin{align*}
&\int_0^T \frac{1}{\gamma_1} \int_{\big\{ x \in \Omega \big\lvert \frac{\gamma_1}{4} < d(x, \partial \Omega) < \frac{\gamma_1}{2} \big\}} \bigg\lvert \bigg( \frac{\lvert u \rvert^2}{2} + p \bigg) u(x,t) \cdot n (\widetilde{\sigma}(x)) \bigg\rvert dx dt \\
&\leq \int_0^T \frac{1}{\gamma_1} \int_{\big\{ x \in \Omega \big\lvert d(x, \partial \Omega) < \frac{\gamma_1}{2} \big\}} \bigg\lvert \bigg[ \bigg( \frac{\lvert u \rvert^2}{2} + p \bigg) u(x,t) - \bigg( \frac{\lvert u \rvert^2}{2} + p \bigg) u(\widetilde{\sigma} (x),t) \bigg] \cdot n (\widetilde{\sigma}(x)) \bigg\rvert dx dt \\
&\lesssim \int_0^T \frac{1}{\gamma_1} \int_{\big\{ x \in \Omega \big\lvert d(x, \partial \Omega) < \frac{\gamma_1}{2} \big\}} \lvert d(x, \partial \Omega) \rvert^{\widetilde{\lambda}} dx dt \lesssim \int_0^T \frac{1}{\gamma_1} \int_{\big\{ x \in \Omega \big\lvert d(x, \partial \Omega) < \frac{\gamma_1}{2} \big\}} \gamma_1^{\widetilde{\lambda}} dx dt \xrightarrow[]{\gamma_1 \rightarrow 0} 0,
\end{align*}
where in the last step we have used $\lvert V_{\gamma_1} \rvert \lesssim \gamma_1$ (see \cite[~p. 199]{titi2018}). As condition \eqref{conservationassumption2} has now also been verified, we conclude that the solution $u$ conserves energy.
\end{proof}
\begin{remark}
We observe that in order to ensure energy conservation one only needs interior Hölder regularity with exponent above $\frac{1}{3}$, but not uniformly with the same exponent up to the boundary, together with Hölder regularity with any positive exponent uniformly up to the boundary. As was said before, part of the purpose of this work has been to remove the need for separate hypotheses on the pressure in results for energy conservation. Indeed, Theorem \ref{onsagerconjecture} does not require separate regularity assumptions on the pressure, unlike the results in the previous works \cite{titi2018,titi2019,drivas2018}. We note that in the 2D case the regularity hypotheses on the pressure have already been removed in \cite{bardos2021}.
\end{remark}

Now we will proceed to give an outline for the rest of this paper. In section \ref{parametrisationsection} we introduce a parametrisation of the boundary region. In particular we define the local coordinate system and state the differential operators in these coordinates. In section \ref{weakformulationsection} we rigorously derive the weak formulation of boundary-value problem \eqref{neumannpressureproblem} for weak solutions of the Euler equations. Moreover, we prove Proposition \ref{weaksolutionproposition} and show that the very weak boundary condition \eqref{veryweakboundcondeq} can be rigorously derived for all Hölder continuous weak solutions of the Euler equations.

Theorem \ref{regularitytheorem} will be proved in sections \ref{mollificationsection}-\ref{limitsection}. The proof then proceeds in the following steps:
\begin{itemize}
    \item We first mollify the velocity field, which is done in section \ref{mollificationsection}. This is not as straightforward as in in the case of the torus or the whole space, as the mollified velocity $u^\epsilon$ has to satisfy the boundary condition \eqref{velocitybc}. One needs to split the velocity field into interior and boundary parts. The parametrisation of the boundary region is then used to extend the velocity field over the boundary. 
    \item Then it is possible to use standard Schauder theory, as can be found in \cite{krylov,ladyzhenskaya,gilbarg,hanbook,giaquinta}. This will give us a candidate for a near the boundary truncated and mollified pressure $P^\epsilon$, for which we are going to take the limit $\epsilon \rightarrow 0$ at the end of the proof. 
    \item We then prove that the $C^{0,\alpha} (\Omega)$ norm of $P^\epsilon$ is bounded by the $C^{0,\alpha} (\Omega)$ norm of $u^\epsilon \otimes u^\epsilon$ uniformly in $\epsilon$. In section \ref{interiorsection} we derive the interior estimates, while in section \ref{decompositionsection} we obtain the boundary estimates. In order to deal with the boundary condition we establish a trace lemma in section \ref{tracesection}. In section \ref{limitsection} we then combine the estimates from sections \ref{interiorsection}-\ref{decompositionsection} and take the limit $\epsilon \rightarrow 0$ to establish the regularity estimate from Theorem \ref{regularitytheorem}.
\end{itemize}
Subsequently, in section \ref{doubleregularitysection} we will prove the interior double Hölder exponent regularity of the pressure and prove Theorem \ref{doubleregularitythm}. Finally, in section \ref{examplesection} we provide an explicit example that illustrates why the very weak boundary condition \eqref{boundarycondition} is necessary and prove Theorem \ref{exampletheorem}. To be more precise, we will construct an example of a velocity field in $C^{0,\alpha} (\Omega)$ for $0< \alpha < \frac{1}{2}$ with $(u \cdot n)\lvert_{\partial \Omega} = 0$ for which $\partial_n (u \cdot n)^2 \lvert_{\partial \Omega} \notin \mathcal{D}' (\partial \Omega)$. Therefore one cannot consider the terms $\partial_n p$ or $\partial_n (u \cdot n)^2$ individually at the boundary (as they are ill-defined). 

In particular, $\partial_n (u \cdot n)^2$ is not well-defined at the boundary and it is definitely not equal to zero in the case when the velocity field $u\in C^{0,\alpha} (\Omega)$ for $0 < \alpha \leq \frac{1}{2}$. As a result, it is necessary to consider (as was done in \eqref{boundarycondition}) the sum $\partial_n (p + (u \cdot n)^2)$ together at the boundary to obtain a well-defined boundary condition. The boundary condition \eqref{usualboundarycond} does not hold for weak solutions of the Euler equations, since one is implicitly assuming that $\partial_n (u \cdot n)^2 \big\lvert_{\partial \Omega} = 0$, while $\partial_n (u \cdot n)^2 \big\lvert_{\partial \Omega}$ may not even make sense as a distribution on $\partial \Omega$ when $\alpha \in (0, \frac{1}{2})$.

In Appendices \ref{schauderappendix1} and \ref{schauderappendix2} we establish some Schauder-type estimates that will be used throughout the paper. To conclude, in Appendix \ref{normalderivativeappendix} we will continue with the example given in section \ref{examplesection} and show that $\partial_n (u \cdot n)^2$ is not defined as a distribution for a dense set of points away from the boundary. 

\section{Local parametrisation of the boundary} \label{parametrisationsection}
We introduce a coordinate system for the region near the boundary $\partial \Omega$. We will assume throughout that the domain $\Omega$ is simply connected with $C^4$ boundary. 

Now we introduce the sets (for a given open set $U \subset \partial \Omega$)
\begingroup
\allowdisplaybreaks
\begin{align}
V_\delta &\coloneqq \{ x \in \Omega \lvert \; d(x, \partial \Omega) < \delta \}, \label{Vdeltadef} \\
V_{\delta , U} &\coloneqq  \{ x \in \Omega \lvert \; d(x, U) < \delta \}.
\end{align}
\endgroup
The fact that $\partial \Omega$ is $C^4$ means that around any point $x_0 \in \partial \Omega$ there exist a three-dimensional Cartesian coordinate system $(x_1, x_2, x_3)$ and a $C^4$ function $a : \mathbb{R}^2 \rightarrow \mathbb{R}$ such that the surface $\partial \Omega$ is locally parametrised as $(x_1, x_2, a (x_1, x_2) )$ on an open subset $U_{x_0} \subset \partial \Omega$. Then by the compactness of the boundary $\partial \Omega$, we know that there exist finitely many such subsets $U_1, \ldots, U_m$ which cover the boundary (about the corresponding points $x_1, \ldots, x_m$). 

Locally on $U_{x_0}$, the outward normal vector to $\partial \Omega$ is given by
\begin{equation*}
n = (n_1, n_2, n_3) = \frac{1}{\sqrt{1 + \big\lvert \frac{\partial a}{\partial x_1} \big\rvert^2 + \big\lvert \frac{\partial a}{\partial x_2} \big\rvert^2 }} \bigg( \frac{\partial a}{\partial x_1} ,  \frac{\partial a}{\partial x_2}, -1 \bigg).
\end{equation*}
Without loss of generality we can assume that $x_0 = 0$. Then for $\delta > 0$ small enough we then introduce the coordinate system (see \cite{bardos2021} and \cite[~Theorem 2.12]{wloka}, as well as \cite[~Appendix C.5]{hormander3})
\begingroup
\allowdisplaybreaks
\begin{align}
x_1 &= \sigma_1 - s n_1 (\sigma_1, \sigma_2), \label{transform1} \\
x_2 &= \sigma_2 - s n_2 (\sigma_1, \sigma_2), \label{transform2} \\
x_3 &= a (\sigma_1, \sigma_2) - s n_3 (\sigma_1, \sigma_2), \label{transform3}
\end{align}
\endgroup
for $(\sigma_1, \sigma_2, s) \in [0, \delta]^3$ and where $a(0,0) = 0$.
This transformation is $C^3$, as the normal vector $n$ is $C^3$. Alternatively, equations \eqref{transform1}-\eqref{transform3} can be written as follows 
\begin{equation}
x (\sigma_1, \sigma_2, s) = y (\sigma_1, \sigma_2) - s \cdot n (\sigma_1, \sigma_2),
\end{equation}
where $y$ moves on the local patch $U_{x_0}$ of the surface $\partial \Omega$ and is given by $y (\sigma_1, \sigma_2) = (y_1, y_2, y_3) = (\sigma_1, \sigma_2, a (\sigma_1, \sigma_2))$. We introduce the following notation for the coordinate transformation
\begin{equation}
\phi_{x_0} (\sigma_1, \sigma_2 , s) = x (\sigma_1, \sigma_2, s) = (x_1, x_2, x_3).
\end{equation}
Taking the derivative of $x$ in the normal coordinate variable $s$, we find that
\begin{equation*}
\partial_s x = - n (\sigma_1, \sigma_2).
\end{equation*}
Now we calculate the partial derivatives of $x$ with respect to the tangential variables $(\sigma_1, \sigma_2)$. We first observe that
\begin{align*}
\partial_{\sigma_1} y &= \bigg( 1 , 0, \frac{\partial a}{\partial \sigma_1} \bigg), \quad \partial_{\sigma_2} y = \bigg(0, 1, \frac{\partial a}{\partial \sigma_2} \bigg).
\end{align*}
It is easy to see that these vectors are orthogonal to $n (\sigma_1, \sigma_2)$. 
We also note that $\partial_{\sigma_1} n$ and $\partial_{\sigma_2} n$ are orthogonal to $n$ by definition (as $n$ has unit length). The ``tangent vectors'' at any point $x(\sigma_1, \sigma_2, s)$ are then given by
\begin{align*}
\tau_1 (\sigma_1, \sigma_2, s) &= \partial_{\sigma_1} x = \partial_{\sigma_1} y - s \partial_{\sigma_1} n, \\
\tau_2 (\sigma_1, \sigma_2, s) &= \partial_{\sigma_2} x = \partial_{\sigma_2} y - s \partial_{\sigma_2} n,
\end{align*}
which are orthogonal to $n(\sigma_1,\sigma_2)$ (as the component parts of the tangent vectors are). The vectors $\tau_1, \tau_2$ and $n$ form a basis for $\mathbb{R}^3$ for every point in $V_{\delta, U_{x_0}}$. However, we observe that in general this coordinate system is not orthogonal. 

Now, we turn to computing the gradient, divergence and Laplacian in this new coordinate system. The Jacobian matrix of the coordinate transformation is given by
\begin{equation} \label{jacobian}
J \coloneqq \frac{\partial x}{\partial (\sigma_1, \sigma_2, s)} = \begin{pmatrix}
1 - s \partial_{\sigma_1} n_1 & - s \partial_{\sigma_2} n_1 & - n_1 \\
- s \partial_{\sigma_1} n_2 & 1 - s \partial_{\sigma_2} n_2 & - n_2 \\
\partial_{\sigma_1} a - s \partial_{\sigma_1} n_3 & \partial_{\sigma_2} a - s \partial_{\sigma_2} n_3 & - n_3
\end{pmatrix} = \big(
\partial_{\sigma_1} y - s \partial_{\sigma_1} n, \partial_{\sigma_2} y - s \partial_{\sigma_2} n, - n \big).
\end{equation}
As shown in the proof of Theorem 2.12 in \cite{wloka}, as the Jacobian has nonzero determinant at $x_0$, the coordinate transformation is locally invertible and is a $C^3$ diffeomorphism.

We now introduce the following notation (for the sake of brevity)
\begin{align}
a_{ij} &\coloneqq \big(J^{-1} (J^{-1})^T \big)_{ij} \quad \text{for } i,j = 1, 2, 3, \label{aijdefinition} \\
b &\coloneqq \sqrt{\det ( J^T J)}.
\end{align}
It is easy to see that $(a_{ij})_{i,j=1}^3$ is a symmetric matrix, in fact it is the metric tensor that is associated with the coordinate system \eqref{transform1}-\eqref{transform3}. 

The gradient, the divergence and the Laplacian in the given coordinate system are given by (see equations 9.60, 9.69 and 9.70 in \cite{grinfeld})
\begingroup
\allowdisplaybreaks
\begin{align}
&\nabla f = \frac{\partial f}{\partial \sigma_1} \big( a_{11} \tau_1 + a_{21} \tau_2 + a_{31} n \big) + \frac{\partial f}{\partial \sigma_2} \big( a_{12} \tau_1 + a_{22} \tau_2 + a_{32} n \big) + \frac{\partial f}{\partial s} \big( a_{13} \tau_1 + a_{23} \tau_2 + a_{33} n \big) \\
&\nabla \cdot v = \frac{1}{b} \bigg[ \frac{\partial}{\partial \sigma_1} ( b v_1) + \frac{\partial}{\partial \sigma_2} ( b v_2) + \frac{\partial}{\partial s} ( b v_3) \bigg], \label{divnewcoordinates} \\
&\Delta f = \frac{1}{b} \bigg[ \frac{\partial}{\partial \sigma_1} \bigg( b \bigg[  a_{11} \frac{\partial f}{\partial \sigma_1} +  a_{12} \frac{\partial f}{\partial \sigma_2} + a_{13} \frac{\partial f}{\partial s} \bigg] \bigg) + \frac{1}{b} \bigg[ \frac{\partial}{\partial \sigma_2} \bigg( b \bigg[ a_{21} \frac{\partial f}{\partial \sigma_1} + a_{22} \frac{\partial f}{\partial \sigma_2} + a_{23} \frac{\partial f}{\partial s} \bigg] \bigg) \nonumber \\
&+ \frac{1}{b} \bigg[ \frac{\partial}{\partial s} \bigg( b \bigg[ a_{31} \frac{\partial f}{\partial \sigma_1} + a_{32} \frac{\partial f}{\partial \sigma_2} + a_{33} \frac{\partial f}{\partial s} \bigg] \bigg),
\end{align}
\endgroup
where $(v_1, v_2, v_3)$ are the components of the vector $v$ in the coordinate system $(\sigma_1, \sigma_2, s)$ (as described by equations \eqref{transform1}-\eqref{transform3}.  

We recall that we showed earlier that the normal vector $n$ has unit length and is orthogonal to the other tangent vectors at every point in $V_{\delta, U_{x_0}}$. This means in particular that (as $(a_{ij})$ is the metric tensor)
\begin{equation} \label{metricsimplification}
a_{33} = 1, \quad a_{31} = a_{13} = a_{21} = a_{12} = 0.
\end{equation}
We now introduce the operator $\Delta_\tau$ to be 
\begin{align}  \label{tangentiallaplacian}
\Delta_\tau f &\coloneqq \frac{1}{b} \bigg[ \frac{\partial}{\partial \sigma_1} \bigg( b \bigg[  a_{11} \frac{\partial f}{\partial \sigma_1} + a_{12} \frac{\partial f}{\partial \sigma_2} \bigg] \bigg) + \frac{1}{b} \bigg[ \frac{\partial}{\partial \sigma_2} \bigg( b \bigg[  a_{21} \frac{\partial f}{\partial \sigma_1} +  a_{22} \frac{\partial f}{\partial \sigma_2} \bigg] \bigg).
\end{align}
This allows us to rewrite the expression for the Laplacian as follows
\begin{align} \label{laplaciancoord}
\Delta f &= \Delta_\tau f + \frac{1}{b} \frac{\partial b}{\partial s} \frac{\partial f}{\partial s}   + \frac{\partial^2 f}{\partial s^2}.
\end{align}
Similarly, we can rewrite the expression for the gradient as follows
\begin{equation} \label{gradcoord}
\nabla f = \bigg( a_{11} \frac{\partial f}{\partial \sigma_1} + a_{12} \frac{\partial f}{\partial \sigma_2} \bigg) \tau_1 + \bigg( a_{21} \frac{\partial f}{\partial \sigma_1} + a_{22} \frac{\partial f}{\partial \sigma_2} \bigg) \tau_2 + \frac{\partial f}{\partial s} n \eqqcolon \nabla_\tau f + \frac{\partial f}{\partial s} n.
\end{equation}

\begin{remark}
We note that we could have derived the expression for the gradient, the divergence and the Laplacian directly. We observe that there is the following relation between the tangent vectors
\begin{align*}
\frac{\partial}{\partial x_1} &=  (J^{-1})_{11} \frac{\partial}{\partial \sigma_1} + (J^{-1})_{21} \frac{\partial}{\partial \sigma_2} + (J^{-1})_{31} \frac{\partial}{\partial s}, \\
\frac{\partial}{\partial x_2} &=  (J^{-1})_{12} \frac{\partial}{\partial \sigma_1} + (J^{-1})_{22} \frac{\partial}{\partial \sigma_2} + (J^{-1})_{32} \frac{\partial}{\partial s}, \\
\frac{\partial}{\partial x_3} &=  (J^{-1})_{13} \frac{\partial}{\partial \sigma_1} + (J^{-1})_{23} \frac{\partial}{\partial \sigma_2} + (J^{-1})_{33} \frac{\partial}{\partial s},
\end{align*}
and similarly we have (by using equation \eqref{jacobian})
\begingroup
\allowdisplaybreaks
\begin{align*}
\frac{\partial}{\partial \sigma_1} &= \bigg(1 - s \frac{\partial n_1}{\partial \sigma_1} (\sigma_1, \sigma_2) \bigg) \frac{\partial}{\partial x_1} - s \frac{\partial n_2}{\partial \sigma_1} (\sigma_1, \sigma_2) \frac{\partial}{\partial x_2} + \bigg(\frac{\partial a}{\partial \sigma_1} - s \frac{\partial n_3}{\partial \sigma_1} \bigg) \frac{\partial}{\partial x_3} , \\
\frac{\partial}{\partial \sigma_2} &= - s \frac{\partial n_1}{\partial \sigma_2} (\sigma_1, \sigma_2) \frac{\partial}{\partial x_1} + \bigg(1 - s \frac{\partial n_2}{\partial \sigma_2} (\sigma_1, \sigma_2) \bigg) \frac{\partial}{\partial x_2} + \bigg(\frac{\partial a}{\partial \sigma_2} - s \frac{\partial n_3}{\partial \sigma_2} \bigg) \frac{\partial}{\partial x_3}, \\
\frac{\partial}{\partial s} &= - n_1 \frac{\partial}{\partial x_1} - n_2 \frac{\partial}{\partial x_2} - n_3 \frac{\partial}{\partial x_3} = - n \cdot \nabla.
\end{align*}
\endgroup
These relations stipulate how the vector components transform between the different bases, this allows one to rewrite the differential operators in the coordinates from equations \eqref{transform1}-\eqref{transform3} by rewriting the standard expressions for these operators in Cartesian coordinates. 
\end{remark}
\begin{remark} 
We note that it is straightforward to extend coordinate system \eqref{transform1}-\eqref{transform3} to higher dimensions. As in the three-dimensional case, one can rely on the compactness of the boundary to obtain a finite number of surface patches to cover the boundary. The formulae for the gradient, divergence and Laplacian operators have higher-dimensional generalisations of the same form. 
\end{remark}
\begin{remark} \label{boundaryregularityremark}
In order to have a working proof we require that $\partial \Omega \in C^4$. Consequently, this implies that the normal vector $n$ is a $C^3$ function of $(\sigma_1, \sigma_2)$. The Jacobian matrix $J$ involves tangential derivatives of $n$, which makes that $J$ has $C^2$ regularity. \\
There are two reasons we need $J$ to have a high level of regularity. The first reason is that the proof of Theorem \ref{regularitytheorem} relies on taking smooth approximations of the velocity field $u$ (which we will construct in section \ref{mollificationsection}). For these smooth approximations there exists an associated smooth approximate pressure by standard Schauder theory. However, as the normal vector is part of the boundary condition \eqref{usualboundarycond} (which is equivalent to \eqref{veryweakboundcondeq} in this case), for the Schauder theory to apply we need $n \in C^{2,\alpha} (\Omega)$ for $\alpha \in (0,1)$.  \\ 
The second reason is that the proofs of several results in this paper involve (boundary) regularity estimates in the coordinate system \eqref{transform1}-\eqref{transform3}. Since the divergence and Laplacian in this coordinate system involve first-order derivatives of $J$, we need $J$ to be at least $C^1$. We expect that by using the variational formulation of the equation with the same choice of coordinates will allow us to weaken the boundary regularity requirement to $C^2$ instead of $C^4$. Then one has to perform the Schauder-type estimates from section \ref{decompositionsection} in the weak formulation, see for example \cite{giaquinta}.
\end{remark}
Throughout the paper, we will consider a modified pressure defined by (cf. Definition \ref{modifiedpressuredef})
\begin{equation}
P \coloneqq p + \phi (x) (u \cdot n)^2,
\end{equation}
where $\phi$ is a smooth cutoff function that was defined in \eqref{phifunction}. Now we will give the following precise statement of Theorem \ref{regularitytheorem} from the introduction.
\begin{theorem} \label{pressureregularity}
Let $\Omega$ be an open set in $\mathbb{R}^3$ with $C^4$ boundary and assume that $u \in C^{0, \alpha} (\Omega)$ for $\alpha \in (0,1)$ is a velocity field which is divergence-free and satisfies $(u \cdot n) \lvert_{\partial \Omega} = 0$. Then there is a unique function $P \in C^{0, \alpha} (\Omega)$ with the following properties:
\begin{enumerate}
    \item The function $P$ satisfies the following equation in $\Omega$
    \begin{equation}
    - \Delta P = (\nabla \otimes \nabla ) : (u \otimes u) - \Delta (\phi (x) ( u \cdot n )^2),
    \end{equation}
    which is satisfied in the sense of distributions. In particular, it means that for any test function $\psi \in \mathcal{D} (\Omega)$ (a weak formulation for the case $\psi \in \mathcal{D} (\overline{\Omega})$ will be given in Definition \ref{weakformulationdef})
    \begin{equation}
    - \int_{\Omega} P \Delta \psi dx = \int_{\Omega} u_i u_j \partial_i \partial_j \psi dx - \int_{\Omega} \phi (x) (u \cdot n)^2 \Delta \psi dx.
    \end{equation}
    Moreover, $P$ satisfies the boundary condition
    \begin{equation}
    \partial_n P = \big( u \otimes u \big) : \nabla n \quad \text{on } \partial \Omega.
    \end{equation}
    This equation holds in $H^{-2} (\partial \Omega)$.
    Moreover, the average of $P$ satisfies
    \begin{equation}
    \int_\Omega P dx = \int_\Omega \phi (x)  (u \cdot n)^2 dx,
    \end{equation}
    where $(u \cdot n)^2$ is defined locally on each patch $V_{\delta,U_i}$ and extended globally by using the partition of unity. 
    \item It satisfies the following estimate
    \begin{equation}
    \lVert P \rVert_{C^{0, \alpha} (\Omega)} \leq C \lVert u \otimes u \rVert_{C^{0, \alpha} (\Omega)},
    \end{equation}
    the positive constant $C$ depends only on $\Omega$ and $\alpha$.
    \item In any region $V_{\delta,U_i}$ for $i=1, \ldots, m$, the map $s \mapsto \partial_{s} P (\cdot, s)$ lies in the space $C([0, \delta); H^{-2} (U_i )  )$ where $s \in [0, \delta)$ is the normal coordinate and $U_i$ is the local patch of the boundary around some point $x_i \in \partial \Omega$. By using a partition of unity over the patches $U_1, \ldots, U_m$, the function $\partial_{s} P(\cdot, s)$ can be extended to a function in $C([0, \delta), H^{-2} (\partial \Omega))$, i.e. a function defined globally near the boundary.
\end{enumerate}
\end{theorem}

\section{Derivation of the weak formulation for the pressure} \label{weakformulationsection}
\subsection{Introduction}
In this section, we will derive a weak formulation of the boundary-value problem for the associated pressure of a weak solution to the Euler equations given in equation \eqref{neumannpressureproblem}. We will then prove Proposition \ref{weaksolutionproposition}, which states that the very weak boundary condition from Definition \ref{veryweakboundconddef} can be derived from the weak formulation (for the velocity field) for any Hölder continuous weak solution of the Euler equations. Finally, we then show that the weak formulation for the boundary-value problem for the pressure can be obtained directly from the weak formulation of the Euler equations assuming only Hölder regularity of the velocity field of the weak solution.

In particular, we consider the following boundary-value problem of the Euler equations
\begin{equation} \label{eulerboundaryvalueproblem}
\begin{cases}
\partial_t u + u \cdot \nabla u + \nabla p = 0, \quad \nabla \cdot u = 0, \quad \text{in } \Omega, \\
 (u \cdot n) = 0, \quad \text{on } \partial \Omega.
\end{cases}
\end{equation}
Next we need to introduce some new terminology.
\begin{definition} \label{modifiedpressuredef}
Let $\Omega$ be a bounded open set with $C^2$ boundary and let $u, p \in C^{0,\alpha} (\Omega)$ be Hölder continuous for some exponent $\alpha \in (0,1)$. Moreover, let $\phi$ be a given cutoff function (as introduced in \eqref{phifunction}). Then we define the \textit{modified pressure} to be
\begin{equation}
P \coloneqq p + \phi (x) (u \cdot n)^2.
\end{equation}
\end{definition}

Now we turn to the derivation of the weak formulation for the pressure. Consider a smooth solution $(u,p)$ of the Euler equations. We add and subtract the term $\nabla \big[ (u \cdot n)^2 \phi (x) \big]$ from the equations. After integrating against $\nabla \Phi$ for a test function $\Phi \in \mathcal{D} (\overline{\Omega})$ we find
\begin{align} \label{subtractedtermeuler}
\int_\Omega \bigg[ \partial_t u \cdot \nabla \Phi + \big[ \nabla \cdot \big(u \otimes u - (u \cdot n)^2 \phi I \big) \big] \cdot \nabla \Phi + \nabla \big( p + (u \cdot n^2) \phi) \cdot \nabla \Phi \bigg] dx = 0.
\end{align}
Due to the incompressibility condition (in the sense of distributions) one finds
\begin{equation*}
\int_\Omega \partial_t u \cdot \nabla \Phi dx = 0.
\end{equation*}
Using the definition of the modified pressure, we can rewrite equation \eqref{subtractedtermeuler} as
\begin{align} \label{H1formulation}
- \int_\Omega \nabla P \cdot \nabla \Phi dx = \int_\Omega \big[ \nabla \cdot \big(u \otimes u - (u \cdot n)^2 \phi (x) I \big) \big] \cdot \nabla \Phi dx.
\end{align}
Integrating by parts (using the definition of the distributional derivative) then gives
\begin{align*} 
&\int_\Omega P \Delta \Phi dx - \int_{\partial \Omega} P \partial_n \Phi dx = \int_\Omega \bigg[ - u \otimes u : (\nabla \otimes \nabla) \Phi + (u \cdot n)^2 \phi (x) \Delta \Phi \bigg] dx \\
&+ \int_{\partial \Omega} \bigg[ (u \cdot n) u \cdot \nabla \Phi  - (u \cdot n)^2 \partial_n \Phi \bigg] dx = \int_\Omega \bigg[ - u \otimes u : (\nabla \otimes \nabla) \Phi + (u \cdot n)^2 \phi \Delta \Phi \bigg] dx.
\end{align*}
Now we are able to define a weak solution for the pressure that is part of the boundary-value problem \eqref{eulerboundaryvalueproblem}.
\begin{definition} \label{weakformulationdef}
Let $u \in C^{0,\alpha} (\Omega)$ for $\alpha \in (0,1)$. We call a pressure function $P \in C^{0,\alpha} (\Omega)$ a weak solution of the boundary-value problem given by \eqref{pressureeq} and \eqref{veryweakboundcondeq}, if for all $\Phi \in \mathcal{D} (\overline{\Omega})$ we have
\begin{align*} 
\int_\Omega P \Delta \Phi dx - \int_{\partial \Omega} P \partial_n \Phi dx  &= \int_\Omega \bigg[ - u \otimes u : (\nabla \otimes \nabla) \Phi + (u \cdot n)^2 \phi (x) \Delta \Phi \bigg] dx.
\end{align*}
In addition, we require that $\partial_n P \in C^{0,\alpha} ([0,\delta) ; H^{-2} (\partial \Omega) )$ (where the regularity refers to the coordinates introduced in section \ref{parametrisationsection}).
\end{definition}
\begin{remark}
One should observe that the additional requirement $\partial_n P \in C^{0,\alpha} ([0,\delta) ; H^{-2} (\partial \Omega) )$ is not needed to make sense of the weak formulation in Definition \ref{weakformulationdef}, but it is necessary to recover the boundary condition as we will show in Theorem \ref{veryweakboundcondthm} (see also Proposition \ref{weaksolutionproposition}). In section \ref{limitsection} we will show (in the proof of Theorem \ref{pressureregularity}) the uniqueness of all weak solutions $P \in C^{0,\alpha} (\Omega)$ (without the additional regularity requirement for $\partial_n P$). Because in Theorem \ref{regularitytheorem} we demonstrate the existence of a pressure $P$ with this additional regularity for $\partial_n P$, including this additional requirement as part of Definition \ref{weakformulationdef} does therefore not provide a significant restriction.
\end{remark}
It is easy to see that for sufficiently smooth solutions one can recover the original boundary-value problem. If one assumes that $u \in H^{3/2+\lambda} (\Omega)$ for $\lambda > 0$ one can apply the Lax-Milgram theorem to formulation \eqref{H1formulation}, which is possible because in three dimensions the Sobolev embedding theorem implies that  $\nabla \cdot \big(u \otimes u - (u \cdot n)^2 \phi I \big) \in L^2 (\Omega)$. Therefore one obtains a unique pressure $P \in H^1 (\Omega)$, under the restriction that $\int_\Omega P dx = 0$. Then by assuming more regularity, for example $u \in H^2 (\Omega)$, one finds by standard elliptic regularity results that $P \in H^2 (\Omega)$ (see for example \cite[~Theorem 8.28]{salsa}). Subsequently, by undoing the integration by parts, we find in this case that
\begin{align*}
- \int_\Omega \Delta P \Phi dx + \int_{\partial \Omega} \Phi \partial_n P dx &= \int_\Omega \big[ (\nabla \otimes \nabla) : (u \otimes u) - \Delta ( (u \cdot n)^2 \phi (x)) \big] \Phi dx \\
&- \int_{\partial \Omega} \bigg[ \big[ \nabla \cdot (u \otimes u) \big] \cdot n - \partial_n (u \cdot n)^2\bigg] \Phi dx.
\end{align*}
By considering test functions $\Phi$ in $H^1_0 (\Omega)$, one deduces that the following equation is satisfied in $\Omega$
\begin{equation*}
-\Delta P = (\nabla \otimes \nabla) : (u \otimes u) - \Delta ((u \cdot n)^2 \phi),
\end{equation*}
where the equality holds in $L^2 (\Omega)$. We notice by using the trace theorem that $\partial_n P \in H^{1/2} (\partial \Omega)$, and similarly $- \big[ \nabla \cdot (u \otimes u) \big] \cdot n + \partial_n (u \cdot n)^2 \in H^{1/2} (\partial \Omega)$. The boundary condition can then be recovered in a strong form by choosing the test function $\big[ \partial_n P + \nabla \cdot (u \otimes u) \cdot n - \partial_n (u \cdot n)^2 \big] \phi (x)$ (which is an admissible test function as it lies in $H^1 (\Omega)$), which then gives
\begin{equation}
\partial_n P = \big[ \nabla \cdot (u \otimes u) \big] \cdot n - \partial_n (u \cdot n)^2 \quad \text{in } H^{1/2} (\partial \Omega).
\end{equation}
We recall that this holds under the assumption that $u , P \in H^2 (\Omega)$.
\subsection{Recovering the very weak boundary condition}
Next we will show that we can recover the very weak boundary condition (as introduced in Definition \ref{veryweakboundconddef}) directly from the weak formulation as given in Definition \ref{weakformulationdef}.
\begin{theorem} \label{veryweakboundcondthm}
Let $P$ be a weak solution in the sense of Definition \ref{weakformulationdef}, then $P$ satisfies the following boundary condition
\begin{equation}
\partial_n P = \big( u \otimes u \big) : \nabla n,
\end{equation}
where the equality holds in $H^{-2} (\partial \Omega)$.
\end{theorem}
\begin{proof}
To fix ideas and for the sake of simplicity, we first prove the result in the case of the half-space, i.e. $\Omega = \mathbb{R}^2 \times \mathbb{R}_+$ and $\partial \Omega = \mathbb{R}^2$. Then the result will be established for a general bounded domain (first locally and then globally). On the half-space, it holds that $\nabla n = 0$ and hence we must show that $\partial_n P = 0$ (in the sense of distributions). We take an arbitrary tangential test function $\Phi_\tau \in \mathcal{D} (\mathbb{R}^2)$ and choose $\Phi_n \in \mathcal{D} ([0, \infty) )$ with the following property
\begin{equation}
\Phi_n (z) \coloneqq \begin{cases}
1 \quad \text{if } 0 \leq z \leq \frac{1}{2}, \\
0 \quad \text{if } z \geq 1. \\
\end{cases}
\end{equation}
Then we define $\Phi_h$ as follows
\begin{equation} \label{Phihdefinition}
\Phi_h (z) \coloneqq \Phi_n \bigg( \frac{z}{h} \bigg).
\end{equation}
As before, we take $h$ small enough such that $\phi (x) \equiv 1$ on the support of $\Phi_h$. We denote the horizontal (tangential) gradient and Laplacian by $\nabla_\tau$ and $\Delta_\tau$, respectively. Inserting the test function $\Phi_h \cdot \Phi_\tau$ in the weak formulation of Definition \ref{weakformulationdef}, we find that
\begin{align}
&\int_\Omega P \Phi_h \Delta_\tau \Phi_\tau dx + \int_\Omega P \Phi_\tau \partial_{zz} \Phi_h dx - \int_{\partial \Omega} P \Phi_\tau \partial_n \Phi_h dx  = - \int_\Omega \Phi_h u_\tau \otimes u_\tau : (\nabla_\tau \otimes \nabla_\tau) \Phi_\tau  dx \nonumber \\
&- 2 \int_\Omega (u \cdot n) \partial_z \Phi_h u_\tau \cdot \nabla_\tau \Phi_\tau  dx + \int_\Omega (u \cdot n)^2 \Phi_h \Delta_\tau \Phi_\tau dx. \label{halfspacequation}
\end{align}
Then by using the continuity of $P$ and $u$, as well as the compact support of $\Phi_\tau$, we find that by sending $h \rightarrow 0$
\begin{align*}
&\int_\Omega P \Phi_h \Delta_\tau \Phi_\tau dx \xrightarrow[]{h \rightarrow 0} 0, \quad \int_\Omega \Phi_h u_\tau \otimes u_\tau : (\nabla_\tau \otimes \nabla_\tau) \Phi_\tau  dx \xrightarrow[]{h \rightarrow 0} 0, \\
&\int_\Omega (u \cdot n)^2 \Phi_h \Delta_\tau \Phi_\tau dx \xrightarrow[]{h \rightarrow 0} 0.
\end{align*}
Now using the fact that $\partial_z P \in C^{0,\alpha} ( [0, \delta]; H^{-2} (\mathbb{R}^2))$ (as is required as part of Definition \ref{weakformulationdef}), we find
\begin{equation*}
\int_\Omega P \Phi_\tau \partial_{zz} \Phi_h dx - \int_{\partial \Omega} P \Phi_\tau \partial_n \Phi_h dx = - \int_0^h \partial_z \Phi_h \langle \partial_z P (\cdot, z), \Phi_\tau \rangle_{H^{-2} (\mathbb{R}^2) \times H^2 (\mathbb{R}^2)} dz,
\end{equation*}
where we used the fact that $\supp (\Phi_h) \subset [0,h]$. We will show that the right-hand side converges to $\langle \partial_z P \lvert_{\partial \Omega} , \Phi_\tau \rangle_{H^{-2} (\mathbb{R}^2) \times H^2 (\mathbb{R}^2)}$ as $h \rightarrow 0$. It follows that
\begin{align*}
&\bigg\lvert - \int_0^h \partial_z \Phi_h \langle \partial_z P (\cdot, z), \Phi_\tau \rangle_{H^{-2} (\mathbb{R}^2) \times H^2 (\mathbb{R}^2)} dz - \langle \partial_z P \lvert_{\partial \Omega} , \Phi_\tau \rangle_{H^{-2} (\mathbb{R}^2) \times H^2 (\mathbb{R}^2)} \bigg\rvert \\
&\leq \bigg\lvert \frac{1}{h} \int_0^h \Phi_n' \bigg( \frac{z}{h} \bigg) \bigg[ \langle \partial_z P (\cdot, z), \Phi_\tau \rangle_{H^{-2} (\mathbb{R}^2) \times H^2 (\mathbb{R}^2)} - \langle \partial_z P \lvert_{\partial \Omega} , \Phi_\tau \rangle_{H^{-2} (\mathbb{R}^2) \times H^2 (\mathbb{R}^2)} \bigg] dz \bigg\rvert \\
&\leq \frac{C}{h} \int_0^h \big\lvert \langle \partial_z P (\cdot, z), \Phi_\tau \rangle_{H^{-2} (\mathbb{R}^2) \times H^2 (\mathbb{R}^2)} - \langle \partial_z P \lvert_{\partial \Omega} , \Phi_\tau \rangle_{H^{-2} (\mathbb{R}^2) \times H^2 (\mathbb{R}^2)} \big\rvert dz \xrightarrow[]{h \rightarrow 0} 0,
\end{align*}
where we have used that $\langle \partial_z P (\cdot, z) , \Phi_\tau \rangle_{H^{-2} (\mathbb{R}^2) \times H^2 (\mathbb{R}^2)}$ is a (Hölder) continuous function in $z$, see Definition \ref{weakformulationdef}. By a similar argument one can show that
\begin{equation*}
- \int_\Omega (u \cdot n) \partial_z \Phi_h u_\tau \cdot \nabla_\tau \Phi_\tau  dx \xrightarrow[]{h \rightarrow 0} \int_{\partial \Omega} (u \cdot n) u_\tau \cdot \nabla_\tau \Phi_\tau  dx = 0.
\end{equation*}
Therefore by combining all these results, by taking the limit $h \rightarrow 0$ of equation \eqref{halfspacequation} we find that 
\begin{equation}
\langle \partial_z P \lvert_{\partial \Omega} , \Phi_\tau \rangle_{H^{-2} (\mathbb{R}^2) \times H^2 (\mathbb{R}^2)} = 0.
\end{equation}
As this equation holds for all $\Phi_\tau \in \mathcal{D} (\mathbb{R}^2)$, we conclude in this case that $\partial_z P = 0$ in $H^{-2} (\mathbb{R}^2)$, which is what we had to show. 

Now we turn to the case of a general bounded domain $\Omega$. We first localise to a region $V_{\delta,U_j}$ (on which we have $\phi (x) \equiv 1$), by taking test functions $\Phi$ supported in this region. The modified pressure $P$ satisfies
\begin{equation} \label{localisedpressureeq}
\int_{V_{\delta,U_j}} P \Delta \Phi dx - \int_{U_j} P \partial_n \Phi dx  = \int_{V_{\delta,U_j}} \bigg[ - u \otimes u : (\nabla \otimes \nabla) \Phi + (u \cdot n)^2  \Delta \Phi \bigg] dx.
\end{equation}
Next we perform the change of variables from equations \eqref{transform1}-\eqref{transform3} (and we recall that $b = \sqrt{\det (J^T J)} = \lvert \det (J) \rvert$ and the $a_{ij}$ were defined in equation \eqref{aijdefinition}). We compute that (by using equation \eqref{gradcoord})
\begin{align*}
(\nabla \otimes \nabla) \Phi &= \nabla \bigg[ \bigg( a_{11} \frac{\partial \Phi}{\partial \sigma_1} + a_{12} \frac{\partial \Phi}{\partial \sigma_2} \bigg) \tau_1 + \bigg( a_{21} \frac{\partial \Phi}{\partial \sigma_1} + a_{22} \frac{\partial \Phi}{\partial \sigma_2} \bigg) \tau_2 + \frac{\partial \Phi}{\partial s} n \bigg] \\
&= I_{\Phi,1} + I_{\Phi,2} + \frac{\partial}{\partial s} \bigg( a_{11} \frac{\partial \Phi}{\partial \sigma_1} + a_{12} \frac{\partial \Phi}{\partial \sigma_2} \bigg) n \otimes \tau_1 + \frac{\partial}{\partial s} \bigg( a_{21} \frac{\partial \Phi}{\partial \sigma_1} + a_{22} \frac{\partial \Phi}{\partial \sigma_2} \bigg) n \otimes \tau_2 \\
&+ \bigg( a_{11} \frac{\partial^2 \Phi}{\partial \sigma_1 \partial s} + a_{12} \frac{\partial^2 \Phi}{\partial \sigma_2 \partial s} \bigg) \tau_1 \otimes n + \bigg( a_{21} \frac{\partial^2 \Phi}{\partial \sigma_1 \partial s} + a_{22} \frac{\partial^2 \Phi}{\partial \sigma_2 \partial s} \bigg) \tau_2 \otimes n + \frac{\partial^2 \Phi}{\partial s^2} n \otimes n \\
&+ \frac{\partial \Phi}{\partial s} \nabla n,
\end{align*}
where we have defined
\begin{align*}
I_{\Phi,1} &\coloneqq \bigg[ a_{11} \frac{\partial}{\partial \sigma_1} \bigg( a_{11} \frac{\partial \Phi}{\partial \sigma_1} + a_{12} \frac{\partial \Phi}{\partial \sigma_2} \bigg) + a_{12} \frac{\partial}{\partial \sigma_2} \bigg( a_{11} \frac{\partial \Phi}{\partial \sigma_1} + a_{12} \frac{\partial \Phi}{\partial \sigma_2} \bigg) \bigg] \tau_1 \otimes \tau_1 \\
&+ \bigg[ a_{21} \frac{\partial}{\partial \sigma_1} \bigg( a_{11} \frac{\partial \Phi}{\partial \sigma_1} + a_{12} \frac{\partial \Phi}{\partial \sigma_2} \bigg) + a_{22} \frac{\partial}{\partial \sigma_2} \bigg( a_{11} \frac{\partial \Phi}{\partial \sigma_1} + a_{12} \frac{\partial \Phi}{\partial \sigma_2} \bigg) \bigg] \tau_2 \otimes \tau_1 \\
&+ \bigg[ a_{11} \frac{\partial}{\partial \sigma_1} \bigg( a_{21} \frac{\partial \Phi}{\partial \sigma_1} + a_{22} \frac{\partial \Phi}{\partial \sigma_2} \bigg) + a_{12} \frac{\partial}{\partial \sigma_2} \bigg( a_{21} \frac{\partial \Phi}{\partial \sigma_1} + a_{22} \frac{\partial \Phi}{\partial \sigma_2} \bigg) \bigg] \tau_1 \otimes \tau_2 \\
&+ \bigg[ a_{21} \frac{\partial}{\partial \sigma_1} \bigg( a_{21} \frac{\partial \Phi}{\partial \sigma_1} + a_{22} \frac{\partial \Phi}{\partial \sigma_2} \bigg) + a_{22} \frac{\partial}{\partial \sigma_2} \bigg( a_{21} \frac{\partial \Phi}{\partial \sigma_1} + a_{22} \frac{\partial \Phi}{\partial \sigma_2} \bigg) \bigg] \tau_2 \otimes \tau_2, \\
I_{\Phi,2} &\coloneqq \bigg( a_{11} \frac{\partial \Phi}{\partial \sigma_1} + a_{12} \frac{\partial \Phi}{\partial \sigma_2} \bigg) \nabla \tau_1 + \bigg( a_{21} \frac{\partial \Phi}{\partial \sigma_1} + a_{22} \frac{\partial \Phi}{\partial \sigma_2} \bigg) \nabla \tau_2.
\end{align*}
Then by using the definitions from equations \eqref{divnewcoordinates} and \eqref{tangentiallaplacian}-\eqref{gradcoord}, we can rewrite equation \eqref{localisedpressureeq} as follows
\begin{align*}
&\int_{[0,\delta]^3} P \bigg[ \Delta_\tau \Phi + \frac{1}{b} \partial_s b \partial_s \Phi + \partial_s^2 \Phi \bigg] b d \sigma_1 d\sigma_2 ds - \int_{[0,\delta]^2} P \partial_s \Phi b d \sigma_1 d \sigma_2  \\
&= - \int_{[0,\delta]^3} u \otimes u : \bigg[  I_{\Phi,1} + I_{\Phi,2} + \frac{\partial \Phi}{\partial s} \nabla n + \frac{\partial^2 \Phi}{\partial s^2} n \otimes n \bigg] b d \sigma_1 d\sigma_2 ds \\
&- 2 \int_{[0,\delta]^3} u \otimes u : \bigg[  \bigg( a_{11} \frac{\partial^2 \Phi}{\partial \sigma_1 \partial s} + a_{12} \frac{\partial^2 \Phi}{\partial \sigma_2 \partial s} \bigg) \tau_1 \otimes n + \bigg( a_{21} \frac{\partial^2 \Phi}{\partial \sigma_1 \partial s} + a_{22} \frac{\partial^2 \Phi}{\partial \sigma_2 \partial s} \bigg) \tau_2 \otimes n \bigg] b d \sigma_1 d\sigma_2 ds \\
&- \int_{[0,\delta]^3} u \otimes u : \bigg[  \bigg( \frac{\partial a_{11}}{\partial s} \frac{\partial \Phi}{\partial \sigma_1 } + \frac{\partial a_{12}}{\partial s} \frac{\partial \Phi}{\partial \sigma_2} \bigg) \tau_1 \otimes n + \bigg( \frac{\partial a_{21}}{\partial s} \frac{\partial \Phi}{\partial \sigma_1 } + \frac{\partial a_{22}}{\partial s} \frac{\partial \Phi}{\partial \sigma_2 } \bigg) \tau_2 \otimes n \bigg] b d \sigma_1 d\sigma_2 ds \\
&+ \int_{[0,\delta]^3} (u \cdot n)^2 \bigg[ \Delta_\tau \Phi + \frac{1}{b} \partial_s b \partial_s \Phi + \partial_s^2 \Phi \bigg] b d \sigma_1 d\sigma_2 ds.
\end{align*}
As before, we choose a test function of the form $\Phi = \Phi_\tau \Phi_h$ where $\Phi_\tau \in \mathcal{D} ([0,\delta]^2)$ is a tangential test function and where $\Phi_h$ is the test function defined in equation \eqref{Phihdefinition} (but now with the $z$-coordinate replaced with the normal coordinate $s$). Inserting this choice of test function in the previous equation then yields
\begin{align*}
&\int_{[0,\delta]^3} P \bigg[ \Phi_h \Delta_\tau \Phi_\tau + \frac{1}{b} \Phi_\tau \partial_s b \partial_s \Phi_h + \Phi_\tau \partial_s^2 \Phi_h \bigg] b d \sigma_1 d\sigma_2 ds - \int_{[0,\delta]^2} P \Phi_\tau \partial_s \Phi_h b d \sigma_1 d \sigma_2  \\
&= - \int_{[0,\delta]^3} u \otimes u : \bigg[ \Phi_h I_{\Phi_\tau,1} + \Phi_h I_{\Phi_\tau,2} + \Phi_\tau \frac{\partial \Phi_h}{\partial s} \nabla n \bigg] b d \sigma_1 d\sigma_2 ds - 2 \int_{[0,\delta]^3} u \otimes u : \bigg[  \bigg( a_{11} \frac{\partial \Phi_\tau}{\partial \sigma_1} \frac{\partial \Phi_h}{\partial s} \\
&+ a_{12} \frac{\partial \Phi_\tau}{\partial \sigma_2} \frac{\partial \Phi_h}{\partial s} \bigg) \tau_1 \otimes n + \bigg( a_{21} \frac{\partial \Phi_\tau}{\partial \sigma_1} \frac{\partial \Phi_h}{\partial s} + a_{22} \frac{\partial \Phi_\tau}{\partial \sigma_2} \frac{\partial \Phi_h}{\partial s} \bigg) \tau_2 \otimes n \bigg] b d \sigma_1 d\sigma_2 ds \\
&- \int_{[0,\delta]^3} u \otimes u : \bigg[ \Phi_h \bigg( \frac{\partial a_{11}}{\partial s} \frac{\partial \Phi_\tau}{\partial \sigma_1 } + \frac{\partial a_{12}}{\partial s} \frac{\partial \Phi_\tau}{\partial \sigma_2} \bigg) \tau_1 \otimes n + \Phi_h \bigg( \frac{\partial a_{21}}{\partial s} \frac{\partial \Phi_\tau}{\partial \sigma_1 } + \frac{\partial a_{22}}{\partial s} \frac{\partial \Phi_\tau}{\partial \sigma_2 } \bigg) \tau_2 \otimes n \bigg] b d \sigma_1 d\sigma_2 ds \\
&+ \int_{[0,\delta]^3} (u \cdot n)^2 \bigg[ \Phi_h \Delta_\tau \Phi_\tau + \frac{\Phi_\tau}{b} \partial_s b \partial_s \Phi_h \bigg] b d \sigma_1 d\sigma_2 ds.
\end{align*}
Then, by proceeding as in the case of the half-space, we find (by using the continuity of $P$ and $u$)
\begin{align*}
&\int_{[0,\delta]^3} P \Phi_h \Delta_\tau \Phi_\tau b d \sigma_1 d \sigma_2 ds \xrightarrow[]{h \rightarrow 0} 0, \quad \int_{[0,\delta]^3} u \otimes u : \bigg[ \Phi_h I_{\Phi_\tau,1} + \Phi_h I_{\Phi_\tau,2}  \bigg] b d \sigma_1 d\sigma_2 ds \xrightarrow[]{h \rightarrow 0} 0, \\
&\int_{[0,\delta]^3} u \otimes u : \bigg[ \Phi_h \bigg( \frac{\partial a_{11}}{\partial s} \frac{\partial \Phi_\tau}{\partial \sigma_1 } + \frac{\partial a_{12}}{\partial s} \frac{\partial \Phi_\tau}{\partial \sigma_2} \bigg) \tau_1 \otimes n \\
&+ \Phi_h \bigg( \frac{\partial a_{21}}{\partial s} \frac{\partial \Phi_\tau}{\partial \sigma_1 } + \frac{\partial a_{22}}{\partial s} \frac{\partial \Phi_\tau}{\partial \sigma_2 } \bigg) \tau_2 \otimes n \bigg] b d \sigma_1 d\sigma_2 ds \xrightarrow[]{h \rightarrow 0} 0, \quad \int_{[0,\delta]^3} (u \cdot n)^2 \Phi_h \Delta_\tau \Phi_\tau b d \sigma_1 d\sigma_2 ds \xrightarrow[]{h \rightarrow 0} 0.
\end{align*}
Then by integrating by parts, and using the regularity of $\partial_s P$, we obtain
\begin{align*}
&\int_{[0,\delta]^3} \Phi_\tau P \bigg[\frac{1}{b} \partial_s b \partial_s \Phi_h + \partial_s^2 \Phi_h \bigg] b d \sigma_1 d\sigma_2 ds - \int_{[0,\delta]^2} \Phi_\tau P \partial_s \Phi_h b d \sigma_1 d \sigma_2 \\
&= - \int_0^h \partial_s \Phi_h \langle \partial_s P (\cdot, s), \Phi_\tau b \rangle_{H^{-2} ([0,\delta]^2) \times H^2 ([0,\delta]^2)} ds.
\end{align*}
Now we recall that $\langle \partial_s P (\cdot, s), \Phi_\tau b \rangle_{H^{-2} ([0,\delta]^2) \times H^2 ([0,\delta]^2)}$ is Hölder continuous in the variable $s$, and hence as before we get
\begin{align*}
- \int_0^\delta \partial_s \Phi_h \langle \partial_s P (\cdot, s), \Phi_\tau b \rangle_{H^{-2} ([0,\delta]^2) \times H^2 ([0,\delta]^2)} ds &\xrightarrow[]{h \rightarrow 0} \langle \partial_s P \lvert_{\partial \Omega} , \Phi_\tau b \rangle_{H^{-2} ([0,\delta]^2) \times H^2 ([0,\delta]^2)} \\
&= \langle \partial_n P \lvert_{\partial \Omega} , \Phi_\tau \rangle_{H^{-2} (U_j) \times H^2 (U_j)}.
\end{align*}
Similarly, we find
\begin{align*}
- \int_{[0,\delta]^3} \Phi_\tau \frac{\partial \Phi_h}{\partial s} u \otimes u : \nabla n \; b d \sigma_1 d\sigma_2 ds \xrightarrow[]{h \rightarrow 0} \int_{[0,\delta]^2} \Phi_\tau (u \otimes u)\lvert_{\partial \Omega} : \nabla n \; b d \sigma_1 d\sigma_2,
\end{align*}
because we can obtain the following estimate
\begin{align*}
&\bigg\lvert - \int_{[0,\delta]^3} \Phi_\tau \frac{\partial \Phi_h}{\partial s} u \otimes u : \nabla n \; b d \sigma_1 d\sigma_2 ds - \int_{[0,\delta]^2} \Phi_\tau (u \otimes u)\lvert_{\partial \Omega} : \nabla n \; b d \sigma_1 d\sigma_2 \bigg\rvert \\
&= \bigg\lvert - \int_0^\delta \frac{\partial \Phi_h}{\partial s} \bigg[ \int_{[0,\delta]^2} \Phi_\tau  \big( b u \otimes u - b  u \otimes u \lvert_{\partial \Omega} \big) : \nabla n \; d \sigma_1 d\sigma_2 \bigg]  ds \bigg\rvert \xrightarrow[]{h \rightarrow 0} 0,
\end{align*}
by using the Hölder regularity of $u$. Proceeding in the same way, we get (using the no-normal flow boundary condition $(u \cdot n) \lvert_{\partial \Omega} = 0$)
\begin{align*}
&\int_{[0,\delta]^3} u \otimes u : \bigg[  \bigg( a_{11} \frac{\partial \Phi_\tau}{\partial \sigma_1} \frac{\partial \Phi_h}{\partial s} + a_{12} \frac{\partial \Phi_\tau}{\partial \sigma_2} \frac{\partial \Phi_h}{\partial s} \bigg) \tau_1 \otimes n \\
&+ \bigg( a_{21} \frac{\partial \Phi_\tau}{\partial \sigma_1} \frac{\partial \Phi_h}{\partial s} + a_{22} \frac{\partial \Phi_\tau}{\partial \sigma_2} \frac{\partial \Phi_h}{\partial s} \bigg) \tau_2 \otimes n \bigg] b d \sigma_1 d\sigma_2 ds \xrightarrow[]{h \rightarrow 0} 0, \\
&\int_{[0,\delta]^3} (u \cdot n)^2  \Phi_\tau \partial_s b \partial_s \Phi_h d \sigma_1 d\sigma_2 ds \xrightarrow[]{h \rightarrow 0} 0.
\end{align*}
Combining these results, we find (for all $\Phi_\tau \in H^2 (U_j)$)
\begin{equation}
\langle \partial_n P , \Phi_\tau \rangle_{H^{-2} (U_j) \times H^2 (U_j)} = \int_{U_j} \Phi_\tau (u \otimes u)\lvert_{\partial \Omega} : \nabla n \; d x = \langle (u \otimes u)\lvert_{\partial \Omega} : \nabla n , \Phi_\tau \rangle_{H^{-2} (U_j) \times H^2 (U_j)}.
\end{equation}
Then by using a partition of unity over the sets $\{ U_j \}$, we find for all $\Phi_\tau \in H^2 (\partial \Omega)$
\begin{equation}
\langle \partial_n P , \Phi_\tau \rangle_{H^{-2} (\partial \Omega) \times H^2 (\partial \Omega)} = \langle (u \otimes u) : \nabla n , \Phi_\tau \rangle_{H^{-2} (\partial \Omega) \times H^2 (\partial \Omega)}.
\end{equation}
As this equation holds for all $\Phi_\tau \in H^2 (\partial \Omega)$, we deduce that $\partial_n P = (u \otimes u) : \nabla n$ as elements in $H^{-2} (\partial \Omega)$, which concludes the proof.
\end{proof}
\subsection{Rigorous derivation of the weak formulation}
We now demonstrate that the weak formulation for the pressure from Definition \ref{weakformulationdef} can be derived directly for weak Hölder continuous solutions of the Euler equations. We first recall the definition of a weak solution.
\begin{definition} \label{weaksoleuler}
We say that a velocity field $u \in L^\infty ((0,T); L^2 (\Omega))$ and a pressure $p \in L^\infty ((0,T); L^1 (\Omega))$ are a weak solution of the Euler equations if for all test functions $\chi \in \mathcal{D} (\Omega \times (0,T); \mathbb{R}^3)$
\begin{equation}
\int_0^T \int_\Omega \bigg[ u \cdot \partial_t \chi + u \otimes u : \nabla \chi + p \nabla \cdot \chi \bigg] dx dt = 0.
\end{equation}
Moreover, the velocity field is weakly divergence-free. Therefore for all $\zeta \in \mathcal{D} (\Omega \times (0,T); \mathbb{R})$ it must hold that
\begin{equation}
\int_0^T \int_\Omega u \cdot \nabla \zeta dx dt = 0.
\end{equation}
Finally, the boundary condition $\big( u \cdot n \big) \lvert_{\partial \Omega} = 0$ is satisfied in the sense of trace in $H^{-1/2} (\partial \Omega)$, see the remark below for further details.
\end{definition}
\begin{remark}
In the setting of this paper, namely when $u \in C^{0,\alpha} (\Omega)$, the boundary condition $\big( u \cdot n \big) \lvert_{\partial \Omega} = 0$ is satisfied pointwise on $\partial \Omega$. However, for general weak solutions with $u \in L^\infty ((0,T); L^2 (\Omega))$ one can apply a standard trace theorem (see for example \cite[~Lemma 20.2]{tartar} and \cite[~Proposition 1.4]{constantinbook}) to show that $\big( u \cdot n \big) \lvert_{\partial \Omega} \in H^{-1/2} (\partial \Omega)$, and hence the no-normal flow boundary condition is satisfied in that space.
\end{remark}
We will show the following result.
\begin{theorem} \label{weakformulationthm}
Let $(u,p) \in L^\infty ((0,T); C^{0,\alpha} (\Omega)$ be a weak solution of the Euler equations in the sense of Definition \ref{weaksoleuler}. Then the modified pressure $P$ satisfies the weak formulation from Definition \ref{weakformulationdef}.
\end{theorem}
\begin{proof}
In the weak formulation from Definition \ref{weaksoleuler} we choose the test function $\chi = \varphi (1 - \widetilde{\Phi}_h) \nabla \Psi$, where $\Psi \in \mathcal{D} (\overline{\Omega})$ and $\varphi \in \mathcal{D} (0,T)$ and we have defined 
\begin{equation*}
\widetilde{\Phi}_h (x) \coloneqq \Phi_h (d(x,\partial \Omega)),
\end{equation*}
where we recall that $\Phi_h$ was defined in equation \eqref{Phihdefinition}. We observe that (using that for an outward-pointing normal vector, we have $\nabla d (x,\partial \Omega) = - n$, see for example \cite[~eq. (1.10)]{titi2018} and \cite[~eq. (4.2)]{titi2019})
\begin{equation*}
\nabla \widetilde{\Phi}_h (x) = - \frac{1}{h} \Phi_n' \bigg( \frac{d(x,\partial \Omega)}{h} \bigg) n (\widetilde{\sigma} (x)),
\end{equation*}
where we recall that $\widetilde{\sigma} (x) \in \partial \Omega$ is the unique point on the boundary such that $d (x, \partial \Omega) = \lvert x - \widetilde{\sigma} (x) \rvert$.
Inserting the aforementioned test function $\chi = \varphi (1 - \widetilde{\Phi}_h) \nabla \Psi$ in the weak formulation from Definition \ref{weaksoleuler} gives
\begin{align*}
0 &= \int_0^T \varphi \bigg[ \int_\Omega \big( 1 - \widetilde{\Phi}_h) \bigg[ u \otimes u : (\nabla \otimes \nabla) \Psi + P \Delta \Psi - \phi (u \cdot n)^2 \Delta \Psi \bigg] dx \bigg] dt \\
&+ \int_0^T \int_\Omega u \cdot \nabla \Psi (1 - \widetilde{\Phi}_h) \partial_t \varphi dx dt + \int_0^T \varphi \bigg[ \int_\Omega \frac{1}{h} \bigg[ (u \cdot n)  \Phi_n' \bigg( \frac{d(x,\partial \Omega)}{h} \bigg) u \cdot \nabla \Psi \\
&+ P  \Phi_n' \bigg( \frac{d(x,\partial \Omega)}{h} \bigg) \partial_n \Psi - \phi (u \cdot n)^2 \Phi_n' \bigg( \frac{d(x,\partial \Omega)}{h} \bigg) \partial_n \Psi \bigg] dx \bigg] dt,
\end{align*}
where we have added and subtracted $\phi (u \cdot n)^2$ in order to replace the pressure $p$ by the modified pressure $P$. Then by sending $h \rightarrow 0$ we find
\begin{align*}
&\int_0^T \varphi \bigg[ \int_\Omega \big( 1 - \widetilde{\Phi}_h) \bigg[ u \otimes u : (\nabla \otimes \nabla) \Psi + P \Delta \Psi - \phi (u \cdot n)^2 \Delta \Psi \bigg] dx \bigg] dt \\
&\xrightarrow[]{h \rightarrow 0} \int_0^T \varphi \bigg[ \int_\Omega \bigg[ u \otimes u : (\nabla \otimes \nabla) \Psi + P \Delta \Psi - \phi (u \cdot n)^2 \Delta \Psi \bigg] dx \bigg] dt, \\
&\int_0^T \int_\Omega u \cdot \nabla \Psi (1 - \widetilde{\Phi}_h) \partial_t \varphi dx dt \xrightarrow[]{h \rightarrow 0} \int_0^T \int_\Omega u \cdot \nabla \Psi \partial_t \varphi dx dt = 0.
\end{align*}
By applying the Lebesgue differentiation theorem we then obtain
\begin{align*}
&\int_0^T \varphi \bigg[ \int_\Omega \frac{1}{h} \bigg[ (u \cdot n)  \Phi_n' \bigg( \frac{d(x,\partial \Omega)}{h} \bigg) u \cdot \nabla \Psi + P \Phi_n' \bigg( \frac{d(x,\partial \Omega)}{h} \bigg) \partial_n \Psi \\
&- \phi (u \cdot n)^2 \Phi_n' \bigg( \frac{d(x,\partial \Omega)}{h} \bigg) \partial_n \Psi \bigg] dx \bigg] dt \xrightarrow[]{h \rightarrow 0} - \int_0^T \varphi \bigg[ \int_{\partial \Omega} \bigg[ (u \cdot n)  u \cdot \nabla \Psi + P \partial_n \Psi \\
&- \phi (u \cdot n)^2 \partial_n \Psi \bigg] dx \bigg] dt = -\int_0^T \varphi \bigg[ \int_{\partial \Omega} P \partial_n \Psi dx \bigg] dt,
\end{align*}
where we have used the boundary condition for $u$. Therefore we find the following equation
\begin{equation*}
\int_0^T \varphi \bigg[ \int_\Omega \bigg[ u \otimes u : (\nabla \otimes \nabla) \Psi + P \Delta \Psi - \phi (u \cdot n)^2 \Delta \Psi \bigg] dx - \int_{\partial \Omega} P \partial_n \Psi dx \bigg] dt = 0.
\end{equation*}
By choosing $\varphi (t) \coloneqq \frac{1}{\epsilon} \Phi_\epsilon (t)$ (i.e. setting $h = \epsilon$ in \eqref{Phihdefinition}) and then sending $\epsilon \rightarrow 0$, we find that for almost every $t \in (0,T)$ and all $\Psi \in \mathcal{D} (\overline{\Omega})$
\begin{equation*}
\int_\Omega \bigg[ u \otimes u : (\nabla \otimes \nabla) \Psi + P \Delta \Psi - \phi (u \cdot n)^2 \Delta \Psi \bigg] dx - \int_{\partial \Omega} P \partial_n \Psi dx = 0,
\end{equation*}
which is what we had to show.
\end{proof}
By combining the results of Theorems \ref{regularitytheorem}, \ref{veryweakboundcondthm} and \ref{weakformulationthm} we can deduce Proposition \ref{weaksolutionproposition}.

\section{Mollification of the velocity field} \label{mollificationsection}
\begin{lemma} \label{mollificationlemma}
Let $\Omega$ be a bounded domain with $\partial \Omega \in C^2$, and consider a velocity field $u \in C^{0, \alpha} (\Omega)$ which is divergence-free and tangential to the boundary (so $(u \cdot n) \lvert_{\partial \Omega} = 0$), there exists a family of divergence-free velocity fields $u^\epsilon \in C^\infty (\overline{\Omega})$ (with $(u^\epsilon \cdot n) \lvert_{\partial \Omega} = 0$) which converge to $u $ in $C^{0,\beta} (\overline{\Omega})$ for $\beta \in (0,\alpha)$ as $\epsilon \rightarrow 0$. In addition, we have the estimate
\begin{equation}
\lVert u^\epsilon \rVert_{C^{0, \alpha} (\Omega)} \leq C \lVert u \rVert_{C^{0, \alpha} (\Omega)},
\end{equation}
where $C$ is a positive constant which is independent of $\epsilon$ and $u$.
\end{lemma}
Before we prove this lemma, we recall several results from the literature on 3D vector potentials (i.e. stream functions).
\begin{proposition} \label{streamfunction}
Let $\Omega$ be a simply connected $C^2$ domain and let $u \in C^{0,\alpha} (\Omega)$ be a divergence-free vector field such that $(u \cdot n) \lvert_{\partial \Omega} = 0$, for a given $\alpha \in (0,1)$. Then there exists a unique vector potential (stream function) $\psi \in C^{1,\alpha} (\Omega)$ with the following properties
\begin{equation} \label{streamfunctionproperties}
u = \nabla \times \psi, \quad \nabla \cdot \psi = 0, \quad (\psi \times n) \lvert_{\partial \Omega} = 0.
\end{equation}
Moreover, the vector potential satisfies the estimate
\begin{equation} \label{streamfunctionestimate}
\lVert \psi \rVert_{C^{1,\alpha} (\Omega)} \leq C \lVert u \rVert_{C^{0,\alpha} (\Omega)},
\end{equation}
where $C$ is a positive constant which is independent of $u$.
Finally, the vector potential $\psi$ can be characterised uniquely as a solution of the following elliptic boundary-value problem
\begin{equation}
\begin{cases}
- \Delta \psi = \nabla \times u, \quad \nabla \cdot \psi = 0, \quad \text{in } \Omega, \\
\psi \times n = 0, \quad \text{on } \partial \Omega.
\end{cases}
\end{equation}
\end{proposition}
\begin{proof}
The existence result for vector potentials in the $L^2$ case can be found in \cite[~Lemma 3.5 and Theorem 3.17]{amrouche} and \cite[~Theorem 3.4 and Theorem 3.6]{girault}. In the $L^p$ setting the existence results and regularity estimates were proved in \cite[~Lemma 4.1 and Theorem 4.3]{amrouche2013}. Finally, regularity estimates on the vector potential in the context of Hölder spaces were obtained in \cite[~Theorem 2.1]{bolik} and also \cite[~Proposition 2.1]{bates} (for the higher order estimates). Note that in this discussion we have restricted our attention to normal vector potentials, but analogous results in the tangential case can also be found in the aforementioned references.
\end{proof}
\begin{proof}[Proof of Lemma \ref{mollificationlemma}]
First by Proposition \ref{streamfunction} we know there exists a stream function (vector potential) $\psi$ which satisfies estimate \eqref{streamfunctionestimate} and has the properties given in equation \eqref{streamfunctionproperties}. Next we introduce the following sets
\begin{align}
\widetilde{V}_\delta &\coloneqq \{ x \in \mathbb{R}^3 \lvert \; d(x, \partial \Omega) < \delta \}, \\
\widetilde{V}_{\delta , U} &\coloneqq  \{ x \in V_\delta \lvert \; d(x, U) < \delta \}.
\end{align}
We introduce a function $\phi_1 \in C^2_c (\mathbb{R}^3)$ such that $\text{supp} (\phi_1) \subset \widetilde{V_\delta}$. We then introduce a partition of unity $\rho_1, \ldots, \rho_m$ of the sets $V_{\delta, U_1}, \ldots, V_{\delta, U_m}$ (which cover the region near the boundary, as was done in section \ref{parametrisationsection}). Recall that we take $\delta > 0$ suitably small such that the normal vector $n$ can be extended inside the sets $V_{\delta, U_1}, \ldots, V_{\delta, U_m}$ in a unique fashion.

We now define the following decomposition
\begin{align*}
\psi &= \psi_b + \psi_i \coloneqq \phi_1 \psi + (1- \phi_1) \psi, \\
\psi_b &= \phi_1 \rho_1 \psi + \ldots + \phi_1 \rho_m \psi \eqqcolon \psi_1 + \ldots + \psi_m.
\end{align*}
We introduce a nonnegative radial mollifier $\varphi$ with support in $B_0 (1)$ and the property $\int_{\mathbb{R}^3} \varphi (x) dx = 1$. Moreover, we define
\begin{equation*}
\varphi_\epsilon (x) \coloneqq \frac{1}{\epsilon^3} \varphi \bigg( \frac{x}{\epsilon} \bigg).
\end{equation*}
Observe that $\psi_j \in C^{1, \alpha}_c (V_{\delta, U_j}) $ for $j=1, \ldots, m$. 

First we deal with the interior part of the velocity field, which we observe to have compact support in $\Omega$. We define the function
\begin{equation*}
\psi^\epsilon_i (x) \coloneqq \varphi_\epsilon * \psi_i \in C_c^\infty (\Omega),
\end{equation*}
for $\epsilon$ suitably small. Note that $\psi^\epsilon_i$ converges to $\psi_i$ in the $C^1 (\overline{\Omega})$ norm as $\epsilon \rightarrow 0$ by standard mollification estimates. Moreover, it holds that $\lVert \psi^\epsilon_i \rVert_{C^{1, \alpha} (\Omega )} \leq C \lVert \psi_i \rVert_{C^{1, \alpha} (\Omega)}$. Therefore $u_i^\epsilon \coloneqq \nabla \times \psi^\epsilon_i$ is the interior part of the mollified velocity and it satisfies the required properties (in particular, it is divergence-free).

Now we consider the boundary parts $\psi_1, \ldots, \psi_m$. In particular, we need to define an extension for these functions in order to prove the mollification estimates. We introduce the notation
\begin{equation*}
(\psi_j)_{\tau} \coloneqq \psi_j - (\psi_j \cdot n).
\end{equation*}
We consider an even/odd higher order reflection $\widetilde{\psi}$ of the form (see for example \cite[~Lemma 6.37]{gilbarg} and \cite[~Theorem 4.26]{adams}, as well as \cite{babich})
\begin{align}
(\widetilde{\psi}_j \cdot n) \lvert_{\widetilde{V}_{\delta , U_j}} (\sigma_1, \sigma_2, s) &= \begin{cases}
&(\psi_j (\sigma_1, \sigma_2, s) \cdot n) \quad \text{ if } s \geq 0, \\
&6 (\psi_j ( \sigma_1, \sigma_2, -s) \cdot n) - 32 (\psi_j ( \sigma_1, \sigma_2, - s/2) \cdot n) \\
&+ 27 (\psi_j ( \sigma_1, \sigma_2, - s/3) \cdot n) \quad  \text{ if } s \leq 0. \\
\end{cases} \\
(\widetilde{\psi}_j)_{\tau} \lvert_{\widetilde{V}_{\delta , U_j}} (\sigma_1, \sigma_2, s) &= \begin{cases}
&(\psi_j)_{\tau} (\sigma_1, \sigma_2, s) \text{ if } s \geq 0, \\
&-(\psi_j)_{\tau} ( \sigma_1, \sigma_2, -s) \text{ if } s \leq 0. \\
\end{cases}
\end{align}
Recall that we assume the normal to point outward. One can check that $\widetilde{\psi}_j \in C^{1, \alpha}_c (\widetilde{V}_{\delta , U_j})$ and it can be extended by zero outside $\widetilde{V}_{\delta , U_j}$, moreover we have $(\widetilde{\psi}_j)_{\tau} (\sigma_1, \sigma_2 , 0) = 0$ (by Proposition \ref{streamfunction}). We then observe that $\widetilde{\psi}_j^\epsilon \in C^\infty_c (\widetilde{V}_{\delta , U_j})$. The odd extension of $(\widetilde{\psi}_j)_{\tau}$ ensures that $(\widetilde{\psi}_j^\epsilon )_{\tau} (\sigma_1 , \sigma_2, 0) = 0$.

We now define the function
\begin{equation}
\widetilde{\psi}^\epsilon \coloneqq \psi_i^\epsilon + \sum_{j=1}^m \widetilde{\psi}_j^\epsilon.
\end{equation}
We prove that this function converges to $\psi$ in the $C^1 (\overline{\Omega})$ norm as $\epsilon \rightarrow 0$. It is easy to see that
\begin{align*}
\lVert \psi - \widetilde{\psi }^\epsilon \rVert_{C^1 (\overline{\Omega})} &\leq \lVert (1 - \phi_1) \psi - \psi_i^\epsilon \rVert_{C^1 (\overline{\Omega})} + \bigg\lVert \phi_1 \psi - \sum_{j=1}^m \widetilde{\psi}_j^\epsilon \bigg\rVert_{C^1 (\overline{\Omega})} \\
&\leq \lVert \psi_i - \psi_i^\epsilon \rVert_{C^1 (\Omega )} + \sum_{j=1}^m \lVert \phi_1 \rho_j \psi - \widetilde{\psi}_j^\epsilon \rVert_{C^1 (\overline{V}_{\delta,U_j})} \xrightarrow[]{\epsilon \rightarrow 0} 0.
\end{align*}
In addition, it holds that $\widetilde{\psi}^\epsilon \rightarrow \psi$ in $C^{1,\beta} (\Omega)$ for $\beta \in (0, \alpha)$ and $\lVert \widetilde{\psi}^\epsilon \rVert_{C^{1,\alpha} (\Omega)} \leq C \lVert \psi \rVert_{C^{1,\alpha} (\Omega)}$ for some constant $C$. Now we take $\widetilde{u}^\epsilon \coloneqq \nabla \times \widetilde{\psi}^\epsilon \in C^\infty_c (\mathbb{R}^3)$, which satisfies the divergence-free condition. We also get that $\Tilde{u}^\epsilon \rightarrow u $ in $C^{0,\beta} (\overline{\Omega})$ for $\beta \in (0,\alpha)$ and moreover, it holds that (using Proposition \ref{streamfunction})
\begin{equation*}
\lVert \Tilde{u}^\epsilon \rVert_{C^{0,\alpha} (\Omega)} \leq C \lVert \widetilde{\psi}^\epsilon \rVert_{C^{1,\alpha} (\Omega)} \leq C \lVert \psi \rVert_{C^{1,\alpha} (\Omega)} \leq C \lVert u \rVert_{C^{0,\alpha} (\Omega)}.
\end{equation*}
Since $(\widetilde{\psi}^\epsilon )_{\tau} \lvert_{\partial \Omega} = 0$, simple calculations using the Stokes theorem show that $(\nabla \times \widetilde{\psi}^\epsilon ) \cdot n \lvert_{\partial \Omega}= (u^\epsilon \cdot n) \lvert_{\partial \Omega} = 0$.  
\end{proof}
\begin{remark}
We remark that this result also holds if the boundary $\partial \Omega$ is $C^2$ instead of $C^4$. Lemma \ref{mollificationlemma} therefore resolves a problem raised in \cite[~Sections 1.3 and 2.3]{derosa}, namely whether one can obtain smooth approximations of the vector field for domains with $C^2$ boundary. The proof in Lemma \ref{mollificationlemma} explicitly relies on the vector stream function, instead of the Schauder estimates which are used in \cite{derosa} (which are the reason for the slightly stronger requirements on the regularity of $\partial \Omega$).
\end{remark}
\begin{corollary}
Suppose $u^\epsilon \in C^{2, \alpha} (\Omega)$, then by standard elliptic theory there exists a unique function $p^\epsilon \in C^{2, \alpha} (\Omega)$ such that 
\begin{align}
&-\Delta p^\epsilon = (\nabla \otimes \nabla) : (u^\epsilon \otimes u^\epsilon) \quad \text{ in } \Omega, \\
&\partial_n p^\epsilon = u^\epsilon \otimes u^\epsilon : \nabla n, \quad \text{on } \partial \Omega, \label{mollifiedboundarycondition} \\
&\int_\Omega p^\epsilon dx = 0. 
\end{align}
\end{corollary}
The whole point of mollifying the velocity field is that it allows to find a candidate mollified pressure $p^\epsilon$ by using standard Schauder theory. Since the mollified candidate pressure lies in $C^{2, \alpha} (\Omega)$, it allows us to do many estimates more easily after which we can take the limit $\epsilon \rightarrow 0$. As has been noted in Remark \ref{boundaryregularityremark}, this approach requires the boundary to be sufficiently regular (we take $\partial \Omega$ to be $C^4$ in this paper). In particular, in order to be able to apply the standard Schauder theory, the boundary datum $u^\epsilon \otimes u^\epsilon : \nabla n$ needs to be sufficiently regular.
\section{Interior regularity estimate} \label{interiorsection}
We will derive estimates for the pressure separately in the interior of the domain and in the region near the boundary. We first establish the interior estimate. We now introduce several cutoff functions.
In order to separate the behaviour of the pressure near and far away from the boundary, we recall that in equation \eqref{phifunction} we defined the function $\phi$.
We consider the parameters $\delta_1, \delta_2$ and $\delta_3$ such that (for some small $\gamma < \delta$)
\begin{equation*}
0 < \delta_1 < \delta_2 - \gamma < \delta_3 < \delta - 2 \gamma.
\end{equation*}
We then introduce functions 
\begin{equation*}
\phi_i (s) = \begin{cases}
0 \text{ if } 0 \leq s \leq \delta_1, \\
1 \text{ if } s \geq \delta_2 - \gamma,
\end{cases} \quad \phi_b (s) = \begin{cases}
1 \text{ if } 0 \leq s < \delta_3 + \gamma, \\
0 \text{ if } s \geq \delta - \gamma.
\end{cases}
\end{equation*}
Observe that $\phi$ and $\phi_b$ are nonincreasing while $\phi_i$ is nondecreasing and that $\phi_i$ and $\phi_b$ are overlapping (and they are not a partition of unity). We will write (for $x \in \Omega$ sufficiently close to the boundary)
\begin{equation*}
\phi_i (x) \coloneqq \phi_i (d(x, \partial \Omega)), \quad \phi_b (x) \coloneqq \phi_b (d(x, \partial \Omega)).
\end{equation*}
Then we introduce the functions $\phi_{b,1}, \ldots, \phi_{b,m}$ through the definition
\begin{equation*}
\phi_{b,j} \coloneqq \phi_b \rho_j.
\end{equation*}
We then define the functions
\begingroup
\allowdisplaybreaks
\begin{align}
P^\epsilon (x) &= p^\epsilon (x) + \phi (x) (u^\epsilon (x) \cdot n(x))^2, \label{mollifiedpressure1} \\
P^\epsilon_i (x) &= \phi_i (x) P^\epsilon (x) = \phi_i (x) ( p^\epsilon (x) + \phi (x) (u^\epsilon (x) \cdot n(x))^2 ), \label{mollifiedpressure2} \\
P^\epsilon_{b} (x) &= \phi_b (x) P^\epsilon (x) = \phi_b (x) ( p^\epsilon (x) + (u^\epsilon (x) \cdot n(x))^2 ), \\
P^\epsilon_{b,j} (x) &= \rho_j (x) \phi_b (x) P^\epsilon (x) = \rho_j (x) \phi_b (x) ( p^\epsilon (x) + (u^\epsilon (x) \cdot n(x))^2 ). \label{mollifiedpressure3}
\end{align}
\endgroup
Note that we used that $\phi_b (x) \phi (x) = \phi_b (x)$ which holds by definition of the cutoff functions. 

We first prove an estimate for the interior pressure $P^\epsilon_i$. 
\begin{proposition} \label{interiorestimate}
Let $P^\epsilon_i$ be the function defined in equation \eqref{mollifiedpressure2}. The following estimate holds for the interior mollified pressure
\begin{equation} \label{interiorestimateeq}
\lVert P^\epsilon_i \rVert_{C^{0, \alpha} (\Omega)} \leq C_i \lVert u^\epsilon \otimes u^\epsilon \rVert_{C^{0, \alpha} (\Omega)} + D_i \lVert P^\epsilon \rVert_{L^\infty (\Omega)}.
\end{equation}
Note that the constants $C_i$ and $D_i$ are independent of $\epsilon$.
\end{proposition}
\begin{proof}
We calculate that $P^\epsilon_i$ satisfies the equation
\begin{align*}
-\Delta P^\epsilon_i &= - (\Delta \phi_i) P^\epsilon - 2  (\nabla \phi_i) \cdot \nabla P^\epsilon - \phi_i \Delta P^\epsilon \\
&= - (\Delta \phi_i) P^\epsilon - 2  (\nabla \phi_i) \cdot \nabla P^\epsilon - \phi_i (\Delta (\phi (u^\epsilon \cdot n)^2) + \Delta p^\epsilon) \\
&= - (\Delta \phi_i) P^\epsilon - 2  (\nabla \phi_i) \cdot \nabla P^\epsilon - \phi_i \big[ \Delta (\phi (u^\epsilon \cdot n)^2) - (\nabla \otimes \nabla) : (u^\epsilon \otimes u^\epsilon) \big].
\end{align*}
We then decompose the interior pressure as $P^\epsilon_{i,1}$ and $P^\epsilon_{i,2}$ which satisfy that
\begin{align*}
&\begin{cases} -\Delta P^\epsilon_{i,1} = - (\Delta \phi_i) P^\epsilon - 2  (\nabla \phi_i) \cdot \nabla P^\epsilon \quad \text{on } \mathbb{R}^3 , 
\end{cases} \\
&\begin{cases}
-\Delta P^\epsilon_{i,2} = - \phi_i \big[ \Delta (\phi (u^\epsilon \cdot n)^2) - (\nabla \otimes \nabla) : (u^\epsilon \otimes u^\epsilon) \big] \quad \text{on } \Omega \backslash V_{\delta_1},  \\
P^\epsilon_{i,2} = - P^\epsilon_{i,1} \quad \text{on } \partial (\Omega \backslash V_{\delta_1}).
\end{cases}
\end{align*}
Note that we can homogenise the boundary condition for $P^\epsilon_{i,2}$ by changing the forcing. We establish estimate \eqref{interiorestimateeq} separately for $P^\epsilon_{i,1}$ and $P^\epsilon_{i,2}$. We recall that the Green's function of the operator $-\Delta$ in $\mathbb{R}^3$ is given by
\begin{equation*}
G(x) \coloneqq \frac{1}{4 \pi \lvert x \rvert}.
\end{equation*}
This means that $P^\epsilon_{i,1}$ is given by
\begin{align*}
P^\epsilon_{i,1} &= \frac{1}{4 \pi \lvert x \rvert} * (- (\Delta \phi_i) P^\epsilon - 2  (\nabla \phi_i) \cdot \nabla P^\epsilon ).
\end{align*}
which allows us to conclude estimate \eqref{interiorestimateeq} for $P^\epsilon_{i,1}$, by proceeding similarly as in \cite[~Theorem 13.1.1]{jostPDE}.

By the Schauder estimate for the Dirichlet problem established in Theorem \ref{dirichletschauderestimate} we know that $P^\epsilon_{i,2}$ also satisfies estimate \eqref{interiorestimateeq}, which concludes the proof. 
\end{proof}
Now we move on to establishing the estimates for the boundary layer pressure $P^\epsilon_{b}$.
\section{Trace lemma} \label{tracesection}
It follows that $P_{b,j}^\epsilon$ satisfies the following equation
\begingroup
\allowdisplaybreaks
\begin{equation} \label{boundaryequation}
-\Delta P^\epsilon_{b,j} = - \Delta (\phi_b \rho_j) P^\epsilon - 2  \nabla (\phi_b \rho_j) \cdot \nabla P^\epsilon + \phi_b \rho_j \bigg((\nabla \otimes \nabla) : (u^\epsilon \otimes u^\epsilon) - \Delta \big( (u^\epsilon \cdot n)^2 \big)  \bigg).
\end{equation}
\endgroup
Now we need to write this equation in terms of the local coordinate system \eqref{transform1}-\eqref{transform3} in the region $\overline{V}_{\delta, U_j}$.
\begin{lemma}
Assume that $x \in \overline{V}_{\delta, U_j}$, then we have the following expression of equation \eqref{boundaryequation} (in terms of the local coordinates $(\sigma_1, \sigma_2, s)$ defined in equations \eqref{transform1}-\eqref{transform3})
\begin{align}
&-\Delta_\tau P^\epsilon_{b,j} - \frac{1}{b} \frac{\partial b}{\partial s} \frac{\partial P^\epsilon_{b,j}}{\partial s} - \frac{\partial^2 P^\epsilon_{b,j}}{\partial s^2} = - \Delta (\phi_b \rho_j) P^\epsilon - 2 \nabla_\tau (\phi_b \rho_j) \cdot \nabla_\tau P^\epsilon - 2 \partial_s (\phi_b \rho_j) \frac{\partial P^\epsilon}{\partial s} \nonumber \\
&+ \phi_b \rho_j \bigg( \frac{1}{b} \bigg[ \sum_{i,j=1}^2 \frac{\partial^2}{\partial \sigma_i \partial \sigma_j} \big( b u^\epsilon_i u^\epsilon_j \big) + 2 \sum_{i=1}^2 \frac{\partial^2}{\partial \sigma_i \partial s} \big( b u^\epsilon_i (u^\epsilon \cdot n) \big) + \frac{\partial^2}{\partial s^2} \big( b (u^\epsilon \cdot n)^2 \big) \bigg] \nonumber \\
&- \Delta_\tau \big( (u^\epsilon \cdot n)^2 \big) - \frac{1}{b} \frac{\partial b}{\partial s} \frac{\partial ( (u^\epsilon \cdot n)^2)}{\partial s} - \frac{\partial^2 ( (u^\epsilon \cdot n)^2)}{\partial s^2} \bigg), \label{boundaryeqcoords}
\end{align}
where the differential operator $\nabla_\tau$ has been introduced in equation \eqref{gradcoord}.
\end{lemma}
\begin{proof}
In the coordinate system \eqref{transform1}-\eqref{transform3} we can write the divergence as follows (by using equation \eqref{divnewcoordinates})
\begingroup
\allowdisplaybreaks
\begin{align} \label{doubledivergence}
&(\nabla \otimes \nabla) : (u^\epsilon \otimes u^\epsilon) = \frac{1}{b} \bigg[ \sum_{i,j=1}^2 \frac{\partial^2}{\partial \sigma_i \partial \sigma_j} \big( b u^\epsilon_i u^\epsilon_j \big) + 2 \sum_{i=1}^2 \frac{\partial^2}{\partial \sigma_i \partial s} \big( b u^\epsilon_i (u^\epsilon \cdot n) \big) + \frac{\partial^2}{\partial s^2} \big( b (u^\epsilon \cdot n)^2 \big) \bigg].
\end{align}
\endgroup
Moreover, we compute that (by using equation \eqref{laplaciancoord})
\begin{align*}
\Delta P^\epsilon_{b,j} &= \Delta_\tau P^\epsilon_{b,j} + \frac{1}{b} \frac{\partial b}{\partial s} \frac{\partial P^\epsilon_{b,j}}{\partial s} + \frac{\partial^2 P^\epsilon_{b,j}}{\partial s^2}, \\
\Delta \big( (u^\epsilon \cdot n)^2 \big) &= \Delta_\tau \big( (u^\epsilon \cdot n)^2 \big) + \frac{1}{b} \frac{\partial b}{\partial s} \frac{\partial ( (u^\epsilon \cdot n)^2)}{\partial s} + \frac{\partial^2 ( (u^\epsilon \cdot n)^2)}{\partial s^2}. 
\end{align*}
We can now compute the gradient of the pressure to be (using equation \eqref{gradcoord})
\begin{align*}
&\nabla P^\epsilon = \frac{\partial P^\epsilon}{\partial \sigma_1} \big( a_{11} \tau_1 + a_{21} \tau_2  \big) + \frac{\partial P^\epsilon}{\partial \sigma_2} \big( a_{12} \tau_1 + a_{22} \tau_2  \big) + \frac{\partial P^\epsilon}{\partial s}  n = \nabla_\tau P^\epsilon + \frac{\partial P^\epsilon}{\partial s}  n,
\end{align*}
and similarly we can compute the gradient of $\phi_b \rho_j$. Then by using the orthogonality of $n$ with respect to $\tau_1$ and $\tau_2$, we find that
\begin{equation*}
\nabla (\phi_b \rho_j) \cdot \nabla P^\epsilon = \nabla_\tau (\phi_b \rho_j) \cdot \nabla_\tau P^\epsilon + \partial_s (\phi_b \rho_j) \frac{\partial P^\epsilon}{\partial s}. 
\end{equation*}
These calculations allow us to express equation \eqref{boundaryequation} as follows
\begin{align*}
&-\Delta_\tau P^\epsilon_{b,j} - \frac{1}{b} \frac{\partial b}{\partial s} \frac{\partial P^\epsilon_{b,j}}{\partial s} - \frac{\partial^2 P^\epsilon_{b,j}}{\partial s^2} = - \Delta (\phi_b \rho_j) P^\epsilon - 2 \nabla_\tau (\phi_b \rho_j) \cdot \nabla_\tau P^\epsilon - 2 \partial_s (\phi_b \rho_j) \frac{\partial P^\epsilon}{\partial s} \\
&+ \phi_b \rho_j \bigg( \frac{1}{b} \bigg[ \sum_{i,j=1}^2 \frac{\partial^2}{\partial \sigma_i \partial \sigma_j} \big( b u^\epsilon_i u^\epsilon_j \big) + 2 \sum_{i=1}^2 \frac{\partial^2}{\partial \sigma_i \partial s} \big( b u^\epsilon_i (u^\epsilon \cdot n) \big) + \frac{\partial^2}{\partial s^2} \big( b (u^\epsilon \cdot n)^2 \big) \bigg] \\
&- \Delta_\tau \big( (u^\epsilon \cdot n)^2 \big) - \frac{1}{b} \frac{\partial b}{\partial s} \frac{\partial ( (u^\epsilon \cdot n)^2)}{\partial s} - \frac{\partial^2 ((u^\epsilon \cdot n)^2)}{\partial s^2} \bigg).
\end{align*}
\end{proof}
It is crucial to observe that the definition of the mollified modified pressure $P^\epsilon$ in equation \eqref{mollifiedpressure1} (i.e. combining the term $(u^\epsilon \cdot n)^2$ as part of $P^\epsilon$) has the consequence that on the right-hand side of equation \eqref{boundaryeqcoords} there are no terms which have second order derivatives in $s$ (after some further manipulations, which will be done in the proof of the next lemma). This makes it possible to establish the following trace lemma, which will be crucial for the proof of the final regularity result (Theorem \ref{pressureregularity}) and proving that the solution obtained attains the very weak boundary condition \eqref{veryweakboundcondeq}. 
\begin{lemma}[Trace lemma] \label{tracelemma}
The following equation holds for $\partial_s P^\epsilon_b$ for every region $\overline{V}_{\delta, U_j}$
\begin{equation} \label{traceequation}
\partial_{s} P^\epsilon_{b,j} (\cdot, \cdot, s) = \Lambda^\epsilon_j ( \cdot, \cdot, s) + \int_{s}^\delta \Theta^\epsilon_j (\cdot, \cdot, s') d s',
\end{equation}
where $\Lambda^\epsilon$ and $\Theta^\epsilon$ (which are specified in the proof) satisfy the local estimates
\begingroup
\allowdisplaybreaks
\begin{align}
\lVert \Lambda^\epsilon_j \rVert_{C^{0, \alpha} ( [0, \delta], H^{-1} (U_j)) } &\leq C_b \lVert u^\epsilon \otimes u^\epsilon \rVert_{C^{0, \alpha} (\Omega)} + D_b \lVert P^\epsilon \rVert_{L^\infty (\Omega)}, \label{traceestimate1} \\
\lVert \Theta^\epsilon \rVert_{C^{0, \alpha} ([0, \delta], H^{-2} (U_j))} &\leq C_b \lVert u^\epsilon \otimes u^\epsilon \rVert_{C^{0, \alpha} (\Omega)} + D_b \lVert P^\epsilon \rVert_{L^\infty (\Omega)}. \label{traceestimate2}
\end{align}
\endgroup
These estimates can then be put together to yield a global estimate for $\Lambda^\epsilon$ and $\Theta^\epsilon$ for the region near the boundary.
\end{lemma}
\begin{proof}
The proof will be done locally, i.e. for a given patch $V_{\delta,U_j} \subset V_\delta$, which can then be extended to the whole region near the boundary by using the partition of unity for $U_1, \ldots, U_m$. We start by rewriting equation \eqref{boundaryeqcoords} as follows
\begingroup
\allowdisplaybreaks
\begin{align*}
& - \frac{\partial^2 P^\epsilon_{b,j}}{\partial s^2} = \Delta_\tau P^\epsilon_{b,j} + \frac{1}{b} \frac{\partial b}{\partial s} \frac{\partial P^\epsilon_{b,j}}{\partial s} - \Delta (\phi_b \rho_j) P^\epsilon - 2 \nabla_\tau (\phi_b \rho_j) \cdot \nabla_\tau P^\epsilon - 2 \partial_s (\phi_b \rho_j) \frac{\partial P^\epsilon}{\partial s}  \\
&+ \phi_b \rho_j \bigg( \frac{1}{b} \bigg[ \sum_{i,j=1}^2 \frac{\partial^2}{\partial \sigma_i \partial \sigma_j} \big( b u^\epsilon_i u^\epsilon_j \big) + 2 \sum_{i=1}^2 \frac{\partial^2}{\partial \sigma_i \partial s} \big( b u^\epsilon_i (u^\epsilon \cdot n) \big) + \frac{\partial^2 b}{\partial s^2}  (u^\epsilon \cdot n)^2 + \frac{\partial b}{\partial s} \frac{\partial ((u^\epsilon \cdot n)^2)}{\partial s}  \bigg]  \\
&- \Delta_\tau \big( (u^\epsilon \cdot n)^2 \big)  \bigg).
\end{align*}
\endgroup
By integrating the above over the interval $[s, \delta]$ and integrating by parts, we find that (using that $\rho_j$ has compact support)
\begingroup
\allowdisplaybreaks
\begin{align*}
&\partial_s P^\epsilon_{b,j} = -\frac{1}{b} \frac{\partial b}{\partial s} P^\epsilon_{b,j} + 2 \partial_s (\phi_b \rho_j) P^\epsilon - \frac{2 \phi_b \rho_j}{b} \sum_{i=1}^2 \frac{\partial}{\partial \sigma_i} \big( b u^\epsilon_i (u^\epsilon \cdot n) \big) - \frac{\phi_b \rho_j}{b} \partial_s b (u^\epsilon \cdot n)^2  \\
&+\int_s^\delta \bigg[ \Delta_\tau P^\epsilon_{b,j} - \Delta (\phi_b \rho_j) P^\epsilon - 2 \nabla_\tau (\phi_b \rho_j) \cdot \nabla_\tau P^\epsilon + \phi_b \rho_j \bigg( \frac{1}{b} \bigg[ \sum_{i,j=1}^2 \frac{\partial^2}{\partial \sigma_i \partial \sigma_j} \big( b u^\epsilon_i u^\epsilon_j \big)  + \frac{\partial^2 b}{\partial s^2}  (u^\epsilon \cdot n)^2  \bigg]  \\
&- \Delta_\tau \big( (u^\epsilon \cdot n)^2 \big)  \bigg) - \frac{\partial}{\partial s} \bigg( \frac{1}{b} \frac{\partial b}{\partial s} \bigg) P^\epsilon_{b,j} + 2 \partial_s^2 (\phi_b \rho_j) P^\epsilon - 2 \frac{\partial}{\partial s} \bigg( \frac{2 \phi_b \rho_j}{b} \bigg) \sum_{i=1}^2 \frac{\partial}{\partial \sigma_i} \big( b u^\epsilon_i (u^\epsilon \cdot n) \big) \\
&- \frac{\partial}{\partial s} \bigg( \frac{\phi_b \rho_j}{b} \partial_s b \bigg) (u^\epsilon \cdot n)^2 \bigg] ds'. 
\end{align*}
\endgroup
Now we introduce the notation
\begin{align*}
\Lambda^\epsilon &\coloneqq -\frac{1}{b} \frac{\partial b}{\partial s} P^\epsilon_{b,j} + 2 \partial_s (\phi_b \rho_j) P^\epsilon - \frac{2 \phi_b \rho_j}{b} \sum_{i=1}^2 \frac{\partial}{\partial \sigma_i} \big( b u^\epsilon_i (u^\epsilon \cdot n) \big) - \frac{\phi_b \rho_j}{b} \partial_s b (u^\epsilon \cdot n)^2, \\
\Theta^\epsilon &\coloneqq \Delta_\tau P^\epsilon_{b,j} - \Delta (\phi_b \rho_j) P^\epsilon - 2 \nabla_\tau (\phi_b \rho_j) \cdot \nabla_\tau P^\epsilon + \phi_b \rho_j \bigg( \frac{1}{b} \bigg[ \sum_{i,j=1}^2 \frac{\partial^2}{\partial \sigma_i \partial \sigma_j} \big( b u^\epsilon_i u^\epsilon_j \big)  + \frac{\partial^2 b}{\partial s^2}  (u^\epsilon \cdot n)^2  \bigg]  \\
&- \Delta_\tau \big( (u^\epsilon \cdot n)^2 \big)  \bigg) - \frac{\partial}{\partial s} \bigg( \frac{1}{b} \frac{\partial b}{\partial s} \bigg) P^\epsilon_{b,j} + 2 \partial_s^2 (\phi_b \rho_j) P^\epsilon - 2 \frac{\partial}{\partial s} \bigg( \frac{2 \phi_b \rho_j}{b} \bigg) \sum_{i=1}^2 \frac{\partial}{\partial \sigma_i} \big( b u^\epsilon_i (u^\epsilon \cdot n) \big) \\
&- \frac{\partial}{\partial s} \bigg( \frac{\phi_b \rho_j}{b} \partial_s b \bigg) (u^\epsilon \cdot n)^2.
\end{align*}
This yields equation \eqref{traceequation}.

Now we multiply the equation we derived above for $\partial_s P^\epsilon_{b,j}$ by a test function $\phi ( \sigma_1,\sigma_2) \in H^2 (U_j)$. We then integrate with respect to $\sigma_1$ and $\sigma_2$ once or twice, dependent on the term (such that the terms $P^\epsilon$ and $u^\epsilon \otimes u^\epsilon$ no longer have any derivatives). From this we obtain estimates \eqref{traceestimate1} and \eqref{traceestimate2} for $\Lambda^\epsilon$ and $\Theta^\epsilon$. Note that we implicitly use the (boundary) regularity of $P^\epsilon_{b,j}$ here, which we will establish in Proposition \ref{boundaryestimates}.

One can see that $\Lambda^\epsilon_j$ only contains first-order derivatives of $u^\epsilon$ with respect to $\sigma_1$ and $\sigma_2$, which means that $\Lambda^\epsilon (\cdot, \cdot, s) \in H^{-1} (U_j)$. The terms in $\Theta^\epsilon$ have at most second-order derivatives with respect to $\sigma_1$ and $\sigma_2$, so $\Theta^\epsilon$ lies in $H^{-2} (U_j)$.
\end{proof}
\section{Estimate of the boundary layer pressure} \label{decompositionsection}
We will now establish an estimate for the boundary layer pressure, analogous to the interior regularity estimate \eqref{interiorestimateeq}. As was calculated before, the local boundary layer pressure satisfies the problem
\begingroup
\allowdisplaybreaks
\begin{align*}
&-\Delta P^\epsilon_{b,j}  = - \Delta (\phi_b \rho_j) P^\epsilon - 2 \nabla_\tau (\phi_b \rho_j) \cdot \nabla_\tau P^\epsilon - 2 \partial_s (\phi_b \rho_j) \frac{\partial P^\epsilon}{\partial s} \nonumber \\
&+ \phi_b \rho_j \bigg( \frac{1}{b} \bigg[ \sum_{i,j=1}^2 \frac{\partial^2}{\partial \sigma_i \partial \sigma_j} \big( b u^\epsilon_i u^\epsilon_j \big) + 2 \sum_{i=1}^2 \frac{\partial^2}{\partial \sigma_i \partial s} \big( b u^\epsilon_i (u^\epsilon \cdot n) \big) + \frac{\partial^2}{\partial s^2} \big( b (u^\epsilon \cdot n)^2 \big) \bigg] \nonumber \\
&- \Delta_\tau \big( (u^\epsilon \cdot n)^2 \big) - \frac{1}{b} \frac{\partial b}{\partial s} \frac{\partial ((u^\epsilon \cdot n)^2)}{\partial s} - \frac{\partial^2 ((u^\epsilon \cdot n)^2)}{\partial s^2} \bigg) \text{ in } V_{\delta, U_j}, \\
&\partial_n P^{\epsilon}_{b,j} = \rho_j \big( u^\epsilon \otimes u^\epsilon : \nabla n \big), \text{ on } \partial \Omega \cap U_j, \; P^{\epsilon}_{b,j} = 0 \text{ on } \partial V_{\delta, U_j} \backslash \partial \Omega.
\end{align*}
\endgroup

The local boundary Schauder-type estimate for $P^{\epsilon}_{b,j}$ is given in the next proposition (for each region $V_{\delta,U_j}$), the local estimates can then be patched together to yield a global estimate for $P^{\epsilon}_{b}$ on $V_\delta$.
\begin{proposition} \label{boundaryestimates}
The boundary layer pressure satisfies the following local estimate
\begin{equation*}
\lVert P^{\epsilon}_{b,j} \rVert_{C^{0, \alpha} (V_{\delta,U_j}) } \leq C \lVert u^\epsilon \otimes u^\epsilon \rVert_{C^{0, \alpha} (V_{\delta,U_j})} + D \lVert P^{\epsilon}_{b,j} \rVert_{L^\infty (V_{\delta,U_j})},
\end{equation*}
where the nonnegative constants $C$ and $D$ are independent of $\epsilon$.
\end{proposition}
\begin{proof}
We first observe that there indeed exists a unique solution $P^{\epsilon}_{b,j}$ because of Theorem \ref{existencedirichletneumann} (see Appendix \ref{schauderappendix2}). The regularity estimate is derived in the proof of Theorem \ref{schauderdirichletneumann} (see Appendix \ref{schauderappendix2}).
\end{proof}

\section{Taking the limit $\epsilon \rightarrow 0$} \label{limitsection}
Now we see that by collecting the estimates from Propositions \ref{interiorestimate} and \ref{boundaryestimates} we find that
\begin{equation} \label{collectedestimates}
\lVert P^\epsilon \rVert_{C^{0, \alpha } (\Omega) } \leq \lVert P^\epsilon_i \rVert_{C^{0, \alpha } (\Omega) } + \lVert P^\epsilon_{b} \rVert_{C^{0, \alpha } (\Omega) } \leq C \lVert u^\epsilon \otimes u^\epsilon \rVert_{C^{0, \alpha} (\Omega) } + D \lVert P^\epsilon \rVert_{L^\infty (\Omega)}.
\end{equation}
We will now show that the inequality still holds without the term $D \lVert P^\epsilon \rVert_{L^\infty (\Omega)}$, up to an enlarging of the constant $C$.
\begin{proposition} \label{pressurebound}
The following estimate holds for $P^\epsilon$
\begin{equation} \label{mollificationlimit}
\lVert P^\epsilon \rVert_{L^\infty (\Omega)} \leq C \lVert u^\epsilon \otimes u^\epsilon \rVert_{C^{0, \alpha} (\Omega)}. 
\end{equation}
Once again, the constant $C$ does not depend on $\epsilon$.
\end{proposition}
\begin{proof}
We argue by contradiction. If the inequality does not hold, there exists a subsequence (which we still call $\{ P^\epsilon \}$) such that
\begin{equation} \label{falseassumption}
\lim_{\epsilon \rightarrow 0} \frac{\lVert u^\epsilon \otimes u^\epsilon \rVert_{C^{0, \alpha} (\Omega)}}{\lVert P^\epsilon \rVert_{L^\infty (\Omega)}} = 0.
\end{equation}
Now we introduce the following functions
\begin{equation*}
\mathcal{G}^\epsilon \coloneqq \frac{P^\epsilon}{\lVert P^\epsilon \rVert_{L^\infty (\Omega)}}.
\end{equation*}
These functions solve the following boundary-value problem
\begingroup
\allowdisplaybreaks
\begin{align}
- \Delta \mathcal{G}^\epsilon &= \frac{1}{\lVert P^\epsilon \rVert_{L^\infty  (\Omega)} } \bigg[ (\nabla \otimes \nabla ) : (u^\epsilon \otimes u^\epsilon) - \Delta ( \phi (u^\epsilon \cdot n)^2 ) \bigg] \quad \text{in } \Omega, \label{HBVP1} \\
&\partial_n \mathcal{G}^\epsilon = \frac{1}{\lVert P^\epsilon \rVert_{L^\infty (\Omega)}} \big( u^\epsilon \otimes u^\epsilon : \nabla n \big) \quad \text{on } \partial \Omega, \label{HBVP2} \\
&\int_\Omega \mathcal{G}^\epsilon (x) dx = \frac{1}{\lVert P^\epsilon \rVert_{L^\infty (\Omega)}} \int_\Omega \phi (x) (u^\epsilon \cdot n)^2 dx. \label{HBVP3}
\end{align}
\endgroup
The sequence $\lVert \mathcal{G}^\epsilon \rVert_{C^{0, \alpha} (\Omega)}$ is bounded by inequality \eqref{collectedestimates}. By using the Arzelà-Ascoli theorem we know that there exists a subsequence, for which we also write $\mathcal{G}^\epsilon$, converging strongly to a given function $\mathcal{G}$ in $C^0 (\overline{\Omega})$. Note that it also converges in any H\"older space with exponent less than $\alpha$, which can be seen by using an interpolation inequality.

By assumption we know that $\lVert \mathcal{G}^\epsilon \rVert_{L^\infty (\Omega)} = \lVert \mathcal{G} \rVert_{L^\infty (\Omega) } = 1$. It follows by equation \eqref{falseassumption} that the right-hand sides of equations \eqref{HBVP1}, \eqref{HBVP2} and \eqref{HBVP3} of the boundary-value problem for $\mathcal{G}^\epsilon$ all go to zero as $\epsilon \rightarrow 0$ in the space $C^{0, \beta} (\Omega)$ for $\beta \in [0, \alpha)$. This means that $\mathcal{G}$ satisfies the equation
\begin{equation*}
- \Delta \mathcal{G} = 0 \quad \text{in } \Omega.
\end{equation*}

Next we show that $\partial_n \mathcal{G}$ is well-defined. We know that the sequence $\{ \partial_n \mathcal{G}^\epsilon \lvert_{\partial \Omega} \}$ converges to zero, but we need to show that the limit coincides with $\partial_n \mathcal{G} \lvert_{\partial \Omega}$. Using the trace lemma (Lemma \ref{tracelemma}), we know that near the boundary
\begin{equation} \label{mollifiedtrace}
\partial_{s} \mathcal{G}^\epsilon_{b,j} (\cdot, \cdot, s) = \frac{\Lambda^\epsilon_j ( \cdot, \cdot, s)}{\lVert P^\epsilon \rVert_{L^\infty}} + \int_{s}^\delta \frac{\Theta^\epsilon_j (\cdot, \cdot, s')}{\lVert P^\epsilon \rVert_{L^\infty}} d s'.
\end{equation}
We first observe that the map $s \mapsto \partial_{s} \mathcal{G}^\epsilon_{b,j} (\cdot, \cdot, s)$ is a map from $[0,\delta)$ to $H^{-2} (U_j)$. By using estimates \eqref{traceestimate1} and \eqref{traceestimate2}, we find that
\begin{align*}
\bigg\lVert \frac{\Lambda^\epsilon}{\lVert P^\epsilon \rVert_{L^\infty}} \bigg\rVert_{C^{0,\alpha} ([0, \delta); H^{-1} (U_j))} &\leq C_b \bigg\lVert \frac{u^\epsilon \otimes u^\epsilon }{\lVert P^\epsilon \rVert_{L^\infty}} \bigg\rVert_{C^{0,\alpha} (\Omega)} + D_b, \\
\bigg\lVert \frac{\Theta^\epsilon}{\lVert P^\epsilon \rVert_{L^\infty}} \bigg\rVert_{C^{0,\alpha} ([0, \delta); H^{-2} (U_j))} &\leq C_b \bigg\lVert \frac{u^\epsilon \otimes u^\epsilon }{\lVert P^\epsilon \rVert_{L^\infty}} \bigg\rVert_{C^{0,\alpha} (\Omega)} + D_b,
\end{align*}
which means that the sequence $\{ \partial_s \mathcal{G}^\epsilon_{b,j} \}$ is equicontinuous in $s$. We next show that for every $s \in [0,\delta]$, the sequence $\{ \partial_s \mathcal{G}^\epsilon (\cdot, \cdot, s) \}$ has a convergent subsequence. We have that the sequence $\bigg\{ \frac{\Lambda^\epsilon (\cdot, \cdot, s)}{\lVert P^\epsilon \rVert_{L^\infty}} \bigg\}$ is bounded uniformly in $\epsilon$ in $H^{-1} (U_j)$ for fixed $s$. Due to the compact embedding of $H^{-1} (U_j)$ into $H^{-2} (U_j)$ we have a strongly convergent subsequence in $H^{-2} (U_j)$. 

Now by examining the expression for $\partial_s P^\epsilon_{b,j}$ in the proof of Lemma \ref{tracelemma}, it can be seen that estimate \eqref{traceestimate2} can be improved to 
\begin{equation*}
\bigg\lVert \frac{\Theta^\epsilon}{\lVert P^\epsilon \rVert_{L^\infty}} \bigg\rVert_{C^{0,\alpha} ([0, \delta); H^{-2+\alpha} (U_j)]} \leq C_b \bigg\lVert \frac{u^\epsilon \otimes u^\epsilon }{\lVert P^\epsilon \rVert_{L^\infty}} \bigg\rVert_{C^{0,\alpha} (\Omega)} + D_b,
\end{equation*}
where we have used Proposition \ref{boundaryestimates} in bounding the term $\Delta_\tau P^\epsilon_{b,j}$ in the $H^{-2+\alpha} (U_j)$ norm (with respect to the variables $\sigma_1$ and $\sigma_2$).

This allows us to conclude that the sequence $\bigg\{ \frac{\Theta^\epsilon (\cdot, \cdot, s)}{\lVert P^\epsilon \rVert_{L^\infty}} \bigg\}$ has a convergent subsequence in $H^{-2} (U_j)$ for fixed $s \in [0, \delta]$. Therefore the same holds for the sequence $\bigg\{ \partial_s \mathcal{G}^\epsilon (\cdot, \cdot, s) \bigg\}$. By the Arzela-Ascoli theorem we therefore conclude that the sequence $\{ \partial_s \mathcal{G}^\epsilon \}$ has a convergent subsequence in $C^0 ([0,\delta]; H^{-2} (U_j) )$, we will refer to the limit as $\mathcal{H} \in C^{0,\alpha} ([0,\delta]; H^{-2} (U_j) )$. It is clear that $\mathcal{H}$ coincides with $\partial_s \mathcal{G}$, as one can integrate equation \eqref{mollifiedtrace} with respect to $s$ and take the limit $\epsilon \rightarrow 0$.
We conclude that (as $\partial_s \mathcal{G}$ is continuous in $s$)
\begin{equation}
\partial_s \mathcal{G}(\cdot, \cdot, 0) = 0 \quad \text{ in } H^{-2} ( \partial \Omega).
\end{equation}
Hence $\mathcal{G}$ satisfies the following boundary-value problem
\begin{align}
&- \Delta \mathcal{G} = 0, \quad \text{ in } \Omega \\
&\partial_n \mathcal{G} \lvert_{\partial \Omega} = 0, \quad \int_\Omega \mathcal{G}(x) dx = 0. 
\end{align}
The only solution to this boundary-value problem is $\mathcal{G} = 0$, this is in contradiction with the assumption $\lVert \mathcal{G} \rVert_{L^\infty (\Omega)} = 1$. Therefore inequality \eqref{mollificationlimit} must hold.
\end{proof}
\begin{remark}
We are providing the full details of the proof of Proposition \ref{pressurebound}, especially regarding the convergence of $\{ \partial_s \mathcal{G}^\epsilon \}$, as this nontrivial point was missing in the proof of Proposition 3.11 in \cite{bardos2021}. In particular, equation (3.41) as given in \cite{bardos2021} is not correct. The statement itself of Proposition 3.11 in \cite{bardos2021} is correct and the given proof can be adapted by using the method outlined above. 
\end{remark}
Finally we are able to prove Theorem \ref{pressureregularity}.
\begin{proof}[Proof of Theorem \ref{pressureregularity}] 
By Lemma \ref{mollificationlemma} we know that
\begin{equation*}
\lVert u^\epsilon \otimes u^\epsilon \rVert_{C^{0, \alpha} (\Omega)} \leq \lVert u \otimes u \rVert_{C^{0, \alpha} (\Omega)}.
\end{equation*}
Moreover, we know that $u^\epsilon \otimes u^\epsilon$ converges to $u \otimes u$ in the $C^{0, \beta} (\Omega)$ norm for any $\beta \in [0, \alpha)$. This follows from standard mollification estimates and interpolation inequalities.

Then by combining inequalities \eqref{collectedestimates} and \eqref{mollificationlimit}, we find that $\lVert P^\epsilon \rVert_{C^{0, \alpha} (\Omega)}$ is bounded. This means that we are able to take a subsequence, which we also denote by $P^\epsilon$, which converges in the $C^0 (\overline{\Omega})$ norm to the limit $P \in C^{0, \alpha} (\Omega)$. It is clear that one can pass to the limit in the weak formulation from Definition \ref{weakformulationdef}.

By using Lemma \ref{tracelemma}, we find that
\begin{align}
P(\sigma_1, \sigma_2, s) - P(\sigma_1, \sigma_2, 0) &= \lim_{\epsilon \rightarrow 0} (P^{\epsilon} (\sigma_1, \sigma_2, s) - P^\epsilon(\sigma_1,\sigma_2,0) ) \nonumber \\
&= \int_0^{s} \sum_{j=1}^m \lim_{\epsilon \rightarrow 0} \bigg( \Lambda^\epsilon_j ( \sigma_1, \sigma_2, s') + \int_{s'}^\delta \Theta^\epsilon_j (\sigma_1, \sigma_2, s'') ds'' \bigg) ds'. \label{H-2regularity}
\end{align}
By estimates \eqref{traceestimate1} and \eqref{traceestimate2}, we know that the limits $\lim_{\epsilon \rightarrow 0} \Lambda^\epsilon$ and $\lim_{\epsilon \rightarrow 0} \Theta^\epsilon$ exist as elements of $C^{0,\alpha} ([0,\delta] ; H^{-2} (\partial \Omega))$. This implies that $P \in C^{1,\alpha} ([0, \delta]; H^{-2} ( \partial \Omega) ) $ near the boundary. 
Because it holds that $\partial_{s} P^{\epsilon}_{b,j} = \rho_j \big( u^\epsilon \otimes u^\epsilon : \nabla n \big) $ on $U_j \cap \partial \Omega$, then by the established convergence we conclude
\begin{equation*}
\partial_{s} P (\sigma_1, \sigma_2, 0) = u \otimes u : \nabla n ,
\end{equation*}
which holds in $H^{-2} (U_j)$ (locally for every patch $U_j$ and then extended globally by using a partition of unity). This allows us to conclude that $P \in C^{0, \alpha} (\Omega)$ solves the boundary-value problem stated in Theorem \ref{pressureregularity} and also $\partial_n P \in C^{0,\alpha} ([0, \delta]; H^{-2} ( \partial \Omega) )$. We observe that the trace lemma (Lemma \ref{tracelemma}) has been crucial in proving that the limiting solution $P$ attains the very weak boundary condition.

Now we prove the uniqueness of the modified pressure. Suppose there exist two mean-free pressures $P_1, P_2 \in C^{0,\alpha} (\Omega)$ both solving the weak formulation from Definition \ref{weakformulationdef}. The difference $\delta P \coloneqq P_1 - P_2$ then satisfies (for $\Phi \in \mathcal{D} (\overline{\Omega})$)
\begin{equation*}
\int_\Omega \delta P \Delta \Phi dx - \int_{\partial \Omega} \delta P \partial_n \Phi dx = 0.
\end{equation*}
Now we proceed similarly to \cite{quittner}. For any $\widetilde{\zeta} \in C^\infty (\partial \Omega)$, one can take $\Phi$ to be a solution to the following Neumann boundary-value problem
\begin{equation*}
\begin{cases}
\Delta \Phi = \frac{1}{\lvert \Omega \rvert} \int_{\partial \Omega} \widetilde{\zeta} dx \quad \text{in } \Omega, \\
\partial_n \Phi = \widetilde{\zeta} \quad \text{on } \partial \Omega.
\end{cases}
\end{equation*}
Inserting this choice of test function into the weak formulation leads to (using the fact that $\delta P$ has zero average)
\begin{equation*}
\int_{\partial \Omega} \delta P \widetilde{\zeta} dx = 0.
\end{equation*}
As this holds for any $\widetilde{\zeta} \in C^\infty (\partial \Omega)$, we conclude that $\delta P \equiv 0$ on $\partial \Omega$. Next we choose $\Phi$ to be the Newtonian potential of any function $\Pi \in \mathcal{D} (\widetilde{\Omega})$ for some $\widetilde{\Omega} \subset \subset \Omega$, this gives
\begin{equation*}
\int_\Omega \delta P \Pi dx = 0.
\end{equation*}
As this holds for any $\Pi \in \mathcal{D} (\widetilde{\Omega})$ with support away from the boundary $\partial \Omega$, we conclude that $\delta P \equiv 0$ in $\Omega$ and hence $P_1 = P_2$ and that the solution to problem \eqref{neumannpressureproblem} is unique. Notice that in order to prove the uniqueness of the modified pressure we did not require regularity assumptions on $\partial_n P$. Therefore the modified pressure $P$ with $\partial_n P \in C^{0,\alpha} ([0, \delta]; H^{-2} ( \partial \Omega) )$, whose existence was proven in this theorem, is unique among all mean-free pressures in $C^{0,\alpha} (\Omega)$ solving the weak formulation from Definition \ref{weakformulationdef}.
\end{proof}

\section{Interior double Hölder exponent regularity of the pressure} \label{doubleregularitysection}
In this section, we will establish that the modified pressure constructed in Theorem \ref{pressureregularity} (cf. Theorem \ref{regularitytheorem}) in fact belongs to the space $C^{0,2\alpha} (\Omega \backslash V_{2 \delta + \gamma} )$ for $\alpha \in \left(0, \frac{1}{2} \right)$. As noted before, the case $\alpha > \frac{1}{2}$ is trivial as one almost immediately obtains the double regularity through the extension of the velocity field to the whole space and applying the regularity result in \cite{silvestre}. We now turn to the proof of Theorem \ref{doubleregularitythm}.

\begin{proof}[Proof of Theorem \ref{doubleregularitythm}]
We introduce a new cutoff function $\Psi_i$ (recall that $\gamma \in (0,\delta)$ was chosen suitably small)
\begin{equation}
\Psi_i (s) = \begin{cases}
0 \text{ if } 0 \leq s \leq 2 \delta, \\
1 \text{ if } s \geq 2 \delta + \gamma.
\end{cases}
\end{equation}
By Theorem \ref{pressureregularity}, we know there exists a unique pressure $P \in C^{0,\alpha} (\Omega)$ which is a weak solution in the sense of Definition \ref{weakformulationdef}. We then define
\begin{equation*}
P_c \coloneqq P \cdot \Psi_i.
\end{equation*}
Following \cite[~Proof of Proposition 3.3]{titi2019}, we decompose the function $P_c$ as follows
\begin{equation*}
P_c = P_{c,1} + P_{c,2} + P_{c,3},
\end{equation*}
where these three functions (formally) solve the following elliptic problems
\begin{align}
&\begin{cases}
-\Delta P_{c,1} = \Psi_i (\nabla \otimes \nabla) : (u \otimes u) \quad \text{in } \mathbb{R}^3, \\
\end{cases} \label{interiorproblem1} \\
&\begin{cases}
-\Delta P_{c,2} = - (\Delta \Psi_i) P - 2 \nabla \Psi_i \cdot \nabla P \quad \text{in } \mathbb{R}^3, \\
\end{cases} \label{interiorproblem2} \\
&\begin{cases}
-\Delta P_{c,3} = 0 \quad \text{in } \Omega \backslash V_{2 \delta}, \\
P_{c,3} = - P_{c,1} - P_{c,2} \quad \text{on } \partial (\Omega \backslash V_{2 \delta}).
\end{cases} \label{interiorproblem3}
\end{align}
Note that in the equation for $P_{c,1}$ there is no need to extend the velocity field $u$ to the whole space because $\Psi_i$ has compact support inside $\Omega$. Moreover, one can see that Hölder continuous solutions for the first two elliptic problems are unique. For given $P_{c,1}, P_{c,2} \in C^{0,\alpha} (\mathbb{R}^3)$, the third problem also possesses (at most) one solution in $C^{0, \alpha} (\Omega \backslash V_{2 \delta})$. 

For this reason it suffices to consider problems \eqref{interiorproblem1}-\eqref{interiorproblem3} while replacing $u$ by $u^\epsilon$ and $P$ by $P^\epsilon$. We then obtain approximate interior component pressures $P_{c,1}^\epsilon, P_{c,2}^\epsilon$ and $P_{c,3}^\epsilon$ and using regularity estimates which are uniform in $\epsilon$, we can then pass to the limit to prove the existence of the unique component pressures solving \eqref{interiorproblem1}-\eqref{interiorproblem3}. As similar approximation arguments have already been detailed extensively in section \ref{limitsection}, we will not repeat them here.

We will now obtain the Hölder space estimates on the component (approximate) pressures. It follows from the results in \cite{silvestre}, as well as \cite[~Theorem 13.1.1]{jostPDE}, that
\begin{equation} \label{Pc1estimate}
\lVert P_{c,1} \rVert_{C^{0,2 \alpha} (\Omega)} \leq C \lVert u \rVert_{C^{0,\alpha} (\Omega)}^2,
\end{equation}
where in this case the constant depends on the choice of $\Psi_i$ (in particular on the $W^{2,\infty} (\Omega)$ norm of $\Psi_i$). Subsequently, to obtain the needed regularity for $P_{c,2}$ we again apply \cite[~Theorem 13.1.1]{jostPDE} (after integrating by parts). By using the regularity of the modified pressure $P$ (which on the support of $\Psi_i$ agrees with the usual pressure $p$) established in Theorem \ref{pressureregularity}, we find
\begin{equation} \label{Pc2estimate}
\lVert P_{c,2} \rVert_{C^{1,\alpha} (\Omega)} \leq \lVert P \rVert_{C^{0,\alpha} (\Omega)}.
\end{equation}
Finally, due to the previously established results we can now observe that $P_{c,3}$ satisfies a homogeneous Dirichlet problem where the boundary condition belongs to $C^{0,2 \alpha} (\partial (\Omega \backslash V_{2 \delta}))$. The third component pressure $P_{c,3}$ can then be written as $P_{c,3} = - P_{c,1} - P_{c,2} + \widetilde{P}_{c,3}$, where $\widetilde{P}_{c,3}$ satisfies the following Dirichlet problem
\begin{equation*}
\begin{cases}
-\Delta \widetilde{P}_{c,3} = -\Delta P_{c,1} -\Delta P_{c,2} \quad \text{in } \Omega \backslash V_{2 \delta}, \\
\widetilde{P}_{c,3} = 0 \quad \text{on } \partial (\Omega \backslash V_{2 \delta}).
\end{cases}
\end{equation*}
Then by applying Theorem \ref{dirichletschauderestimate} one finds
\begin{equation} \label{Pc3estimate}
\lVert P_{c,3} \rVert_{C^{0,2\alpha} (\Omega \backslash V_{2 \delta})} \leq C \big(\lVert P_{c,1} \rVert_{C^{0,2\alpha} (\Omega)} + \lVert P_{c,2} \rVert_{C^{0,2\alpha} (\Omega)} \big).
\end{equation}
Then by combining estimates \eqref{Pc1estimate}-\eqref{Pc3estimate} we obtain the final result.
\end{proof}
\begin{remark}
We would like to observe that the double Hölder regularity of the pressure established in this section does not hold uniformly up to the boundary. This can be seen for example in estimates \eqref{Pc1estimate} and \eqref{Pc2estimate}, where the constants depend on the $W^{2,\infty} (\Omega)$ norm of $\Psi_i$. By letting $\Psi_i$ converge to the indicator function on $\Omega$ (while maintaining $\Psi_i$ being zero on $\partial \Omega)$ it is clear that the constants in these estimates blow up. The pressure $p$ does have double Hölder regularity (and also $P$ if the function $\phi$ is suitably modified) anywhere away from the boundary, although with increasing (double) Hölder norm as $\delta, \gamma$ are taken smaller. 
\end{remark}
\begin{remark}
It follows immediately that in the case $\alpha > \frac{1}{2}$ the double Hölder regularity of the modified pressure does not hold uniformly up to the boundary. This is because one can establish the existence of the pressure $p \in C^{0,2 \alpha} (\Omega)$ in a straightforward manner, without the need of introducing a modified pressure $P$ as the very weak boundary condition is equivalent to the usual weak boundary condition (as established in Lemma \ref{equivalencelemma}). However, as $p \in C^{0,2 \alpha} (\Omega)$ but $\phi (u \cdot n)^2 \notin C^{0,2 \alpha} (\Omega)$ (for a generic velocity field $u$) one concludes that the modified pressure $P \notin C^{0,2 \alpha} (\Omega)$.
\end{remark}

\section{The necessity of the very weak boundary condition} \label{examplesection}
It has been argued extensively earlier on that we should consider the very weak boundary condition for the pressure
\begin{equation*}
\frac{\partial}{\partial n} ( p + (u \cdot n)^2 ) \bigg\lvert_{\partial \Omega} = \big( u \otimes u \big) : \nabla n,
\end{equation*}
when $u\in C^{0,\alpha} (\Omega)$ for $\alpha \in (0,\frac{1}{2}]$, rather than the standard weak formulation of the boundary condition $\partial_n p = (u \otimes u) : \nabla n$. In this section we present an example of a H\"older continuous incompressible vector field $u\in C^{0,\alpha} (\Omega)$ for $\alpha \in (0,\frac{1}{2}]$ for which $\partial_n (u \cdot n)^2$ is not well-defined as a distribution on $\partial \Omega$ when $0<\alpha < \frac{1}{2}$; moreover, when $\alpha = \frac{1}{2}$ this quantity is not equal to zero even if it makes sense as a distribution on  $\partial \Omega$, while the velocity field satisfies the boundary condition $(u \cdot n) \lvert_{\partial \Omega} = 0$. This example provides a proof of Theorem \ref{exampletheorem}.
\begin{example} \label{weierstrassexample}
We consider the following stream function (for $0 < \alpha \leq \frac{1}{2}$)
\begin{equation}
\psi (x,y,t) = -\frac{1}{\pi} \sum_{k=0}^\infty 2^{-(\alpha + 1) k} \sin (2^k \pi x) \sin (2^k \pi y),
\end{equation}
in the two-dimensional periodic channel, i.e.
\begin{equation*}
\Omega \coloneqq \mathbb{T} \times [0,1].
\end{equation*} 
The velocity field corresponding to this stream function is given by
\begingroup
\allowdisplaybreaks
\begin{align}
u_1 (x,y,t) &= -\sum_{k=0}^\infty 2^{-\alpha k} \sin (2^k \pi x) \cos (2^k \pi y), \label{weierstrass1} \\
u_2 (x,y,t) &= \sum_{k=0}^\infty 2^{-\alpha k} \cos (2^k \pi x) \sin (2^k \pi y). \label{weierstrass2}
\end{align}
\endgroup
We will refer to this velocity field as the Weierstrass flow.

We claim that the Weierstrass flow satisfies the following properties:
\begin{enumerate}
    \item The velocity field $u = (u_1, u_2)$ belongs to $C^{0,\alpha} (\Omega)$ for every $0 < \alpha \leq 1$.
    \item It satisfies the boundary condition $(u \cdot n)\lvert_{\partial \Omega} = 0$.
    \item It is divergence-free in the sense of distributions.
    \item It holds that $\partial_n (u \cdot n)^2 \lvert_{\partial \Omega} \notin \mathcal{D}' (\partial \Omega)$.
\end{enumerate}
\end{example}
We now present a proof of this claim.
\begin{proof}
1) Let
\begingroup
\allowdisplaybreaks
\begin{align*}
u_1^N (x,y) &= -\sum_{k=0}^N 2^{-\alpha k} \sin (2^k \pi x) \cos (2^k \pi y) , \\
u_2^N (x,y) &= \sum_{k=0}^N 2^{-\alpha k} \cos (2^k \pi x) \sin (2^k \pi y).
\end{align*}
\endgroup

Observe that the partial sums $u_1^N$ and $u_2^N$ are smooth $C^\infty (\Omega)$ functions, which converge uniformly, as $N \rightarrow \infty$, to $u_1$ and $u_2$, respectively. Therefore the limit $u_1$ and $u_2$ are continuous. 

Next, we prove that the Weierstrass flow is $C^{0,\alpha} (\Omega)$. We will only prove that $u_2$ is H\"older continuous, as the proof for $u_1$ is similar. 
Observe that
\begin{align*}
u_{2}(x+h_1, y+h_2) - u_{2}(x,y) &= \sum_{k=0}^\infty 2^{-\alpha k} \cos (2^k \pi (x +h_1)) \sin (2^k \pi (y + h_2)) \\
&- \sum_{k=0}^\infty 2^{-\alpha k} \cos (2^k \pi x) \sin (2^k \pi y).
\end{align*}
We can rewrite this as follows
\begingroup
\allowdisplaybreaks
\begin{align*}
u_{2}(x+h_1, y+h_2) - u_{2}(x,y) &= \sum_{k=0}^\infty 2^{-\alpha k} \cos (2^k \pi (x +h_1)) \bigg(\sin (2^k \pi (y + h_2)) - \sin (2^k \pi y) \bigg) \\
&+ \sum_{k=0}^\infty 2^{-\alpha k} \bigg(\cos (2^k \pi (x + h_1)) - \cos (2^k \pi x) \bigg) \sin (2^k \pi y) \\
&= \sum_{k=0}^\infty 2^{-\alpha k + 1} \cos (2^k \pi (x +h_2)) \cos ( 2^{k-1} \pi (2 x + h_2) ) \sin ( 2^{k-1} \pi h_2) \\
&- \sum_{k=0}^\infty 2^{-\alpha k + 1} \sin ( 2^{k-1} \pi (2 x + h_1) ) \sin ( 2^{k-1} \pi h_1) \sin (2^k \pi y).
\end{align*}
\endgroup
Now we split the above sums into parts $0 \leq k \leq p_1$, $k > p_1$ respectively $0 \leq k \leq p_2$ and $k > p_2$ for some positive integers $p_1$ and $p_2$ satisfying $2^{p_1-1} \lvert h_1 \rvert \leq 1 < 2^{p_1} \lvert h_1 \rvert$ and $2^{p_2-1} \lvert h_2 \rvert \leq 1 < 2^{p_2} \lvert h_2 \rvert$. This implies
\begingroup
\allowdisplaybreaks
\begin{align*}
&\lvert u_{2}(x+h_1, y+h_2) - u_{2}(x,y) \rvert \leq \sum_{k=0}^\infty 2^{-\alpha k + 1} \lvert \cos (2^k \pi (x +h_1)) \rvert \lvert \cos ( 2^{k-1} \pi (2 y + h_2) ) \rvert \lvert \sin ( 2^{k-1} \pi h_2) \rvert \\
&+ \sum_{k=0}^\infty 2^{-\alpha k + 1} \lvert \sin ( 2^{k-1} \pi (2 x + h_1) ) \rvert \lvert \sin ( 2^{k-1} \pi h_1) \rvert \lvert \sin (2^k \pi y) \rvert \leq \sum_{k=0}^\infty 2^{-\alpha k + 1} \lvert \sin ( 2^{k-1} \pi h_2) \rvert \\
&+ \sum_{k=0}^\infty 2^{-\alpha k + 1} \lvert \sin ( 2^{k-1} \pi h_1) \rvert \leq \sum_{k=0}^{p_2} \bigg( 2^{-\alpha k + 1} \cdot 2^{k-1} \pi \lvert h_2 \rvert \bigg) + \sum_{k=p_2+1}^\infty 2^{-\alpha k + 1} \\
&+ \sum_{k=0}^{p_1} \bigg( 2^{-\alpha k + 1} \cdot 2^{k-1} \pi \lvert h_1 \rvert \bigg) + \sum_{k=p_1 + 1}^\infty 2^{-\alpha k + 1} \\
&= (2^{-\alpha (p_1 + 1) } + 2^{-\alpha (p_2 + 1) } ) \frac{2}{1 - 2^{-\alpha}} + \pi \lvert h_2 \rvert \frac{1 - 2^{(1 - \alpha) (p_2 + 1)}}{1 - 2^{1 - \alpha}} + \pi \lvert h_1 \rvert \frac{1 - 2^{(1 - \alpha) (p_1 + 1)}}{1 - 2^{1 - \alpha}} \\
&\leq (2^{-\alpha p_1  } + 2^{-\alpha p_2  } ) \frac{2^{1 - \alpha}}{1 - 2^{-\alpha}} + \pi \lvert h_2 \rvert \frac{2^{(1 - \alpha) (p_2 + 1)}}{2^{1 - \alpha} - 1} + \pi \lvert h_1 \rvert \frac{2^{(1 - \alpha) (p_1 + 1)}}{2^{1 - \alpha} - 1}.
\end{align*}
\endgroup
In the above we have used the inequality $|\sin z| \le |z|$. Now using that $1 < 2^{p_1} \lvert h_1 \rvert \leq 2$ and $1 < 2^{p_2} \lvert h_2 \rvert \leq 2$, we find
\begin{align*}
2^{-\alpha p_1} \leq \lvert h_1 \rvert^{\alpha}, \quad 2^{-\alpha p_2} \leq \lvert h_2 \rvert^{\alpha}, \quad 2^{(p_1 - 1)(1 - \alpha)} \leq \lvert h_1 \rvert^{- (1 - \alpha)}, \quad 2^{(p_2 - 1)(1 - \alpha)} \leq \lvert h_2 \rvert^{- (1 - \alpha)}.
\end{align*}
From this we are able to conclude that
\begin{align*}
&\lvert u_{2}(x+h_1, y+h_2) - u_{2}(x,y) \rvert \leq (\lvert h_1 \rvert^{\alpha} + \lvert h_2 \rvert^{\alpha} ) \frac{2^{1 - \alpha}}{1 - 2^{-\alpha}} + \pi \frac{2^{2 (1 - \alpha) }}{2^{1 - \alpha} - 1} \lvert h_2 \rvert^\alpha + \pi \frac{2^{2 (1 - \alpha) }}{2^{1 - \alpha} - 1} \lvert h_1 \rvert^\alpha \\
&\leq 2^{2 - \alpha} \bigg( \frac{1}{1 - 2^{-\alpha}} + \frac{2^{1 - \alpha } \pi }{2^{1 - \alpha} - 1} \bigg) \lvert h \rvert^{\alpha},
\end{align*}
where $h \coloneqq (h_1, h_2)$. Therefore $u \in C^{0,\alpha} (\Omega)$.\\
2) We demonstrate that the velocity field satisfies the boundary condition $u_2 \lvert_{\partial \Omega} =u\cdot n \lvert_{\partial \Omega} =0$. Indeed, one can check that 
\begin{equation*}
u_2 (x,0) = \sum_{k=0}^\infty 2^{-\alpha k} \cos (2^k \pi x) \sin (2^k \pi \cdot 0) = 0,
\end{equation*}
since the function is continuous and the series converges uniformly. Similarly, one can check that $u_2 (x,1) = 0$ and hence $(u \cdot n)\lvert_{\partial \Omega} = 0$.\\
3) We will now show that the velocity field is divergence-free in the sense of distributions. We can easily check that
\begin{align*}
\partial_x u^N_1 + \partial_y u^N_2 &= 0.
\end{align*}
This means that the partial sums $u^N$ are weakly divergence-free, i.e.
\begin{equation*}
\int_{\Omega} u^N \cdot \nabla \phi dx = 0, \quad \forall \phi \in \mathcal{D} (  \Omega ; \mathbb{R}).
\end{equation*}
Since $u^N_1$ and $u^N_2$ converge in $L^\infty (\Omega)$ to $u_1 $ and $u_2$, therefore it follows that
\begin{equation*}
\int_\Omega u \cdot \nabla \phi dx=0.
\end{equation*}
We conclude that $u$ is divergence-free in the sense of distributions. \\
4) 
Now we show that $\partial_n (u \cdot n)^2 \lvert_{\partial \Omega}$ cannot be defined as an element of $\mathcal{D}' (\partial \Omega)$ in the case when $\alpha \in \left(0,\frac{1}{2} \right)$. In particular, this implies that $\partial_n (u \cdot n)^2 \notin H^{-2} (\partial \Omega)$. In fact, in Appendix \ref{normalderivativeappendix} we will show that away from the boundary $\partial_y u_2^2 (\cdot, y)$ cannot be defined as an element of $\mathcal{D}' (\mathbb{T})$ for a dense set of points $y \in [0,1]$. It should be noted that $\partial_y u_2 $ is perfectly well-defined as a distribution on the whole domain, but as we will show below $\partial_y u_2^2 (\cdot, y)$ might not be a distribution on $\mathbb{T}$ for $y = 0,1$.

In this section, we will consider the case $y=0$, as this concerns the boundary condition. More precisely first one observes that the following function (for $\theta \in \mathcal{D} (\mathbb{T})$)
\begin{equation}
U(y; \theta) \coloneqq \langle  u_2^2 (\cdot,y),  \theta \rangle = \int_{\mathbb{T}} u_2^2 (x,y) \theta (x)dx
\end{equation}
belongs to $C^{0,\alpha} [0,1]$ and is equal to $0$ for $y=0,1$. 

Hence the existence of the derivative on the boundary (i.e. for $y=0$) follows if the following limit exists
$$
\lim_{y\rightarrow 0 +} \frac1{y}  \int_{\mathbb{T}} u_2^2 (x,y) \theta (x)dx.
$$
As already has been proved in Lemma \ref{equivalencelemma} (and observed in \cite{bardos2021}) this limit exists and is equal to $0$ as long as $\frac12<\alpha \leq 1$. 
To explore the behaviour of this limit for the case $0<\alpha \le \frac12$ we will consider the Weierstrass series defined in equations \eqref{weierstrass1} and \eqref{weierstrass2}. We find
\begingroup
\allowdisplaybreaks
\begin{equation}\label{basica1}
\begin{aligned}
&v(x,y) \coloneqq (u_2(x,y))^2= \sum_{k_1, k_2=0}^\infty \bigg( 2^{-\alpha (k_1 + k_2)} \cos(2^{k_1} \pi x)  \cos(2^{k_2}\pi  x)  \sin(2^{k_1}y)  \sin(2^{k_2} y)\bigg), \\
&U(y;\theta) =  \sum_{k_1, k_2=0}^\infty \int_{\mathbb{T}} v (x,y) \theta (x)dx \\
&= \sum_{k_1, k_2=0}^\infty \bigg( \int_{\mathbb{T}} 2^{-\alpha (k_1 + k_2)} \cos(2^{k_1}  \pi x)  \cos(2^{k_2}\pi  x)\theta (x) dx\bigg)  \sin(2^{k_1}y)  \sin(2^{k_2} y) .
\end{aligned}
\end{equation}
\endgroup
The purpose of this section is to consider the case {\bf $0<\alpha \le \frac12$} and to prove the following proposition regarding the behaviour of the function  $U(y;\theta)$ as defined in equation \eqref{basica1}.
\begin{proposition}\label{th1}
$$\,$$
\begin{enumerate}
\item Suppose $\theta \in \mathcal D({\mathbb{T}})$ satisfies
\begin{equation}
\int_{\mathbb{T}} \theta(x) dx =0,
\end{equation}
then the limit
$$
\lim_{y\rightarrow 0 +} \frac1{y}  \int_{\mathbb{T}} v (x,y) \theta (x)dx
$$
is well-defined and is equal to $0\,.$

\item Otherwise if $0 < \alpha < \frac{1}{2}$ and
\begin{equation}
\int_{\mathbb{T}} \theta(x) dx \not=0,
\end{equation}
then it holds that
$$
\liminf _{y \rightarrow 0 +} \frac1{y } \bigg\lvert \int_{\mathbb{T}} v (x,y) \theta (x)dx \bigg\rvert =\infty.
$$
As a consequence the function
\begin{equation}
U(y; \theta) = \int_{\mathbb{T}} v (x,y) \theta (x)dx
\end{equation}
does not have a well-defined derivative at the point $y=0\,.$ 
\item If $\alpha = \frac{1}{2}$ and 
\begin{equation}
\int_{\mathbb{T}} \theta(x) dx \not=0,
\end{equation}
then we have that
\begin{equation}
\liminf _{y \rightarrow 0 +} \frac1{y } \bigg\lvert \int_{\mathbb{T}} v (x,y) \theta (x)dx \bigg\rvert \geq 2.
\end{equation}
Consequently, if the derivative of $U(y,\theta)$ exists it is not equal to zero.
\end{enumerate}
\end{proposition}
\begin{proof}
For the proof one first eliminates the nonresonant terms (i.e., the terms involving $k_1\not=k_2$ ) and then a comparison argument is used. As such the subscripts $R$ and $NR$ are used to denote the resonant and nonresonant parts of $U$, respectively.

First one has the following.
\begin{lemma}
The function
\begin{equation}
U_{NR} (y,\theta) \coloneqq \sum_{k_1,k_2 = 0,k_1\not=k_2}^\infty \bigg(2^{-\alpha (k_1 + k_2)} \bigg[ \int_{\mathbb{T}} \cos(2^{k_1}  \pi x )  \cos(2^{k_2}  \pi x)\theta (x) dx \bigg] \sin(2^{k_1}y)  \sin(2^{k_2} y)  \bigg)
\end{equation}
belongs to $C^1([0,1])$, moreover we have that $\del_y U_{NR} (y;\theta) \lvert_{y=0} =0$.
\end{lemma}
\begin{proof}
One first recalls the following trigonometric identity
\begin{equation*}
\cos(2^{k_1}  \pi x)  \cos(2^{k_2} \pi x)=\frac12 \bigg[\cos((2^{k_1}+ 2^{k_2})  \pi x) +\cos((2^{k_1}- 2^{k_2}) \pi x)\bigg].
\end{equation*}
Then for $k_1 \neq k_2$ and $m > 0$ it holds that
\begin{equation} \label{nonres}
\begin{aligned}
&\int_{\mathbb{T}} \cos(2^{k_1}  \pi x)  \cos(2^{k_2}  \pi x)\theta (x) dx=\frac{(-1)^{m}}{2( \pi( 2^{k_1}+ 2^{k_2}))^{2m} } \int_{\mathbb{T}}  \cos((2^{k_1}+ 2^{k_2}) \pi x) \frac{d^{2m}}{dx^{2m}}\theta(x) dx\\
&+\frac{(-1)^{m}}{2(\pi(2^{k_1}- 2^{k_2}))^{2m} } \int_{\mathbb{T}} \cos((2^{k_1}- 2^{k_2}) \pi x) \frac{d^{2m}}{dx^{2m}}\theta(x) dx,
\end{aligned}
\end{equation}
moreover we have
\begin{equation}\label{nonres2}
\frac{d}{dy} ( \sin(2^{k_1} \pi y)  \sin(2^{k_2} \pi y) )= 2^{k_1} \pi \cos(2^{k_1} \pi y)  \sin(2^{k_2} \pi y) + 2^{k_2} \pi \sin(2^{k_1} \pi y)  \cos(2^{k_2} \pi y).
\end{equation}
Combining equations (\ref{nonres}) and (\ref{nonres2} ) one obtains that
\begin{equation}
\begin{aligned}
&\bigg\lvert\frac{d}{dy} \bigg[ \bigg( \int_{\mathbb{T}} \cos(2^{k_1}\pi x)  \cos(2^{k_2}\pi x)\theta (x) dx \bigg)  \sin(2^{k_1}y)  \sin(2^{k_2} y) \bigg] \bigg\rvert\\
&\le C \pi \bigg( \frac 1{2(\pi (2^{k_1}+ 2^{k_2}))^{2m} } +
   \frac 1{2((\pi(2^{k_1}- 2^{k_2}))^{2m} } \bigg) (2^{k_1} + 2^{k_2} )\int_{\mathbb{T}} \bigg\lvert\frac{d^{2m}}{dx^{2m}} \theta(x)\bigg\rvert dx 
\end{aligned}
 \end{equation}
which for $m \geq 1$ is a sequence constituting the terms of an absolutely converging series.

As a consequence  the function
\begin{equation*}
U_{NR}(y;\theta)= \sum_{k_1,k_2=0 , k_1\not=k_2}^\infty \bigg(2^{-\alpha (k_1 + k_2)} \bigg[ \int_{\mathbb{T}} \cos(2^{k_1}\pi x)  \cos(2^{k_2}\pi  x)\theta (x) dx \bigg] \sin(2^{k_1} \pi y)  \sin(2^{k_2} \pi y) \bigg)  
\end{equation*}
belongs to the space $C^{1}([0,1])$ and satisfies the relation:
\begin{equation}
\del_y U_{NR} (y;\theta) \lvert_{y=0} =0.
\end{equation}

\end{proof}
Therefore one only has to consider the resonant part of the Weierstrass series (i.e. the case $k_1 = k_2$ in equation \eqref{basica1}), which is
\begin{equation}
U_R(y;\theta) =  \sum_{k=0 }^\infty 2^{-2\alpha k} \bigg( \int_{\mathbb{T}} (\cos(2^{k } \pi  x)  )^2  \theta (x) dx \bigg)  (\sin(2^{k } \pi y))^2.
\end{equation}
By using the identity
\begin{equation}
 (\cos(2^{k } \pi x)  )^2=\frac12 \bigg(1+ \cos(2^{k+1 } \pi  x)  \bigg),
\end{equation}
one has that
\begin{align*}
U_R(y;\theta) &= \frac12 \bigg( \int_{\mathbb{T}} \theta (x) dx \bigg) \sum_{k =0}^\infty \bigg( 2^{-2\alpha k} (\sin(2^{k }\pi y) )^2 \bigg) \\
&+ \frac{1}{2} \sum_{k =0}^\infty \bigg[ 2^{-2\alpha k} \bigg( \int_{\mathbb{T}} \cos(2^{k +1}  \pi x)    \theta (x) dx \bigg) (\sin(2^{k } \pi y))^2 \bigg].
\end{align*}
Similarly as before, the function
\begin{equation}
U_{RNR}(y;\theta)= \frac{1}{2} \sum_{k=0}^\infty 2^{-2\alpha k} \bigg( \int_{\mathbb{T}} \cos(2^{k +1}\pi  x) \theta (x) dx \bigg)  (\sin(2^{k } \pi y))^2
\end{equation}
is a series converging in $C^{1}([0,1]) $ with derivative equal to $0$ for $y=0$.
Hence the completion of the proof now relies only on the analysis of the behaviour of the term
\begin{equation}
\frac{1}{y} U_{RR}(y;\theta)= \frac{1}{2y} \int_{\mathbb{T}} \theta (x) dx \sum_{k=0 }^\infty 2^{-2\alpha k} (\sin(2^{k } \pi y))^2,
\end{equation}
which is equal to $0$ when 
\begin{equation}
\int_{\mathbb{T}} \theta (x) dx =0.
\end{equation}
This proves point 1 of Proposition \ref{th1}. To prove point 2 one introduces the sequence $y_n=2^{-n}$ (which is converging to 0, as $n \rightarrow \infty$)  and consider the expression
\begin{equation}
\frac{1}{ y_n } U_{RR} (y_n;\theta)=2^{n-1}   \sum_{k =0}^{\infty} 2^{-2\alpha k} (\sin(2^{k } \pi 2^{-n}))^2.
\end{equation}
We first observe that $\sin(2^{k } \pi 2^{-n}) = 0$ for $k \geq n$. Therefore, the above sum is actually given by
\begin{equation}
\frac{1}{ y_n } U_{RR} (y_n;\theta)=2^{n-1}   \sum_{k =0}^{n-1} 2^{-2\alpha k} (\sin(2^{k -n} \pi ))^2.
\end{equation}
Now we observe that for $0 \leq k \leq n-1$ we have that $0 \leq 2^{k -n} \pi \leq \frac{\pi}{2}$. We recall that for $x \in \big[ 0, \frac{\pi}{2} \big]$ it holds that $\sin (x) \geq \frac{2}{\pi } x$. Applying this inequality to the series above gives
\begin{align*}
\frac{1}{ y_n } U_{RR} (y_n;\theta) &\geq 2^{n-1} \sum_{k =0}^{n-1} 2^{-2\alpha k} \bigg( \frac{2}{\pi} 2^{k -n} \pi \bigg)^2 = 2^{-n+1} \sum_{k =0}^{n-1} 2^{2k (1 - \alpha)} \\
&= 2^{-n+1} \frac{2^{2n(1-\alpha)} - 1}{2^{2(1-\alpha)} -1} = \frac{2}{2^{2(1-\alpha)} -1} \big( 2^{n(1-2\alpha)} - 2^{-n} \big).
\end{align*}
From the above we conclude that
\begin{equation}
\liminf_{n \rightarrow \infty} \frac{1}{ y_n } U_{RR} (y_n;\theta) \begin{cases}
= \infty \quad \text{if } 0 < \alpha < \frac{1}{2}, \\
\geq 2 \quad \text{if } \alpha = \frac{1}{2}, \\
\geq 0 \quad \text{if } \alpha > \frac{1}{2}.
\end{cases}
\end{equation}
Observing that 
\begin{equation*}
U(y,\theta) = U_{NR}(y,\theta) +  U_{R}(y,\theta) =  U_{NR}(y,\theta) + U_{RNR}(y,\theta) + U_{RR}(y,\theta)
\end{equation*}
completes the proof of point 2 of Proposition \ref{th1}. Since for $\alpha >\frac 12$ we have that $\lim 2^{n(1-2\alpha)}$ goes to $0$ with $n\rightarrow \infty$ point 2 is not in contradiction with point 1 of Proposition \ref{th1} and Lemma \ref{equivalencelemma}.
\end{proof}

\end{proof}
\begin{remark}
The definition of $\del_n(u\cdot n)^2$ turns out to be a subtle issue for solutions in $C^{0,\alpha} (\Omega)$. In the case $\alpha >\frac12$ the definition is trivial, as was observed before. For $0<\alpha\le \frac12$ it depends on the mean of the function $\theta (x)$. The trace of the $\del_y u_2^2$ term remains well-defined as an element of the dual of test functions with mean value $0$, on the other hand it is no longer defined when the mean value of the test functions is not $0\,.$ We note that the pressure associated with the Weierstrass flow \eqref{weierstrass1}
-\eqref{weierstrass2} cannot lie in $H^{3/2+\lambda} (\Omega)$ for any $\lambda > 0$, as otherwise $\partial_n p \in \mathcal{D} (\partial \Omega)$ (and hence also $\partial_n (u \cdot n)^2 \lvert_{\partial \Omega} \in \mathcal{D} (\partial \Omega)$) by using a trace theorem.
\end{remark}
\begin{remark}
We observe that the Weierstrass flow is not a stationary solution of the 2D Euler equations. We recall that the pressure solves the following equation (in a weak sense)
\begin{equation*}
\nabla \cdot \big[ \nabla \cdot (u \otimes u) + \nabla p \big] = 0.
\end{equation*}
A stationary vector field is a solution of the Euler equations if $\nabla \cdot (u \otimes u)$ is also curl-free. Now we observe that $\nabla \cdot (u \otimes u) + \nabla p$ is not curl-free. By standard identities, one can compute that (as we are in the two-dimensional case)
\begin{align*}
\nabla \times \big[ \nabla \cdot (u \otimes u)  \big] &= \nabla \cdot (u \omega).
\end{align*}
The vorticity associated with the Weierstrass flow can formally be computed to be
\begin{equation}
\omega = \partial_x u_2 - \partial_y u_1 = - 2 \pi \sum_{k=0}^\infty 2^{(1-\alpha)k} \sin (2^k \pi x) \sin (2^k \pi y).
\end{equation}
Then we compute that
\begin{align*}
\nabla \cdot (u \omega) &= (u_1 \partial_x + u_2 \partial_y) \omega \\
&= 2 \pi^2 \bigg( \sum_{k_1=0}^\infty 2^{-\alpha k_1} \sin (2^{k_1} \pi x) \cos (2^{k_1} \pi y) \bigg) \cdot \bigg( \sum_{k_2=0}^\infty 2^{(2-\alpha)k_2} \cos (2^{k_2} \pi x) \sin (2^{k_2} \pi y) \bigg) \\
&-2 \pi^2 \bigg( \sum_{k_1=0}^\infty 2^{-\alpha k_1} \cos (2^{k_1} \pi x) \sin (2^{k_1} \pi y) \bigg) \cdot \bigg( \sum_{k_2=0}^\infty 2^{(2-\alpha)k_2} \sin (2^{k_2} \pi x) \cos (2^{k_2} \pi y) \bigg) \\
&\neq 0.
\end{align*}
We remark that these formal computations can be justified by considering partial sums of the series in equations \eqref{weierstrass1} and \eqref{weierstrass2}, and subsequently taking limits as was done earlier in this section. \\
We note that to the best of our knowledge it remains an open problem to construct a weak solution of the Euler equations with $u \in C^{0,\alpha} (\Omega)$ for $0 < \alpha \leq \frac{1}{2}$ such that $(u \cdot n) \lvert_{\partial \Omega} = 0$ and $\partial_n (u \cdot n)^2 \lvert_{\partial \Omega} \notin \mathcal{D}' (\partial \Omega)$.
\end{remark}
\begin{remark}
In two spatial variables in a domain $\Omega$ with a geodesic change of variable the same results as in Example \ref{weierstrassexample}  hold when using the Weierstrass flow from equations \eqref{weierstrass1} and \eqref{weierstrass2}. Now $x$ is the tangential variable and $y$ the distance to $\del\Omega\,.$ In fact the whole derivation from Example \ref{weierstrassexample} is local in nature and could be considered for any hypersurface $\Sigma \subset \del \Omega$ (also in the three-dimensional case). This leads to the following theorem.
\end{remark} 
\begin{theorem} \label{generaldomainweierstrass}
Let $\Sigma \subset \overline{\Omega}$ be a hypersurface with local geodesic coordinates, namely a tangential coordinate $x$ and normal coordinate $y$. We consider the trace
\begin{equation*}
\del_y(u \cdot n)^2 \lvert_{\Sigma}.
\end{equation*}
Then: 
\begin{enumerate}
    \item If $u\in C^{0,\alpha}(\overline \Omega)$ with $\frac12<\alpha$ and $(u \cdot n)\lvert_{\partial \Omega} = 0$, then the trace of $\del_y (u \cdot n)^2$ on $\Sigma$ is well-defined and equal to zero. 
    \item Otherwise if $0< \alpha < \frac12$, there exists a velocity field $u\in C^{0,\alpha}(\overline \Omega)$ with $(u \cdot n)\lvert_{\partial \Omega} = 0$ such that the trace $\del_y (u \cdot n)^2 \not=0$ on $\Sigma$ is not well-defined even as an element of $\mathcal D'(\Sigma)$.
    \item If $\alpha = \frac{1}{2}$, there exists a velocity $u\in C^{0,\frac{1}{2}}(\overline \Omega)$ with $(u \cdot n)\lvert_{\partial \Omega} = 0$ such that if the trace $\partial_y (u \cdot n)^2$ is well-defined, it is nonzero.
\end{enumerate}
\end{theorem}  
\begin{proof}
We leave the proof of this result to the interested reader, as it is a straightforward adaption (using the coordinate transformation \eqref{transform1}-\eqref{transform3} from section \ref{parametrisationsection}) of the proof given on the half-space in Example \ref{weierstrassexample}.
\end{proof}

\section{Conclusion}
In its present form this contribution concerns the propagation of Hölder regularity from the velocity to the pressure for weak solutions of the 3D Euler equations in a domain with boundary with given initial data and over a finite time interval $(0,T)$. The regularity result proven in this paper is a generalisation of the result proven in \cite{bardos2021} from the 2D to the 3D incompressible Euler equations, by using the very weak boundary condition \eqref{veryweakboundcondeq} that was introduced in \cite{bardos2021}.

One of the other main contributions of this work was to prove that the very weak boundary condition \eqref{veryweakboundcondeq} is essential, namely by constructing an example of a velocity field for which $\partial_n (u \cdot n)^2 \lvert_{\partial \Omega} \notin \mathcal{D}' (\partial \Omega)$ (as outlined in section \ref{examplesection}), which implies for the associated pressure that $\partial_n p \notin \mathcal{D}' (\partial \Omega)$. Moreover, in Proposition \ref{weaksolutionproposition} and Theorem \ref{veryweakboundcondthm} we rigorously derived the very weak boundary condition directly from the weak formulation of the Euler equations. Therefore in this paper we have established that the very weak boundary condition is satisfied by all Hölder continuous weak solutions of the Euler equations.

One key difference in the regularity proof for the pressure with the work \cite{bardos2021} is that we now use a local parametrisation of the boundary as opposed to a global parametrisation (as in 3D a global parameterisation of the boundary is generally not available, unlike in 2D). Note that in principle our approach can be generalised easily without many problems to any dimension, by using the higher-dimensional analogue of the coordinate transformation \eqref{transform1}-\eqref{transform3}. In Theorem \ref{doubleregularitythm} we then proved the interior double Hölder regularity of the pressure.

The pressure regularity results from this work constitute the last part of the proof of the first half of the Onsager conjecture (the sufficient conditions for energy conservation of weak solutions) in the presence of physical boundaries, which was given in \cite{titi2018}. As was noted before, the results in this paper remove the need for separate regularity assumptions on the pressure in the proof of such Onsager-type results. A sharper Onsager-type result without pressure regularity assumptions was established in Theorem \ref{onsagerconjecture}, relying on the results in \cite{titi2019}. Moreover, the results of this paper have clarified the precise boundary-value problem satisfied by the pressure of weak solutions of the Euler equations. 

Because the phenomenon of anomalous dissipation is intimately related with low regularity weak solutions of the Euler equations, we expect that the newly introduced very weak boundary condition for the pressure considered in this work to have connections with the dissipation anomaly and turbulence. Finally, the results of this paper are a very preliminary step towards the analysis of the distinguished limit of $\nu\rightarrow 0$ and $T\rightarrow \infty$ of the Navier-Stokes equations (potentially also in the presence of external forcing), which turns out to be one of the basic issues for hydrodynamic turbulence.

\section*{Acknowledgements}
The authors would like to thank the anonymous referees for their many useful suggestions and constructive comments, which have improved the quality of the paper. C.B. and E.S.T. acknowledge the partial support by the Simons Foundation Award No. 663281 granted to the Institute of Mathematics of the Polish Academy of Sciences for the years 2021-2023. C.B. thanks also the Laboratory Jacques-Louis Lions (Sorbonne University) for its support during the final completion of the mansucript. D.W.B. acknowledges support from the Cambridge Trust, the Cantab Capital Institute for Mathematics of Information and the Hendrik Muller fund. The work of E.S.T. has benefited from the inspiring environment of the CRC 1114 ``Scaling Cascades in Complex Systems", Project Number 235221301, Project C06, funded by Deutsche Forschungsgemeinschaft (DFG). The authors would like to thank the Isaac Newton Institute for Mathematical Sciences, Cambridge, for support and hospitality during the programmes ``Mathematical aspects of turbulence: where do we stand?'' (TUR) and also ``Frontiers in kinetic theory: connecting microscopic to macroscopic scales - KineCon 2022'' (FKT) where work on this paper was undertaken. This work was supported by EPSRC grant no EP/R014604/1.

\section*{Declaration}
The authors declare that they have no competing interests. Data sharing is not applicable to this article as no datasets were generated or analysed during the current study.

\begin{appendices}
\section{Schauder-type estimate for Dirichlet problem} \label{schauderappendix1}
In this appendix, we will prove a Schauder-type estimate that will be used in the main body of the paper. As this result does not seem to be present in the literature, we provide a proof here for the sake of completeness. We will use the Einstein summation convention in what follows. The estimate is given in the following theorem.
\begin{theorem} \label{dirichletschauderestimate}
Let $v \in H^1_0 (\Omega)$ (where $\Omega \subset \mathbb{R}^3$ is an open set with $C^{3,\alpha}$ boundary for $\alpha > 0$) be the unique solution of the following problem
\begin{equation} \label{ellipticschaudereq}
\begin{cases}
\Delta v = \partial_i \partial_j (F_{ij}) \quad \text{in } \Omega, \\
v = 0 \quad \text{on } \partial \Omega,
\end{cases}
\end{equation}
where $F_{ij} \in C^{2,\alpha} (\Omega)$ for $1 \leq i,j \leq 3$. Then the following estimate holds
\begin{equation}
\lVert v \rVert_{C^{0,\alpha} (\Omega)} \leq C \lVert F \rVert_{C^{0,\alpha} (\Omega)} + C \lVert v \rVert_{L^\infty (\Omega)} ,
\end{equation}
where the constant $C$ depends on $\Omega$ and $\alpha$.
\end{theorem}
\begin{remark}
We observe that in Theorem \ref{dirichletschauderestimate} the requirement $F_{ij} \in C^{2,\alpha} (\Omega)$ can be relaxed to $F_{ij} \in C^{0,\alpha} (\Omega)$ by a density argument (using a result of the type of Lemma \ref{mollificationlemma}). We will not do so here, as in the cases we will apply Theorem \ref{dirichletschauderestimate} both the solution $v$ as well as the forcing $F$ will be smooth but we need estimates which are independent of the mollification parameter $\epsilon$.
\end{remark}
In order to prove Theorem \ref{dirichletschauderestimate}, we have to obtain the interior and boundary regularity estimates separately. We first prove the interior Schauder-type estimate.
\begin{proposition} \label{interiorschauderprop}
Let $v \in H^1_0 (\Omega)$ be the unique solution to problem \eqref{ellipticschaudereq}, again with $F_{ij} \in C^{2,\alpha} (\Omega)$. Then for any $\Omega_0 \subset \subset \Omega$ we have 
\begin{equation}
\lVert v \rVert_{C^{0,\alpha} (\Omega_0)} \leq C \lVert F \rVert_{C^{0,\alpha} (\Omega)} + \lVert v \rVert_{L^\infty (\Omega)}. \label{interiorschauderestimate}
\end{equation}
\end{proposition}
\begin{proof}
In order to prove the interior regularity estimate, we can restrict to the case that $\Omega_0$ is the unit ball. If $\Omega_0$ is a different type of domain, by using a compactness argument, we can cover $\Omega_0$ by a finite number of open balls (which can be rescaled to the unit ball). As Hölder regularity is a local property, the global regularity estimate for $\Omega_0$ follows from the local regularity estimate for each open ball. Therefore we can reduce the proof of the proposition to the case of the unit ball. Therefore we assume $\Omega_0 = B_1 (0)$ from now on. We introduce a smooth cutoff function $\widetilde{\chi}$ of the following form (again taking $\gamma > 0$ suitably small)
\begin{equation*}
\widetilde{\chi} (x) = \begin{cases}
1 \quad \text{if } \lvert x \rvert \leq 1, \\
0 \quad \text{if } \lvert x \rvert \geq 2 - \gamma.
\end{cases}
\end{equation*}
We see that $\widetilde{v} \coloneqq \widetilde{\chi} v$ satisfies the following equation
\begin{equation*}
\Delta \widetilde{v} = \Delta \widetilde{\chi} v + \nabla \widetilde{\chi} \cdot \nabla v + \widetilde{\chi} \partial_i \partial_j (F_{ij}),
\end{equation*}
together with the boundary condition $\widetilde{v} \lvert_{\partial B_2 (0)} = 0$. We introduce the splitting $\widetilde{v} = \widetilde{v}_1 + \widetilde{v}_2$, each of which satisfy the following Dirichlet problems
\begin{align*}
&\begin{cases}
\Delta \widetilde{v}_1 = \Delta \widetilde{\chi} v + \nabla \widetilde{\chi} \cdot \nabla v \quad \text{in } B_2 (0), \\
\widetilde{v}_1 = 0 \quad \text{on } \partial B_2 (0).
\end{cases} \\
&\begin{cases}
\Delta \widetilde{v}_2 = \widetilde{\chi} \partial_i \partial_j (F_{ij}) \quad \text{in } B_2 (0), \\
\widetilde{v}_2 = 0 \quad \text{on } \partial B_2 (0).
\end{cases}
\end{align*}
We recall from \cite[~Theorem 1]{dolzmann} the existence of the Green's function $G$ which satisfies the following problem
\begin{equation}
\begin{cases}
\Delta_y G(x,y) = - \delta (x-y) \quad \text{for } y \in B_2 (0), \\
G(x, y) = 0 \quad \text{for } y \in \partial B_2 (0). 
\end{cases}
\end{equation}
Moreover, from \cite[~Theorem 1]{dolzmann} we know that it satisfies the following pointwise estimates (in the case of three dimensions)
\begingroup
\allowdisplaybreaks
\begin{align}
\lvert G(x,y) \rvert &\leq C \lvert x - y \rvert^{-1}, \label{pointwise1} \\
\lvert D^\beta G(x,y) \rvert &\leq C \lvert x - y \rvert^{-2}, \label{pointwise2} \\
\lvert D^\gamma G(x, y)  \rvert &\leq C \lvert x - y \rvert^{-3}, \label{pointwise3} \\
\lvert D^\gamma G(x, y) - D^\gamma G(z,y) \rvert &\leq \lvert x -z \rvert^{ \widetilde{\alpha}} \max \{ \lvert x - y \rvert^{-3 - \widetilde{\alpha}} , \lvert z - y \rvert^{-3 - \widetilde{\alpha}} \}, \label{pointwise4}
\end{align}
\endgroup
where $\widetilde{\alpha} \in (0,1)$, $\beta$ is a multi-index of order 1 and $\gamma$ is a multi-index of order 2. By using the approach from Theorem 13.1.1 from \cite{jostPDE} or Theorem 4.15 from \cite{gilbarg}, we obtain the following estimate
\begin{equation*}
\lVert \widetilde{v}_1 \rVert_{C^{0,\alpha} (\Omega_0)} \lesssim \lVert v \rVert_{L^\infty (\Omega)},
\end{equation*}
where the constant depends on $\widetilde{\chi}$. Now we turn to estimating $\widetilde{v}_2$. Using the Green's function we can write the solution as (where the partial derivatives are with respect to $y$)
\begin{align*}
\widetilde{v}_2 (x) &= - \int_{B_2 (0)} G(x,y) \widetilde{\chi} \partial_i \partial_j (F_{ij} (y)) dy = \int_{B_2 (0)} \partial_i (G(x,y) \widetilde{\chi} (y)) \partial_j (F_{ij} (y) - F_{ij} (x) ) dy  \\
&= - \int_{B_2 (0)} \partial_i \partial_j (G(x,y) \widetilde{\chi} (y)) (F_{ij} (y) - F_{ij} (x) ) dy,
\end{align*}
where we have used that $\widetilde{\chi}$ vanishes on $\partial B_2 (0)$. Now we derive the Schauder estimate (where we let $x_1, x_2 \in \Omega_0$, $\overline{x} = (x_1 + x_2)/2$ and $\delta = \lvert x_1 - x_2 \rvert$, see \cite{jostPDE} for a related derivation)
\begingroup
\allowdisplaybreaks
\begin{align*}
&\widetilde{v}_2 (x_1) - \widetilde{v}_2 (x_2) =  -\int_{B_2 (0)} \partial_i \partial_j (G(x_1,y) \widetilde{\chi} (y)) (F_{ij} (y) - F_{ij} (x_1) ) dy  \\
&+ \int_{B_2 (0)} \partial_i \partial_j (G(x_2,y) \widetilde{\chi} (y)) (F_{ij} (y) - F_{ij} (x_2) ) dy \\
&= \int_{B_2 (0) \backslash B_\delta (\overline{x})} (\partial_i \partial_j (G(x_1,y) \widetilde{\chi} (y)) - \partial_i \partial_j (G(x_2,y) \widetilde{\chi} (y))) ( F_{ij} (x_2) - F_{ij} (y) ) dy \\
&+ (F_{ij} (x_1) - F_{ij} (x_2)) \int_{B_2 (0) \backslash B_\delta (\overline{x})} \partial_i \partial_{j} (G(x_1,y) \widetilde{\chi} (y)) dy \\
&-\int_{ B_2 (0) \cap B_\delta (\overline{x})} \partial_i \partial_j (G(x_1,y) \widetilde{\chi} (y)) (F_{ij} (y) - F_{ij} (x_1) ) dy  \\
&+ \int_{B_2 (0) \cap B_\delta (\overline{x})} \partial_i \partial_j (G(x_2,y) \widetilde{\chi} (y)) (F_{ij} (y) - F_{ij} (x_2) ) dy.
\end{align*}
\endgroup
Now we estimate the different integrals in the above expression separately (where we rely on the pointwise estimates on the Green's function \eqref{pointwise1}-\eqref{pointwise4}). We will not derive the estimates on the terms with derivatives on $\widetilde{\chi}$, as they are more regular. We find (where without loss of generality we can take $\delta \leq \frac{1}{2}$)
\begingroup
\allowdisplaybreaks
\begin{align*}
&\int_{B_2 (0) \backslash B_\delta (\overline{x})} \widetilde{\chi} (\partial_i \partial_j G(x_1, y) - \partial_i \partial_j G(x_2, y)) ( F_{ij} (x_2) - F_{ij} (y) ) dy \lesssim \int_{B_2 (0) \backslash B_\delta (\overline{x})} \lvert x_1 - x_2 \rvert^{\widetilde{\alpha}} \lvert x_2 - y \rvert^\alpha \\
&\cdot \max \{ \lvert x_1 - y \rvert^{ -3 - \widetilde{\alpha}} , \lvert x_2 - y \rvert^{ -3 - \widetilde{\alpha}} \} dy \lesssim \delta^{\widetilde{\alpha}} \int_{B_2 (0) \backslash B_\delta (\overline{x})} \bigg[ \lvert x_2 - y \rvert^{ -3 - \widetilde{\alpha} + \alpha} \\
&+ \lvert x_2 - y \rvert^\alpha \lvert x_1 - y \rvert^{ -3 - \widetilde{\alpha} } \bigg] dy \lesssim \delta^{\widetilde{\alpha}} \int_{B_2 (0) \backslash B_\delta (\overline{x})} \bigg[ \lvert x_2 - y \rvert^{ -3 - \widetilde{\alpha} + \alpha} + \delta^\alpha \lvert x_1 - y \rvert^{ -3 - \widetilde{\alpha} } + \lvert x_1 - y \rvert^{ -3 - \widetilde{\alpha} + \alpha} \bigg] dy \\
&\lesssim \delta^{\widetilde{\alpha}} \int_{\mathbb{R}^3 \backslash B_{\frac{1}{2} \delta} (x_2)} \lvert x_2 - y \rvert^{ -3 - \widetilde{\alpha} + \alpha} dy + \delta^{\widetilde{\alpha}} \int_{\mathbb{R}^3 \backslash B_{\frac{1}{2} \delta} (x_1)} \bigg[ \delta^\alpha \lvert x_1 - y \rvert^{ -3 - \widetilde{\alpha} } + \lvert x_1 - y \rvert^{ -3 - \widetilde{\alpha} + \alpha} \bigg] dy \lesssim \delta^\alpha, \\
&(F_{ij} (x_1) - F_{ij} (x_2)) \int_{B_2 (0) \backslash B_\delta (\overline{x})} \widetilde{\chi} \partial_i \partial_{j} G(x_1, y) dy  \lesssim (F_{ij} (x_1) - F_{ij} (x_2)) \int_{ \partial (B_2 (0) \backslash B_\delta (\overline{x}))} \widetilde{\chi} \partial_j G(x_1, y) dy \\
&- (F_{ij} (x_1) - F_{ij} (x_2)) \int_{ B_2 (0) \backslash B_\delta (\overline{x})} \partial_i \widetilde{\chi} \partial_j G(x_1, y) dy \lesssim \delta^\alpha \int_{\partial B_\delta (\overline{x})} \frac{4}{\delta^2} dy + \delta^\alpha \int_{\partial B_2 (0)} 4 dy \lesssim \delta^\alpha, \\
&\int_{B_2 (0) \cap B_\delta (\overline{x})} \widetilde{\chi} \partial_i \partial_j G(x_1,y) (F_{ij} (y) - F_{ij} (x_1) ) dy \lesssim \int_{B_2 (0) \cap B_\delta (\overline{x})} \lvert x_1 - y \rvert^{-3 + \alpha} dy \\
&\lesssim \int_{B_2 (0) \cap B_{\frac{3}{2} \delta} (x_1)} \lvert x_1 - y \rvert^{-3 + \alpha} dy \lesssim \delta^\alpha, 
\end{align*}
\endgroup
the other bounds follow in a similar fashion. 
In the above we have used that $\lvert y - x_1 \rvert^{-1}, \lvert y - x_2 \rvert^{-1} \leq \frac{2}{\delta}$ on $B_2 (0) \backslash B_\delta (\overline{x})$ (which can be seen by using the reverse triangle inequality), and also that $\lvert y - x_1 \rvert \leq \frac{3}{2} \delta$ for $y \in B_2 (0) \cap B_\delta (\overline{x})$. Moreover, we chose $\alpha < \widetilde{\alpha} < 1$. Therefore we are able to conclude that
\begin{equation*}
\lvert \widetilde{v}_2 (x_1) - \widetilde{v}_2 (x_2) \rvert \lesssim \delta^\alpha = C \lvert x_1 - x_2 \rvert^\alpha,
\end{equation*}
which concludes the proof of the interior Schauder estimate.
\end{proof}
\begin{remark}
Using a similar approach as in the proof of Proposition \ref{interiorschauderprop} (by replacing $\Omega_0$ with a general bounded domain $\Omega$), one can obtain a global regularity estimate (which is uniform up to the boundary) directly. However, such an estimate would only yield logarithmic $C^{0,\alpha} (\Omega)$ regularity, as the terms of the type $(F_{ij} (x_1) - F_{ij} (x_2)) \int_{\Omega_0 \backslash B_\delta (\overline{x})} \partial_i \partial_{j} G(x_1, y) dy$ will have an additional logarithmic scaling in $\delta$ for general domains $\Omega$. As one of the goals of this paper is to prove the $C^{0,\alpha} (\Omega)$ regularity of the pressure, such a result does not suffice and therefore we prove separate interior and boundary regularity estimates in this appendix.
\end{remark}

Now we prove the boundary Hölder regularity of the solution $v$.
\begin{proposition}
Let $v \in H^1_0 (\Omega)$ be the unique solution to problem \eqref{ellipticschaudereq}, and assume that $F_{ij} \in C^{2,\alpha} (\Omega)$. Then there exists a $\delta > 0$ such that
\begin{equation}
\lVert v \rVert_{C^{0,\alpha} (V_\delta)} \leq \lVert F \rVert_{C^{0,\alpha} (\Omega)} + \lVert v \rVert_{L^\infty (\Omega)}.
\end{equation}
\end{proposition}
\begin{proof}
As before, in order to obtain a covering of the set $V_\delta$ we use the partition of unity $\rho_1, \ldots, \rho_m$ for the sets $V_{\delta, U_1}, \ldots, V_{\delta, U_m}$. Therefore it suffices to establish the Hölder regularity estimate for each of the sets $V_{\delta, U_k}$ (for $k = 1, \ldots, m$), from which the global estimate follows. There exists a mapping $\psi_k$ such that $\psi_k (\partial \Omega \cap U_k) \subset \mathbb{R}^2$ is flat, and $\psi_k (V_{\delta,U_k})$ is the upper half-ball $B_1^+ (0)$. One can check that $\rho_k v$ satisfies the following Dirichlet problem 
\begin{align*}
\begin{cases}
\Delta (\rho_k v) = \Delta \rho_k v + \nabla \rho_k \cdot \nabla v + \rho_k \partial_i \partial_j (F_{ij}) \quad \text{on } V_{\delta, U_k}, \\
\rho_k v = 0 \quad \text{on } \partial V_{\delta, U_k}.
\end{cases}
\end{align*}
Then by using the mapping $\psi_k$ this Dirichlet problem can be mapped onto
\begin{align} \label{boundarydirichletproblem}
\begin{cases}
a^{ij} (x) \partial_i  \partial_j (\rho_k v) + b^i (x) \partial_i (\rho_k v) = \Delta \rho_k v + \nabla \rho_k \cdot \nabla v + \rho_k \partial_i \partial_j (F_{ij}) \quad \text{on } B_1^+ (0), \\
\rho_k v = 0 \quad \text{on } \partial B_1^+ (0),
\end{cases}
\end{align}
where we have used equation \eqref{doubledivergence} to rewrite $(\nabla \otimes \nabla) : F$ in the new coordinates (and absorbed the factors of $b$ into $F$ and $\rho_k$). Now by using a similar splitting of $\rho_k v$ as in the interior case, we only consider the following Dirichlet problem
\begin{align} \label{dirichletproblemball}
\begin{cases}
a^{ij} (x) \partial_i  \partial_j (\rho_k v) + b^i (x) \partial_i (\rho_k v) = \partial_i \partial_j (\rho_k  F_{ij}) \quad \text{on } B_1^+ (0), \\
\rho_k v = 0 \quad \text{on } \partial B_1^+ (0).
\end{cases}
\end{align}
The regularity estimate for the other Dirichlet problems of the decomposition can be achieved in a similar manner as before, again by using Theorem 13.1.1 in \cite{jostPDE} or Theorem 4.15 in \cite{gilbarg}. We now wish to prove the estimate (where we choose the partition of unity such that $\rho_k \equiv 1$ on $B^+_{1/2} (0)$)
\begin{equation*}
\lVert v \rVert_{C^{0,\alpha} (B_{1/2}^+ (0))} = \lVert \rho_k v \rVert_{C^{0,\alpha} (B_{1/2}^+ (0))} \lesssim \lVert F_{ij} \rVert_{C^{0,\alpha} (B_1^+ (0))}.
\end{equation*}
As was mentioned already, using the partition of unity of $V_\delta$ this estimate can then be extended to a global boundary regularity estimate for $V_\delta$. Moreover, we recall that the Hölder norms remain bounded under the mapping $\psi$, see for example equations (6.29) and (6.30) in \cite{gilbarg}. By Theorem 1 in \cite{dolzmann} we know that there exists a Green's function for the Dirichlet problem \eqref{dirichletproblemball} but with the ball $B_1 (0)$ as domain, which satisfies
\begin{equation}
\begin{cases}
a^{ij} (y) \partial_i  \partial_j G(x,y) + b^i (y) \partial_i G(x,y) = - \delta (x-y) \quad \text{for } y \in B_1 (0), \\
G(x, y) = 0 \quad \text{for } y \in \partial B_1 (0),
\end{cases}
\end{equation}
which satisfies the pointwise estimates \eqref{pointwise1}-\eqref{pointwise4}. Now by following the approach in \cite[~Theorem 4.11]{gilbarg}, we extend $\rho_k F_{ij}$ to $B_1 (0)$ by means of an even extension, such that
\begin{equation*}
\reallywidetilde{\rho_k F_{ij}} (z_1, z_2, z_3) = \begin{cases}
\rho_k F_{ij} (z_1, z_2, z_3) \quad \text{if } z \in B_1^+ (0), \\
\rho_k F_{ij} (z_1, z_2, -z_3) \quad \text{if } z \in B_1^- (0).
\end{cases}
\end{equation*}
We also introduce the following notation for $z = (z_1, z_2, z_3) \in \mathbb{R}^3$
\begin{equation*}
\widetilde{z} \coloneqq (z_1, z_2, - z_3).
\end{equation*}
In addition, we denote the boundary of the half-space (i.e. the set $\{ z \in \mathbb{R}^3 : z_3 = 0 \}$) by $E$.
By the uniqueness of the solution to problem \eqref{dirichletproblemball} it follows that we can represent the solution as follows
\begin{equation}
\rho_k v (x) = \int_{B_1^+ (0)} \bigg[ G(x,y) - G(\widetilde{x},y) \bigg] \partial_i \partial_j (\rho_k  F_{ij} (y)) dy.
\end{equation}
It is clear that $\rho_k v \lvert_{E \cap \partial B_1^+ (0)} = 0$ and also that $a^{ij} (x) \partial_i  \partial_j (\rho_k v) + b^i (x) \partial_i (\rho_k v) = \partial_i \partial_j (\rho_k  F_{ij})$ on $B_1^+ (0)$. In addition, one can check that $\rho_k v \lvert_{\partial B_1^+ (0) \backslash E} = 0$ by the properties of the Green's function. We can rewrite the representation formula as follows
\begin{align*}
&\rho_k v (x) = - \int_{B_1^+ (0)} \partial_i \bigg[ G(x,y) - G(\widetilde{x},y) \bigg] \partial_j (\rho_k  F_{ij} (y)) dy \\
&+ \int_{\partial B_1^+ (0)} \bigg[ G(x,y) - G(\widetilde{x},y) \bigg] \partial_j (\rho_k  F_{ij} (y)) dy = \int_{B_1^+ (0)} \partial_i \partial_j \bigg[ G(x,y) - G(\widetilde{x},y) \bigg] (\rho_k  F_{ij} (y)) dy \\
&= \int_{B_1^+ (0)} \partial_i \partial_j \bigg[ G(x,y) - G(\widetilde{x},y) \bigg] \big[ \rho_k  F_{ij} (y) - \rho_k  F_{ij} (x) \big] dy \\
&+ \rho_k  F_{ij} (x) \int_{\partial B_1^+ (0) \backslash E} \partial_i \bigg[ G(x,y) - G(\widetilde{x},y) \bigg] dy,
\end{align*}
where we have used the fact that $\rho_k F_{ij}$ has compact support in $B_1 (0)$. By using the same approach as in the proof of Proposition \ref{interiorschauderprop}, one can write for $x_1, x_2 \in B_{1/2}^+ (0)$ (again defining $\overline{x} = (x_1 + x_2) / 2$ and $\delta = \lvert x_1 - x_2 \rvert$)
\begin{align*}
&\rho_k v (x_1) - \rho_k v (x_2) = \int_{B_1^+ (0)} \partial_i \partial_j \bigg[ G(x_1,y) - G(\widetilde{x}_1,y) \bigg] \big[ \rho_k  F_{ij} (y) - \rho_k  F_{ij} (x_1) \big] dy \\
&- \int_{B_1^+ (0)} \partial_i \partial_j \bigg[ G(x_2,y) - G(\widetilde{x}_2,y) \bigg] \big[ \rho_k  F_{ij} (y) - \rho_k  F_{ij} (x_2) \big] dy \\
&+ \rho_k  F_{ij} (x_1) \int_{\partial B_1^+ (0) \backslash E} \partial_i \bigg[ G(x_1,y) - G(\widetilde{x}_1,y) \bigg] dy - \rho_k  F_{ij} (x_2) \int_{\partial B_1^+ (0) \backslash E} \partial_i \bigg[ G(x_2,y) - G(\widetilde{x}_2,y) \bigg] dy \\
&= \int_{B_1^+ (0) \backslash B_\delta (\overline{x})} \big[\partial_i \partial_j G(x_1,y) - \partial_i \partial_j G(x_2,y) + \partial_i \partial_j G(\widetilde{x}_2,y) - \partial_i \partial_j G(\widetilde{x}_1,y) \big] \big[ \rho_k  F_{ij} (y)  - \rho_k  F_{ij} (x_1) \big] dy \\
&+ \big[ \rho_k  F_{ij} (x_2) - \rho_k  F_{ij} (x_1) \big] \int_{B_1^+ (0) \backslash B_\delta (\overline{x})} \big[ \partial_i \partial_j G(x_2,y) - \partial_i \partial_j G(\widetilde{x}_2,y) \big] dy \\
&+ \int_{B_1^+ (0) \cap B_\delta (\overline{x})} \partial_i \partial_j \bigg[ G(x_1,y) - G(\widetilde{x}_1,y) \bigg] \big [ \rho_k  F_{ij} (y) - \rho_k  F_{ij} (x_1) \big] dy \\
&- \int_{B_1^+ (0) \cap B_\delta (\overline{x})} \partial_i \partial_j \bigg[ G(x_2,y) - G(\widetilde{x}_2,y) \bigg] \big [ \rho_k  F_{ij} (y) - \rho_k  F_{ij} (x_2) \big] dy \\
&+ \rho_k  F_{ij} (x_1) \int_{\partial B_1^+ (0) \backslash E} \partial_i \bigg[ G(x_1,y) - G(\widetilde{x}_1,y) \bigg] dy - \rho_k  F_{ij} (x_2) \int_{\partial B_1^+ (0) \backslash E} \partial_i \bigg[ G(x_2,y) - G(\widetilde{x}_2,y) \bigg] dy.
\end{align*}
Except for the last two terms, the integrals can be bounded by $C \delta^\alpha$ in the same way as in the proof of Proposition \ref{interiorschauderprop} (using the symmetry of $\reallywidetilde{\rho_k F_{ij}}$). Therefore we only estimate the last two integrals. We can write 
\begin{align*}
&\rho_k  F_{ij} (x_1) \int_{\partial B_1^+ (0) \backslash E} \partial_i \bigg[ G(x_1,y) - G(\widetilde{x}_1,y) \bigg] dy - \rho_k  F_{ij} (x_2) \int_{\partial B_1^+ (0) \backslash E} \partial_i \bigg[ G(x_2,y) - G(\widetilde{x}_2,y) \bigg] dy \\
&= \rho_k  F_{ij} (x_1) \int_{\partial B_1^+ (0) \backslash E} \partial_i \bigg[ G(x_1,y) - G(x_2,y) + G(\widetilde{x}_2,y) - G(\widetilde{x}_1,y) \bigg] dy \\
&+ \big[\rho_k  F_{ij} (x_1) - \rho_k  F_{ij} (x_2) \big] \int_{\partial B_1^+ (0) \backslash E} \partial_i \bigg[ G(x_2,y) - G(\widetilde{x}_2,y) \bigg] dy \\
&\lesssim \lVert F \rVert_{L^\infty} \delta^\alpha \int_{\partial B_1^+ (0) \backslash E} \max \{ \lvert x_1 - y \rvert^{-2-\alpha}, \lvert x_2 - y \rvert^{-2-\alpha}, \lvert \widetilde{x}_1 - y \rvert^{-2-\alpha}, \lvert \widetilde{x}_2 - y \rvert^{-2-\alpha} \} dy \\
&+ \delta^\alpha \int_{\partial B_1^+ (0) \backslash E} \bigg[ \frac{1}{\lvert x_2 - y \rvert^2} + \frac{1}{\lvert \widetilde{x}_2 - y \rvert^2} \bigg] dy.
\end{align*}
One can check that these two integrals can be bounded independent of $\delta$, which concludes the proof.
\end{proof}
\begin{remark}
The proof of Theorem \ref{dirichletschauderestimate} can be easily adapted to the case where terms of the form $c(x) v$ appear on the right-hand side of the equation of problem \eqref{ellipticschaudereq}.
\end{remark}

\section{Schauder-type estimate for Dirichlet-Neumann problem} \label{schauderappendix2}
In section \ref{decompositionsection} we considered the regularity of the pressure near the boundary. The localised boundary layer pressure $P^\epsilon_{b,j}$ satisfies a Dirichlet-Neumann problem of the following type
\begingroup
\allowdisplaybreaks
\begin{equation} \label{dirichletneumannproblem}
\begin{cases}
- \Delta v = f \quad \text{in } \Omega, \\
v = 0 \quad \text{in } \Gamma_D, \\
\partial_n v = g \quad \text{in } \Gamma_N. 
\end{cases}
\end{equation}
\endgroup
Here $\Omega$ is still the domain, while $\Gamma_D$ and $\Gamma_N$ are open sets in $\partial \Omega$ such that $\Gamma_N = \partial \Omega \backslash \overline{\Gamma_D}$. We will first prove the existence and uniqueness of solutions to this problem in the space $H^1_{0,\Gamma_D} (\Omega)$, which consists of the $H^1 (\Omega)$ functions with zero trace on $\Gamma_D$. The weak formulation of the problem is (as can be found in \cite[~p. 516]{salsa})
\begin{equation}
\int_\Omega \nabla v \cdot \nabla \psi dx = \int_\Omega f \psi dx + \int_{\Gamma_N} g \psi d \sigma,
\end{equation}
where $\psi$ is an arbitrary test function in $H^1_{0,\Gamma_D} (\Omega)$. We will first prove existence and uniqueness, based on the method outlined in \cite[~Chapter 8]{salsa}.
\begin{theorem} \label{existencedirichletneumann}
The Dirichlet-Neumann problem \eqref{dirichletneumannproblem} has a unique solution in $H^1_{0,\Gamma_D} (\Omega)$.
\end{theorem}
\begin{proof}
By Theorem 7.91 (see also page 516) in \cite{salsa}, we know that the Poincaré inequality holds for functions in $H^1_{0,\Gamma_D} (\Omega)$. This means in particular that we may take the following norm for the space $H^1_{0,\Gamma_D} (\Omega)$
\begin{equation}
\lVert u \rVert_{H^1_{0,\Gamma_D}} \coloneqq \lVert \nabla u \rVert_{L^2}.
\end{equation}
Now we prove the existence and uniqueness for the Dirichlet-Neumann problem by using the Lax-Milgram theorem. We introduce the bilinear form $B : H^1_{0,\Gamma_D} (\Omega) \times H^1_{0,\Gamma_D} (\Omega) \rightarrow \mathbb{R}$ which is given by
\begin{equation*}
B[v_1, v_2] \coloneqq \int_{\Omega} \nabla v_1 \cdot \nabla v_2 dx. 
\end{equation*}
It is easy to verify that $B$ is coercive and continuous in the space $H^1_{0,\Gamma_D}$. Moreover, since $f \in L^2 (\Omega)$ and $g \in L^2 (\partial \Omega)$ it is easy to check that the map $\psi \mapsto \int_\Omega f \psi dx + \int_{\Gamma_N} g \psi d \sigma$ is in $H^{-1}_{0,\Gamma_D} (\Omega)$. Therefore by the Lax-Milgram theorem there exists a unique solution in $H^1_{0,\Gamma_D} (\Omega)$ for the Dirichlet-Neumann problem given in equation \eqref{dirichletneumannproblem}.
\end{proof}
Now for the sake of completeness we would like to establish a Schauder estimate for the Dirichlet-Neumann problem, as it does not seem to be stated in standard references such as \cite{gilbarg,giaquinta,krylov}. In order to do so, we will rely on the approach given in \cite{lieberman}. We will prove the following result. Note that we establish the regularity result for the specific Dirichlet-Neumann problem from section \ref{decompositionsection}, but the method also allows one to establish a more general result which we will omit here. 
\begin{theorem} \label{schauderdirichletneumann}
Let $\partial \Omega \cap U_j$ be a patch of $\partial \Omega$ given by the localisation from section \ref{parametrisationsection}. Let $P^\epsilon_{b,j}$ be a solution to the following problem from section \ref{decompositionsection}
\begingroup
\allowdisplaybreaks
\begin{equation} \label{boundarylayerproblem}
\begin{cases}
- \Delta P^\epsilon_{b,j} = - \Delta (\phi_b \rho_j) P^\epsilon - 2  \nabla (\phi_b \rho_j) \cdot \nabla P^\epsilon + \phi_b \rho_j \bigg((\nabla \otimes \nabla) : (u^\epsilon \otimes u^\epsilon) - \Delta \big( (u^\epsilon \cdot n)^2 \big)  \bigg) \; \text{in } V_{\delta,U_j}, \\
P^\epsilon_{b,j} = 0 \quad \text{on } \partial V_{\delta,U_j} \backslash (U_j \cap \partial \Omega) , \quad \partial_n P^\epsilon_{b,j} = \rho_j \big( u^\epsilon \otimes u^\epsilon : \nabla n \big) \quad \text{on } \partial \Omega \cap U_j.
\end{cases}
\end{equation}
\endgroup
Then $P^\epsilon_{b,j}$ satisfies the following Schauder-type estimate
\begin{equation} \label{schauderestimate}
\lVert P^\epsilon_{b,j} \rVert_{C^{0,\alpha} (V_{\delta,U_j})} \leq C \lVert u^\epsilon \otimes u^\epsilon \rVert_{C^{0,\alpha} (V_{\delta,U_j})} + D \lVert P^\epsilon_{b,j} \rVert_{L^\infty (V_{\delta,U_j})}.
\end{equation}
\end{theorem}
\begin{proof}
By Theorem \ref{existencedirichletneumann} we know that problem \eqref{boundarylayerproblem} has a unique solution in $H^1_{0,\partial V_{\delta, U_j} \backslash (U_j \cap \partial \Omega)} (\Omega)$ (for given $P^\epsilon$ and $u^\epsilon$). 

Now because $\partial \Omega \in C^2$ we can map $U_j \cap \Omega$ by a $C^2$ mapping $\psi$ such that $\psi (\partial \Omega \cap U_j)$ is flat (or alternatively, it is characterised by the third coordinate being 0) and $\psi (V_{\delta,U_j})$ is the upper half of an open ball. If this is not possible, we can restrict to subsets of $U_j$ such that the part of the boundary which intersects with $\partial \Omega$ can be mapped to a flat set, by compactness there are finitely many such sets. We can therefore transform the problem to (see \cite[~Section 6.2]{krylov} for concrete computations)
\begingroup
\allowdisplaybreaks
\begin{equation} \label{transformedproblem}
\begin{cases}
- a^{ij} (x) \partial_i  \partial_j P^\epsilon_{b,j} + b_i (x) \partial_i P^\epsilon_{b,j} (x) = - \Delta (\phi_b \rho_j) P^\epsilon - 2  \nabla (\phi_b \rho_j) \cdot \nabla P^\epsilon + \phi_b \rho_j \bigg((\nabla \otimes \nabla) : (u^\epsilon \otimes u^\epsilon) \\
- \Delta \big( (u^\epsilon \cdot n)^2 \big)  \bigg) \eqqcolon F' \quad \text{in } \psi(V_{\delta,U_j}) , \\
P^\epsilon_{b,j} = 0 \quad \text{in } \psi \big(\partial V_{\delta, U_j} \backslash (\partial \Omega \cap U_j) \big), \quad \partial_n P^\epsilon_{b,j} = \rho_j \big( u^\epsilon \otimes u^\epsilon : \nabla n \big) \quad \text{if } x_n = 0,
\end{cases}
\end{equation}
\endgroup
for some coefficients $a^{ij}$ and $b^i$ which depend on the coordinate transformation $\psi$. Note that the uniform ellipticity is preserved by Lemma 6.2.1 in \cite{krylov} (this is why the $C^2$ regularity assumption on the boundary is crucial). 
Now we want to homogenise the Neumann boundary condition. One can find a function $G \in C^2 (\psi (V_{\delta, U_j}))$ such that
\begin{equation*}
\frac{\partial G}{\partial x_n} = \rho_j \big( u^\epsilon \otimes u^\epsilon : \nabla n \big) \quad \text{if } x_n = 0.
\end{equation*}
Then by taking $(P^\epsilon_{b,j})' = P^\epsilon_{b,j} - G$  and introducing the following notation (where $F'$ is the forcing on the right-hand side of the equation in problem \ref{transformedproblem})
\begin{align*}
F_1 &= F' + a^{ij} \partial_i \partial_j G - b_i \partial_i G, \\
F_2 &= - G,
\end{align*}
we end up with the problem (where for notational convenience we will refer to $(P^\epsilon_{b,j})'$ as $P^\epsilon_{b,j}$)
\begin{equation}
\begin{cases}
- a^{ij} (x) \partial_i \partial_j P^\epsilon_{b,j} + b_i (x) \partial_i P^\epsilon_{b,j} (x) = F_1 \quad \text{in } \psi (V_{\delta,U_j}) , \\
P^\epsilon_{b,j} = F_2 \quad \text{in } \psi \big(\partial V_{\delta, U_j} \backslash (\partial \Omega \cap U_j) \big), \\
\partial_n P^\epsilon_{b,j} = 0 \quad \text{if } x_n = 0.
\end{cases}
\end{equation}
Then we take an even extension of $P^\epsilon_{b,j}$, $a^{ij}$, $b^i$, $F$ and $\psi$ with respect to $x_n$. We write $\widehat{\psi( V_{\delta, U_j})}$ for the even extension of $\psi(V_{\delta, U_j})$ through $x_n = 0$. This then leads to the problem
\begin{equation} \label{dirichletproblem}
\begin{cases}
- a^{ij} (x) \partial_i \partial_j  P^\epsilon_{b,j} + b_i (x) \partial_i P^\epsilon_{b,j} (x) = F_1 \quad \text{in } \widehat{\psi(V_{\delta, U_j})}, \\
P^\epsilon_{b,j} = F_2 \quad \text{on } \partial (\widehat{\psi(V_{\delta, U_j})} ).
\end{cases}
\end{equation}
Now we will prove the Schauder estimate by the continuity method, see \cite[~Section 5.5.1]{giaquinta} (and also \cite[~Section 13.3]{jostPDE} and \cite[~Section 5.2]{gilbarg}) for example. Note that one can also establish the result using the Green's function, similarly to what was done in Appendix \ref{schauderappendix1}. We first define the following operator
\begin{equation*}
L \coloneqq - a^{ij} (x) \partial_i \partial_j + b_i (x) \partial_i,
\end{equation*}
which is the differential operator of the stated Dirichlet problem \eqref{dirichletproblem}. Then we introduce the continuous family of operators
\begin{equation*}
L_t \coloneqq (1 - t) \Delta + t L.
\end{equation*}
For all $t \in [0,1]$ the problem associated with $L_t$ and the data $F_1$ and $F_2$ has a unique solution (using the Lax-Milgram theorem). Now we need to prove that it satisfies the Schauder estimate from Theorem \ref{dirichletschauderestimate}. We define the subset $\Sigma \subset [0,1]$ such that for all $t \in \Sigma$ the estimate is satisfied. Now we will prove that $\Sigma$ is both open and closed and hence $\Sigma = [0,1]$, as $0 \in \Sigma$ due to Theorem \ref{dirichletschauderestimate}.

We will first show that $\Sigma$ is closed. Suppose we take a sequence $t_k \rightarrow t$, such that for all $t_k$ estimate \eqref{schauderestimate} holds. Then by compactness we know that there exists a subsequence of solutions to the problems with operators $L_{t_k}$ (which we label with $v_{t_k}$) converging uniformly to $v_{t}$ (i.e. strong convergence in $C^0 (\Omega)$). This means that the Schauder estimate also holds for $v_t$. 

Now we have to show that $\Sigma$ is open. We take an arbitrary $t_0 \in \Sigma$. We denote by $u_w = T_t w$ the solution to the problem
\begin{equation*}
L_{t_0} u_w = (L_{t_0} - L_t) w + F_1 \text{ in } \widehat{\psi(V_{\delta, U_j})}, \quad u_w = F_2 \text{ on } \partial (\widehat{\psi(V_{\delta, U_j})}).
\end{equation*}
Then we observe that $L_{t_0} - L_t = (t - t_0) \Delta + (t_0 - t) L$. By the assumed Schauder estimate we get that for some $w_1, w_2 \in C^{0,\alpha} (\widehat{\psi(V_{\delta, U_j})})$ (by an adaption of the proof of Theorem \ref{dirichletschauderestimate})
\begin{equation}
\lVert u_{w_1} - u_{w_2} \rVert_{C^{0,\alpha}} = \lVert T_t w_1 - T_t w_2 \rVert_{C^{0,\alpha}} \leq c \lvert t - t_0 \rvert \lVert w_1 - w_2 \rVert_{C^{0,\alpha}}.
\end{equation}
Then if $\lvert t - t_0 \rvert$ is sufficiently small, the mapping $T_t$ is a contraction and hence by the Banach fixed point theorem there exists a fixed point for this equation, which solves the Dirichlet problem with the operator $L_t$. Therefore the Schauder estimate holds for a small neighbourhood around $t$ and hence $\Sigma$ is open and therefore $\Sigma = [0,1]$ (as $\Sigma$ cannot be the empty set). In particular, estimate \eqref{schauderestimate} is true for the case $t=1$, which is what we wanted to show. 
\end{proof}
\section{The normal derivative of the Weierstrass flow away from the boundary} \label{normalderivativeappendix}
We recall that in Example \ref{weierstrassexample} we constructed a flow (given in equations \eqref{weierstrass1} and \eqref{weierstrass2}) such that $\partial_n (u \cdot n)^2 \lvert_{\partial \Omega} \notin \mathcal{D}' (\partial \Omega)$. In this appendix we will extend the results of this example, in particular we will show the following theorem.
\begin{theorem}
Let $u \in C^{0,\alpha} (\Omega)$ be the Weierstrass flow introduced in equations \eqref{weierstrass1} and \eqref{weierstrass2} and consider test functions $\theta \in \mathcal{D} (\mathbb{T})$ such that
\begin{equation*}
\int_{\mathbb{T}} \theta(x) dx \not=0.
\end{equation*}
Then for $\alpha \in \left(0, \frac{1}{2} \right)$ we have that $\partial_n (u \cdot n)^2 (\cdot, y) \notin \mathcal{D}' (\mathbb{T})$ for any $y = \frac{j}{2^m}$, where $j = 1, 2, \ldots , 2^m - 1$ and $m \geq 1$. 
\end{theorem}
\begin{proof}
By reexamining the proofs in section \ref{examplesection}, one can check that for any $y > 0$ $U_{NR} (\cdot ; \theta)$ and $U_{RNR} (\cdot ; \theta)$ are $C^1 $ functions. 

We recall that $U_{RR}$ was defined by
\begin{equation}
U_{RR}(y;\theta)= \frac{1}{2} \int_{\mathbb{T}} \theta (x) dx \sum_{k=0 }^\infty 2^{-2\alpha k} (\sin(2^{k } \pi y)^2.
\end{equation}
Now we will consider the following difference quotient for some $y_1 = \frac{j}{2^m}$ (for some $m \geq 1$ and $j= 1, 2, \ldots, 2^m$)
\begin{equation*}
\frac{U_{RR} (y_1 + h; \theta) - U_{RR} (y_1 ; \theta)}{\lvert h \rvert} = \frac{1}{2 \lvert h \rvert} \int_{\mathbb{T}} \theta (x) dx \sum_{k=0 }^\infty 2^{-2\alpha k} \bigg[ (\sin(2^{k } \pi (y_1+h))^2 - (\sin(2^{k } \pi y_1))^2 \bigg].
\end{equation*}
We will first rewrite the difference quotient as
\begin{align*}
&\frac{U_{RR} (y_1 + h; \theta) - U_{RR} (y_1 ; \theta)}{\lvert h \rvert} = \frac{1}{4 \lvert h \rvert} \int_{\mathbb{T}} \theta (x) dx \sum_{k=0 }^\infty 2^{-2\alpha k} \bigg[ \cos (2^{k +1} \pi y_1) - \cos(2^{k +1} \pi (y_1 + h)) \bigg] \\
&= \frac{1}{2 \lvert h \rvert} \int_{\mathbb{T}} \theta (x) dx \sum_{k=0 }^\infty 2^{-2\alpha k} \sin ( 2^k \pi (2 y_1 + h) ) \sin (2^k \pi h).
\end{align*}
Now once again we select the sequence $h_n = 2^{-n}$, as we did in Example \ref{weierstrassexample}. Once again, we notice that $\sin (2^{k-n} \pi ) = 0$ for $k \geq n$, so we end up with the partial sum
\begin{align*}
&\frac{U_{RR} (y_1 + h_n; \theta) - U_{RR} (y_1 ; \theta)}{\lvert h_n \rvert} = 2^{n-1} \int_{\mathbb{T}} \theta (x) dx \sum_{k=0 }^{n-1} 2^{-2\alpha k} \sin ( 2^k \pi (2 y_1 + 2^{-n}) ) \sin (2^{k-n} \pi ) \\
&= 2^{n-1} \int_{\mathbb{T}} \theta (x) dx \sum_{k=0 }^{n-1} 2^{-2\alpha k} \bigg[ \sin ( 2^{k+1} \pi y_1) \cos (2^{k-n} \pi ) + \cos ( 2^{k+1} \pi y_1) \sin (2^{k-n} \pi ) \bigg] \sin (2^{k-n} \pi ).
\end{align*}
Now we substitute our choice for $y_1$. Since we are interested in the behaviour of $U_{RR}$ as $n \rightarrow \infty$ we may assume without loss of generality that $m \leq n$. Then we obtain that
\begingroup
\allowdisplaybreaks
\begin{align}
&\frac{U_{RR} (y_1 + h_n; \theta) - U_{RR} (y_1 ; \theta)}{\lvert h_n \rvert} = 2^{n-1} \int_{\mathbb{T}} \theta (x) dx \sum_{k=0 }^{n-1} 2^{-2\alpha k} \bigg[ \sin ( 2^{k+1-m} \pi j) \cos (2^{k-n} \pi ) \nonumber \\
&+ \cos ( 2^{k+1 - m} \pi j) \sin (2^{k-n} \pi ) \bigg] \sin (2^{k-n} \pi ) \nonumber \\
&= 2^{n-1} \int_{\mathbb{T}} \theta (x) dx \sum_{k=0 }^{m-2} \bigg[ 2^{-2\alpha k} \sin ( 2^{k+1-m} \pi j) \cos (2^{k-n} \pi ) \sin (2^{k-n} \pi ) \bigg] \label{sum1} \\
&+ 2^{n-1} \int_{\mathbb{T}} \theta (x) dx \sum_{k=0 }^{m-1} \bigg[ 2^{-2\alpha k} \cos ( 2^{k+1-m} \pi j) (\sin (2^{k-n} \pi ))^2 \bigg] \label{sum2} \\
&+ 2^{n-1} \int_{\mathbb{T}} \theta (x) dx \sum_{k=m  }^{n-1} \bigg[ 2^{-2\alpha k} (\sin (2^{k-n} \pi ))^2 \bigg]. \label{sum3}
\end{align}
\endgroup
We first investigate the third sum in line \eqref{sum3}, we estimate (using that for $x \in \big[0, \frac{\pi}{2} \big]$ we have that $\frac{2}{\pi} x \leq \sin (x) \leq x$)
\begingroup
\allowdisplaybreaks
\begin{align*}
&2^{n-1} \int_{\mathbb{T}} \theta (x) dx \sum_{k=m  }^{n-1} 2^{-2\alpha k} (\sin (2^{k-n} \pi ))^2 \geq 2^{n-1} \int_{\mathbb{T}} \theta (x) dx \sum_{k=m  }^{n-1} 2^{-2\alpha k} \cdot 2^{2k-2n+2} \\
&= 2^{-n+1} \int_{\mathbb{T}} \theta (x) dx \sum_{k=m}^{n-1} 2^{2(1-\alpha) k} = 2^{-n+1} \int_{\mathbb{T}} \theta (x) dx \bigg[ \frac{ 2^{2(1-\alpha)n} - 1}{2^{2(1-\alpha)} - 1} - \frac{2^{2(1-\alpha)m} - 1}{2^{2(1-\alpha)} - 1} \bigg] \\
&= \int_{\mathbb{T}} \theta(x) dx \cdot  \frac{ 2^{(1 - 2\alpha)n + 1} }{2^{2(1-\alpha)} - 1} + 2^{-n+1} \int_{\mathbb{T}} \theta (x) dx \bigg[ -\frac{  1}{2^{2(1-\alpha)} - 1} - \frac{2^{2(1-\alpha)m} - 1}{2^{2(1-\alpha)} - 1} \bigg].
\end{align*}
\endgroup
Now we need to show that the other two sums remain bounded, we find
\begingroup
\allowdisplaybreaks
\begin{align*}
&2^{n-1} \int_{\mathbb{T}} \theta (x) dx \sum_{k=0 }^{m-1} \bigg[ 2^{-2\alpha k} \cos ( 2^{k+1-m} \pi j) (\sin (2^{k-n} \pi ))^2 \bigg] \\
&\geq 2^{n-1} \int_{\mathbb{T}} \theta (x) dx \sum_{k=0 }^{m-1} \bigg[ 2^{-2\alpha k} \min \{ \cos ( 2^{k+1-m} \pi j) 2^{2k-2n+2}, 2^{-2\alpha k} \min \{ \cos ( 2^{k+1-m} \pi j) 2^{2k-2n} \pi^2  \} \bigg] \\
&= 2^{-n-1} \int_{\mathbb{T}} \theta (x) dx \sum_{k=0 }^{m-1} \bigg[ 2^{-2\alpha k} \min \{ \cos ( 2^{k+1-m} \pi j) 2^{2k+2}, 2^{-2\alpha k} \min \{ \cos ( 2^{k+1-m} \pi j) 2^{2k} \pi^2  \} \bigg], \\
&2^{n-1} \int_{\mathbb{T}} \theta (x) dx \sum_{k=0 }^{m-2} \bigg[ 2^{-2\alpha k} \sin ( 2^{k+1-m} \pi j) \cos (2^{k-n} \pi ) \sin (2^{k-n} \pi ) \bigg] \\
&\geq 2^{n-1} \int_{\mathbb{T}} \theta (x) dx \sum_{k=0 }^{m-2} \bigg[ 2^{-2\alpha k} \min \{ \sin ( 2^{k+1-m} \pi j) \cos (2^{k-n} \pi ) 2^{k-n+1}, \sin ( 2^{k+1-m} \pi j) \cos (2^{k-n} \pi )2^{k-n} \pi \} \bigg] \\
&= 2^{-1} \int_{\mathbb{T}} \theta (x) dx \sum_{k=0 }^{m-2} \bigg[ 2^{-2\alpha k} \min \{ \sin ( 2^{k+1-m} \pi j) \cos (2^{k-n} \pi ) 2^{k+1}, \sin ( 2^{k+1-m} \pi j) \cos (2^{k-n} \pi )2^{k} \pi \} \bigg].
\end{align*}
\endgroup
Therefore both sums can be either bounded from below independent of $n$ or they go to zero as $n \rightarrow \infty$. Because the sum in equation \eqref{sum3} is going to infinity, we conclude that if $\alpha < \frac{1}{2}$ we have that
\begin{equation}
\liminf_{n \rightarrow \infty} \bigg\lvert \frac{U_{RR} (y_1 + h_n; \theta) - U_{RR} (y_1 ; \theta)}{ h_n } \bigg\rvert = \infty,
\end{equation}
for points $y_1$ of the form $y_1 = \frac{j}{2^m}$ for $j=1, \ldots, 2^{m-1}$ and $m \geq 1$. We conclude as a result that $\partial_y u_2^2 (\cdot, y)$ cannot be defined as a distribution (i.e. as an element of $\mathcal{D}' (\mathbb{T})$) for a dense set of points $y \in [0,1]$.

In the case $\alpha = \frac{1}{2}$ we have that
\begin{align*}
&\liminf_{n \rightarrow \infty} \bigg\lvert \frac{U_{RR} (y_1 + h_n; \theta) - U_{RR} (y_1 ; \theta)}{ h_n } \bigg\rvert \geq \bigg\lvert \int_{\mathbb{T}} \theta (x) dx \bigg\rvert  \\
&\cdot \bigg\lvert 2 + \sum_{k=0 }^{m-2} \bigg[ 2^{-2\alpha k - 1} \min \big\{ \sin ( 2^{k+1-m} \pi j) \cos (2^{k} \pi ) 2^{k+1}, \sin ( 2^{k+1-m} \pi j) \cos (2^{k} \pi )2^{k} \pi \big\} \bigg] \bigg\rvert,
\end{align*}
it therefore depends on the values of $j$ and $m$ whether this lower bound is nonzero.
\end{proof}
\end{appendices}

\bibliographystyle{abbrv}
{\footnotesize \bibliography{pressure_boundary_regularity}}

\end{document}